%% file: BMAB_Revision1.tex
\documentclass[letterpaper,11pt,twoside,keywordsasfootnote,addressatend,noinfoline]{article}

\usepackage{fullpage}
\usepackage[english]{babel}
\usepackage{amssymb}
\usepackage{amsmath}
\usepackage{theorem}
\usepackage{epsfig}
\usepackage{subfigure}
\usepackage{imsart}

\theorembodyfont{\slshape}
\newtheorem{theor}{Theorem}

\newtheorem{lemma}[theor]{Lemma}
\newtheorem{corol}[theor]{Corollary}

\newenvironment{proof}{\noindent{\scshape Proof.}}{\hspace{2mm} $\square$}

\newcommand{\Z}{\mathbb{Z}}
\newcommand{\R}{\mathbb{R}}

\newcommand{\ep}{\epsilon}

\DeclareMathOperator{\card}{card}
\DeclareMathOperator{\sign}{sign}
\DeclareMathOperator{\Span}{Span}

\begin{document}

\begin{frontmatter}

\title     {The role of space in the exploitation of resources}
\runtitle  {The role of space in the exploitation of resources}
\author    {Y. Kang and N. Lanchier\thanks{Research supported in part by NSF Grant DMS-10-05282.}}
\runauthor {Y. Kang and N. Lanchier}
\address   {Applied Sciences and Mathematics, \\ Arizona State University, \\ Mesa, AZ 85212, USA.}
\address   {School of Mathematical and Statistical Sciences, \\ Arizona State University, \\ Tempe, AZ 85287, USA.}

\begin{abstract} \ \
 In order to understand the role of space in ecological communities where each species produces a certain type of
 resource and has varying abilities to exploit the resources produced by its own species and by the other species,
 we carry out a comparative study of an interacting particle system and its mean-field approximation.
 For a wide range of parameter values, we show both analytically and numerically that the spatial model results in
 predictions that significantly differ from its nonspatial counterpart, indicating that the use of the mean-field
 approach to describe the evolution of communities in which individuals only interact locally is invalid.
 In two-species communities, the disagreements between the models appear either when both species compete by producing
 resources that are more beneficial for their own species or when both species cooperate by producing resources that
 are more beneficial for the other species.
 In particular, while both species coexist if and only if they cooperate in the mean-field approximation, the inclusion
 of space in the form of local interactions may prevent coexistence even in cooperative communities.
 Introducing additional species, cooperation is no longer the only mechanism that promotes coexistence.
 We prove that, in three-species communities, coexistence results either from a global cooperative behavior, or from
 rock-paper-scissors type interactions, or from a mixture of these dynamics, which excludes in particular all cases
 in which two species compete.
 Finally, and more importantly, we show numerically that the inclusion of space has antagonistic effects on coexistence
 depending on the mechanism involved, preventing coexistence in the presence of cooperation but promoting coexistence
 in the presence of rock-paper-scissors interactions.
 Although these results are partly proved analytically for both models, we also provide somewhat more explicit heuristic
 arguments to explain the reason why the models result in different predictions.
\end{abstract}



\end{frontmatter}

%%%%%%%%%%%%%%%%%%%%%%%%%%%%%%%%%%%%%%%%%%%%%%%%%%%%%%%%%%%%%%%%%%%%%%%%%%%%%%%%%%%%%%%%%%%%%%%%%%%%%%%%%%%%%%%%%%%%%%%%%%%%%%%%%%%%%%%%%%

\section{Introduction}
\label{sec:intro}

\indent The main objective of this article is to understand the role of space, taking the form of local but also
 stochastic interactions, in the long-term behavior of ecological communities.
 This is carried out through the comparison of analytical and numerical results for an interacting particle system and its
 nonspatial mean-field approximation.
 Both the spatial and nonspatial models mimic the dynamics of communities in which each species produces a certain type of
 resource and has varying abilities to exploit the resources produced by its own species and by the other species.
 As explained later in this paper after the rigorous mathematical description of the interacting particle system and
 its mean-field approximation, this framework can be used to model a wide variety of biological interactions such as
 competition, mutualism, allelopathy, or predation.
 In the interacting particle system, individuals are located on an infinite grid and only have access to the resources produced
 by their nearest neighbors.
 In contrast, the mean-field approximation assumes that all individuals interact globally therefore the dynamics only depend
 on the overall densities of species, which results in a nonspatial deterministic model that consists of a system of coupled
 differential equations.
 The reason for this comparative study is mainly motivated by the following key question raised in the seminal article of
 Durrett and Levin \cite{durrett_levin_1994b} about the importance of space:
 which details should be included in a mathematical model and which ones can be ignored?
 The simplest approach in modeling the evolution of inherently spatial interacting populations is to assume that the system is
 homogeneously mixing and to use a system of ordinary differential equations, which is a common approach in the life science
 literature.
 At the other extreme, one can use the framework of interacting particle systems in which individuals are discrete and space
 is treated explicitly.
 For a wide range of parameter values, we show that both approaches result in different predictions as for the sets of
 species that survive or coexist in the long run.
 This indicates that, at least for the type of interactions considered in this article, the use of the mean-field approach to
 model the evolution of communities in which individuals are static and interact only locally is invalid.
 In other words, local interactions cannot be ignored and should be included in the model.

\indent Using the terminology of game theory, a species that has a higher ability, respectively a lower ability, than
 all other species to exploit the resources it produces can be seen as a defector, respectively a cooperator.
 In the presence of only two species, the spatial and nonspatial models strongly disagree when both species are defectors
 or both species are cooperators.
 In particular, while mutual cooperation always leads to coexistence in the mean-field model, the inclusion of local
 interactions translates into a reduction of the coexistence region.
 Past research has revealed that the inclusion of space can indeed prevent coexistence.
 This has been proved analytically for a spatially explicit version of the Lotka-Volterra model introduced
 in \cite{neuhauser_pacala_1999} and the non Mendelian diploid model introduced in \cite{lanchier_neuhauser_2009}.
 These two references and the companion paper \cite{lanchier_2010}, which provides rigorous proofs of the analytical results
 stated in this article for the interacting particle system, focus on spatial and nonspatial models with only two types.
 Our approach is somewhat different.
 First, one of our main objectives is to give more explicit intuitive explanations of the reason why the inclusion
 of local interactions can affect drastically the long-term behavior of ecological communities.
 In particular, heuristic arguments are provided to this extent.
 Second, the emphasis here is on the three-type models.
 Not only the analysis of the models in the presence of three species is somewhat more challenging but also, and more importantly,
 both the spatial and the nonspatial three-type models exhibit new interesting behaviors that cannot be observed when only
 two species are involved.
 For instance, whereas cooperation is the only mechanism that promotes coexistence in the presence of two species, we prove that
 there are other such mechanisms when additional species are introduced.
 Interestingly, our analysis also indicates that the inclusion of local interactions has antagonistic effects on coexistence
 depending on the mechanism involved, promoting coexistence in some situations but preventing coexistence in other situations.
 In particular, we believe that the three-type spatial and nonspatial models can capture most of the important aspects of
 communities involving a large number of species, whereas the two-type models do not.

%%%%%%%%%%%%%%%%%%%%%%%%%%%%%%%%%%%%%%%%%%%%%%%%%%%%%%%%%%%%%%%%%%%%%%%%%%%%%%%%%%%%%%%%%%%%%%%%%%%%%%%%%%%%%%%%%%%%%%%%%%%%%%%%%%%%%%%%%%

\section{Models description}
\label{sec:models}

\indent In order to understand the role of local interactions (and stochasticity) in the dynamics of ecological communities
 including varying abilities to exploit resources, we introduce and investigate a stochastic spatial model and its deterministic
 nonspatial analog, and then confront the predictions based on both models.
 The first model is an example of interacting particle system while the second one is its mean-field approximation which consists
 of a system of differential equations.
 For a description of the theoretical framework of interacting particle systems along with biological motivations, we refer
 the reader to Durrett and Levin \cite{durrett_levin_1994a} and Neuhauser \cite{neuhauser_2001}.
 For the sake of concreteness, we think of both models as describing the dynamics of ecological communities with $n$ species that
 we label from species 1 to species $n$, but we point out that these models can also be seen as more general models in
 game theory with potential applications in population genetics, economics, or political sciences.
 In the ecological context, the individuals of species $j$ has to be thought of as producing a resource that we simply call resource $j$.
 The interactions, either local in the spatial model or global in the nonspatial one, are dictated by the $n \times n$ matrix
 $$ M \ = \ \left(\hspace{-3pt} \begin{array}{ccccc}
      a_{1, 1} & a_{1, 2} & \cdots & a_{1, n} \\
      a_{2, 1} & a_{2, 2} & \cdots & a_{2, n} \\
       \vdots  &  \vdots  & \ddots &  \vdots  \\
      a_{n, 1} & a_{n, 2} & \cdots & a_{n, n}
    \end{array} \hspace{-3pt} \right) $$
 where coefficient $a_{i,j}$ represents the ability of an individual of species $i$ to exploit resource $j$.
 It is assumed that matrix $M$ has only nonnegative entries and, to avoid degenerated cases later, that each column of the
 matrix has at least one positive entry.
 We note that each row can be seen as the abilities of a given species to exploit each of the $n$ resources, and each column
 as the abilities of each of the $n$ species to exploit a given resource.
 In particular, the assumption that each column of the matrix has at least one positive entry means from an ecological
 point of view that each of the resources can be exploited by at least one species.
 Since different species have \emph{a priori} different abilities to exploit each of the resources, their fitness strongly
 depends upon the configuration of their environment as well.
 Therefore, there is a constant feedback between the set of fitnesses and the configuration of the environment, which creates
 nontrivial dynamics.

% % % % % % % % % % % % % % % % % % % % % % % % % % % % % % % % % % % % % % % % % % % % % % % % % % % % % % % % % % % % % % % % % % % % % 

\subsection*{The spatial model.}
 Following the traditional framework of interacting particle systems, we assume that space is discrete and time is continuous.
 More precisely, the spatial structure is represented by the $d$-dimensional integer lattice, $\Z^d$, while the temporal structure is
 represented by the variable $t$ which is any nonnegative real number.
 Each lattice site has to be thought of as a spatial location which, at any time, is occupied by an individual of one of the $n$ species.
 This corresponds to a high-density limit in which, at death, an individual is instantaneously replaced by the offspring
 of a nearby individual.
 Hence, the spatial configuration of the system at time $t$ is represented by a function $\eta_t$ that maps the integer lattice
 into the set of species, with $\eta_t (x)$ indicating the species of the individual at site $x$ at time $t$.

\indent To describe the dynamics of the spatial model but also to derive its mean-field approximation, it is convenient to think of the
 spatial structure as a graph.
 In the spatial model, each lattice site is connected by an edge to each of its $2d$ nearest neighbors, therefore, writing $x \sim y$
 to indicate that sites $x$ and $y$ are nearest neighbors, the integer
 $$ N_i (x) \ = \ \card \,\{y \in \Z^d : y \sim x \ \hbox{and} \ \eta_t (y) = i \} $$
 denotes the number of type $i$ neighbors of site $x$.
 The species at each site is updated at the arrival times of a Poisson process with intensity 1, that is the times between consecutive
 updates at a given site are independent and exponentially distributed random variables with parameter 1.
 The dynamics are dictated by stochastic and local interactions, meaning that the new type at each update is chosen randomly from the
 neighborhood, in order to model the inclusion of an explicit space.
 More precisely, given that the species at site $x$ is of type $j$ at the time of an update, the new species at this site is chosen to
 be species $i$ with probability
\begin{equation}
\label{eq:proba}
 \frac{a_{i, j} \,N_i (x)}{a_{1, j} \,N_1 (x) + a_{2, j} \,N_2 (x) + \cdots + a_{n, j} \,N_n (x)}
\end{equation}
 and we assume, in the event that the denominator is equal to zero, that the update is canceled so the species at site $x$ remains
 unchanged.
 This probability can be interpreted as follows.
 Since the coefficient $a_{i,j}$ represents the ability of an individual of species $i$ to exploit resource $j$, the numerator
 is simply the overall ability of the type $i$ neighbors to exploit the resource at site $x$ at the time of the update.
 Similarly, the denominator represents the overall ability of all the neighbors to exploit the resource at site $x$.
 Therefore, the probability in \eqref{eq:proba} is the relative ability of the type $i$ neighbors to exploit the resource at site $x$,
 which we also naturally consider as the probability of the new species at site $x$ to be species $i$.
 Hence, the new species at each update is chosen randomly from the local neighborhood with a selective advantage for neighbors that
 have a higher ability to exploit the resource present at the site to be updated.
 Finally, we point out that, because there are countably many lattice sites, updates at different sites cannot occur simultaneously,
 so the new type can indeed be chosen from the local neighborhood without ambiguity.

% % % % % % % % % % % % % % % % % % % % % % % % % % % % % % % % % % % % % % % % % % % % % % % % % % % % % % % % % % % % % % % % % % % % % 

\subsection*{The mean-field approximation.}
 The mean-field model is naturally derived from the original interacting particle system by excluding both space and
 stochasticity.
 For a discussion about the connections between interacting particle systems and their mean-field approximations, we refer the
 reader to Durrett and Levin \cite{durrett_levin_1994b}.
 To motivate the definition of the mean-field model, we first replace the infinite lattice by a finite complete graph, that is a
 finite graph in which each site is connected to all other sites including itself, so that all sites have the same neighborhood
 that consists of all the vertex set.
 Describing the interactions as previously but using this universal neighborhood rather than local neighborhoods results in
 a nonspatial environment in which the types at different sites are independent and identically distributed random variables.
 In particular,
 $$ u_i (t) \ = \ P \,(\eta_t (x) = i) \quad \hbox{for all} \ t \geq 0 \ \hbox{and} \ i = 1, 2, \ldots, n $$
 where $x$ is a denominated site, are probabilities that indeed no longer depend on $x$.
 The mean-field approximation is then defined as the deterministic system of coupled ordinary differential equations that describes
 the evolution of these probabilities.
 To derive this deterministic system from the stochastic dynamics, we first observe that
\begin{equation}
\label{eq:decompose}
 \begin{array}{rcl}
     u_i (t + h) - u_i (t) & = &
     P \,(\eta_{t + h} (x) = i, \,\eta_t (x) \neq i) - P \,(\eta_{t + h} (x) \neq i, \,\eta_t (x) = i) \vspace{8pt} \\ & = &
    \displaystyle \sum_{j \neq i} \ P \,(\eta_{t + h} (x) = i, \,\eta_t (x) = j) -
    \displaystyle \sum_{j \neq i} \ P \,(\eta_{t + h} (x) = j, \,\eta_t (x) = i). \end{array}
\end{equation}
 Moreover, the probability that, at the next update, a type $j$ is replaced by a type $i$ is also the probability $u_j$ that the site
 chosen to be updated is of type $j$ times the probability \eqref{eq:proba} that the new type selected is type $i$.
 Since in addition sites are independent and the probability of an update at site $x$ in a very short time interval is roughly equal
 to the length of this interval, we obtain
\begin{equation}
\label{eq:j->i}
  P \,(\eta_{t + h} (x) = i, \,\eta_t (x) = j) \ = \
  h \ u_j (t) \times \frac{a_{i, j} \,u_i (t)}{a_{1, j} \,u_1 (t) + \cdots + a_{n, j} \,u_n (t)} \ + \ o (h)
\end{equation}
 for $h > 0$ small.
 Substituting \eqref{eq:j->i} in \eqref{eq:decompose}, dividing by $h$, and taking the limit as $h \to 0$, we get
\begin{equation}
\label{eq:mean-field-1}
  \frac{du_i}{dt} \ = \
  \sum_{j = 1}^n  \ \frac{a_{i, j} \,u_i}{a_{1, j} \,u_1 + \cdots + a_{n, j} \,u_n} \times u_j \ - \
  \sum_{j = 1}^n  \ \frac{a_{j, i} \,u_j}{a_{1, i} \,u_1 + \cdots + a_{n, i} \,u_n} \times u_i
\end{equation}
 for all $i = 1, 2, \ldots, n$.
 Each fraction in the first sum of \eqref{eq:mean-field-1} is the relative ability of species $i$ to exploit resource $j$ over
 all the community which, when different species have different abilities to exploit each of the resources, strongly depends on
 the species distribution.
 Therefore, the first sum can be seen as the overall birth rate of individuals of species $i$.
 Note also that the second sum is simply equal to $u_i$, so the mean-field approximation reduces to the system
\begin{equation}
\label{eq:mean-field-2}
  \frac{du_i}{dt} \ = \ \bigg(\sum_{j = 1}^n \ \frac{a_{i, j} \,u_j}{a_{1, j} \,u_1 + \cdots + a_{n, j} \,u_n} \ - \ 1 \bigg) \ u_i
  \qquad \hbox{for} \ i = 1, 2, \ldots, n.
\end{equation}
 The minus one term indicates that the per capita death rate of each species is equal to 1, which results from the fact that, in
 the original stochastic process, each lattice site is updated at the arrival times of a Poisson process with intensity 1.

% % % % % % % % % % % % % % % % % % % % % % % % % % % % % % % % % % % % % % % % % % % % % % % % % % % % % % % % % % % % % % % % % % % % % 

\subsection*{The importance of space.}
 Even though, following the traditional terminology, the mean-field model is called \emph{nonspatial} deterministic
 approximation, both models, in fact, include space, but at different levels of details.
 In the interacting particle system, the state space consists of spatial configurations of species on the lattice that can only
 interact with their nearest neighbors therefore space takes the form of local interactions, which is referred
 to as \emph{explicit space}.
 The mean-field model is obtained by replacing the lattice by a complete graph, in particular space is again included but it now
 takes the form of global interactions, which is referred to as \emph{implicit space}.
 The interacting particle system describes more suitably populations in which individuals are static and can only disperse their
 offspring over short distances, whereas the mean-field model is more appropriate to mimic populations in which individuals
 either move or have the ability to disperse their offspring over very long distances.
 Most mathematical models introduced in the life science literature that describe spatially explicit interacting populations
 consist however of systems of ordinary differential equations that assume that populations are well-mixed.
 The objective of this study is to prove both analytically and numerically that, at least for the type of interactions considered
 in this article, the use of the mean-field approach is invalid for a wide range of parameters in the sense that both models
 result in different predictions.
 Before stating our results, it is necessary to specify what is meant here by: both models result in different predictions.
 In the interacting particle system, since the types at different lattice sites are not independent, some quantities of interest
 are the so-called spatial correlations, the correlations between the types at two sites as a function of the distance between
 these two sites.
 There is obviously no analog of these quantities in the mean-field model since it does not include any geometrical structure.
 Also, the traditional approach to compare both models is to compare the density of a given species in the mean-field model with
 the probability of a denominated site being occupied by this species type in the interacting particle system:
 whenever one of these two quantities converges to zero whereas the other one is bounded away from zero uniformly in time, we say
 that both models result in different predictions.
 In other words, both models are said to result in different predictions for a given set of parameters whenever the set of
 species surviving in the long run differ between both models.

% % % % % % % % % % % % % % % % % % % % % % % % % % % % % % % % % % % % % % % % % % % % % % % % % % % % % % % % % % % % % % % % % % % % % 

\subsection*{Biological interactions.}
 The limiting behavior of both the interacting particle system and its mean-field approximation depends upon two factors:
 the initial distribution of species, which is a spatial configuration on the lattice for the first model and an $n$-tuple
 of densities that sums up to one for the second model, and more importantly the matrix $M$ that describes the type of
 biological interactions among species.
 These different types of biological interactions are modeled by the relationships among the matrix coefficients, which
 can be conveniently described by using the terminology of game theory.
 Also, we say that a species is a \emph{cooperator}, respectively a \emph{defector}, if it has a lower ability, respectively
 a better ability, than all other species to exploit the resource it produces.
 That is, in the column corresponding to the resource produced by this species, the smallest, respectively the largest,
 coefficient is the coefficient on the diagonal.
 While a species can be neither a cooperator nor a defector in communities with more than two species, each species is
 either a cooperator or a defector in its relationship with a second species, which induces three types of interactions for
 each pair of species: we call \emph{cooperation} the interaction between two cooperators, \emph{competition} the
 interaction between two defectors, and \emph{cheating} relationship the interaction between a cooperator and a defector
 in which case the defector is called a \emph{cheater}.
 To understand the variety of biological interactions that can be encoded by the parameter matrix, it is important to
 point out that, even if the individuals of a certain species do not strictly speaking produce resources that are beneficial
 for their species or the other species, their death produces nevertheless a resource available to their neighbors: space.
 In particular, the matrix coefficient $a_{i,j}$ can represent more generally the ability of species $i$ to exploit space
 with whatever an individual of species $j$ produces:
 space with resources that are beneficial for the nearby individuals such as nutrients, or at the opposite space with
 biochemicals that have a harmful effect on the growth of other species and reduce their ability to exploit the spatial
 location available in the context of allelopathic interactions.

\indent Focusing on interactions involving only two species, when none of the species is a producer, the interactions are
 described by the matrix in which both columns are equal.
 This models competition for space with the common coefficient in each row representing the ability of the corresponding
 species to invade space.
 The species with the better ability to exploit space is a defector while the other species is a cooperator.
 Competition for space is common in plant communities, but also occurs among territorial animals such as the grey wolf,
 \emph{Canis lupus}, although in this case this is more symptomatic of a defensive behavior.
 In the former context, interactions among plants that spread nearby through their rhizome such as \emph{Euphorbia} are
 better captured by the interacting particle system while the mean-field model describes more suitably the interactions
 among seed plants that can disperse over large distances through anemochory, zoochory or hydrochory such as the dandelion,
 \emph{Taraxacum officinale}.
 When both species indeed produce resources and these resources are more beneficial for the other species, a symbiotic
 relationship called mutualism, both species are cooperators: the coefficients off diagonal are larger than the
 coefficients on the diagonal.
 Mutualism is common for instance in terrestrial plants which live in association with mycorrhizal fungi, with the plant
 providing carbon to the fungus and the fungus providing nutrients to the plant.
 As mentioned above, our models are also suited to mimic allelopathic interactions in which a so-called allelopathic
 species produces biochemicals that have detrimental effects on the growth and reproduction of other organisms.
 In this case and when the other species is indeed susceptible to the allelochemicals, the allelopathic species is a
 defector since it has a better ability to invade a spatial location saturated with allelochemicals, while the susceptible
 species can be either a defector or a cooperator depending on its relative ability to invade spatial locations void of
 allelochemicals.
 Allelopathic interactions are common among invasive plants such as the Spotted Knapweed, \emph{Centaurea maculosa},
 that uses biochemicals as a defense against other plants and herbivory.
 In predator-prey systems, the resource that the predator exploits is not produced by the prey, it is the prey itself,
 therefore one needs to slightly reinterpret the microscopic rules of the models:
 a spatial location is not invaded by a nearby predator just after the prey dies but just before, with the death of the
 prey being caused by the predator.
 In this context, the prey is always a cooperator, while the predator is either a defector or a cooperator depending on
 its relative ability to invade empty spatial locations.
 Note however that our models do not reproduce the oscillations typical of certain predator-prey systems such as the
 one designed in the popular Huffaker's mite experiment \cite{huffaker_1958} since it assumes that the predator can
 survive in the absence of the prey.
 A similar choice of parameters can also be used to model some cases of facultative parasitism in which the parasite can
 complete its life cycle without being associated with the host, provided one identifies the infected host with the
 parasite itself, which is only valid in contexts where infected hosts are sterile.
 The analysis of the mean-field model in this article and the analysis of the interacting particle system in the
 companion paper \cite{lanchier_2010} give a precise picture of the asymptotic properties of systems involving only two
 species, which includes all the biological interactions mentioned in this paragraph.

\indent For larger ecological communities involving $n$ species, the appropriate relationships among the matrix
 coefficients can be deduced by looking at the type of biological interactions that relates any two species.
 In this case, the analysis of the mean-field model pays a particular attention to coexistence and gives sharp results
 for communities that exhibit rock-paper-scissors dynamics in which three species together may coexist whereas, in
 the absence of any of these three species, the other two species cannot:
 paper covers rock, scissors cut paper and rock smashes scissors.
 Such dynamics have been observed in the common side-blotched lizard, \emph{Uta stansburiana}, that exhibits color
 polymorphisms \cite{sinervo_lively_1996}.
 Rather than including three different species, the system consists of male individuals of the common side-blotched lizard
 with three different phenotypes: orange-throated males, yellow stripe throated males, and blue-throated males,
 corresponding to three different mating strategies.
 While evolutionary theory predicts selection of the best fit, all three phenotypes coexist in nature which, according
 to the model of Sinervo and Lively \cite{sinervo_lively_1996}, is due to the fact that the three phenotypes interact
 like in the rock-paper-scissors game.
 Our numerical and anaytical results give a valuable insight into this type of dynamics.

%%%%%%%%%%%%%%%%%%%%%%%%%%%%%%%%%%%%%%%%%%%%%%%%%%%%%%%%%%%%%%%%%%%%%%%%%%%%%%%%%%%%%%%%%%%%%%%%%%%%%%%%%%%%%%%%%%%%%%%%%%%%%%%%%%%%%%%%%%

\section{Analytical results for the nonspatial deterministic model}
\label{sec:deterministic}

\indent In this section, we collect a number of analytical results for the mean-field model.
 We start with general lemmas that will be applied repeatedly afterwards to understand the limiting behavior of the system with
 two and three species, respectively.
 The first step is to identify the positive invariant sets of the mean-field model.
 For $I \subset \{1, 2, \ldots, n \}$, $I \neq \varnothing$, we let
 $$ \begin{array}{rcl}
           S_I & = & \Big\{(u_1, u_2, \ldots, u_n) \in \R_+^n : u_i = 0 \ \hbox{for} \ i \notin I \ \hbox{and} \ \sum_{i \in I} \ u_i = 1 \Big\} \vspace{8pt} \\
    \tilde S_I & = & \Big\{(u_1, u_2, \ldots, u_n) \in S_I : u_i > 0 \ \hbox{for} \ i \in I \Big\} \end{array} $$
 and simply write $S_I = S_n$ and $\tilde S_I = \tilde S_n$ when $I = \{1, 2, \ldots, n \}$.
 From the fact that $u_i$ can be interpreted as a probability, it should be intuitively clear that the set $S_n$
 is positive invariant for the mean-field model.
 Also, since individuals of either type cannot appear spontaneously in the interacting particle system, the same holds in the
 mean-field model, so each set $S_I$ should be positive invariant.
 These statements are made rigorous in the following lemma.
\begin{lemma}
\label{lem:pi}
 For all $I \subset \{1, 2, \ldots, n \}$, $I \neq \varnothing$, the set $S_I$ is positive invariant, i.e.,
 $$ (u_1 (0), \ldots, u_n (0)) \in S_I \ \ \hbox{implies that} \ \ (u_1 (t), \ldots, u_n (t)) \in S_I \ \ \hbox{for all} \ t > 0. $$
\end{lemma}
\begin{proof}
 This directly follows from the fact that
 $$ \sum_{i = 1}^n \ \frac{du_i}{dt} \ = \ 0 \quad \hbox{and} \quad u_i (t) = 0 \ \hbox{whenever} \ u_i (0) = 0 \ \hbox{for all} \ i \in \{1, 2, \ldots, n \} $$
 and by noticing that the second assertion together with the continuity of the trajectories implies that for any species $i$, we have
 $u_i (t) \geq 0$ at all times $t > 0$.
 This completes the proof.
\end{proof}
\begin{lemma}
\label{lem:invariant}
 For all $I \subset \{1, 2, \ldots, n \}$, $I \neq \varnothing$, the set $\tilde S_I$ is positive invariant.
\end{lemma}
\begin{proof}
 In view of Lemma \ref{lem:pi}, it suffices to prove that $u_i (t) > 0$ whenever $u_i (0) > 0$.
 Recalling the expression of the derivatives, we have
 $$ \frac{du_i}{dt} \ = \ \bigg(\sum_{j = 1}^n \ \frac{a_{i, j} \,u_j}{a_{1, j} \,u_1 + \cdots + a_{n, j} \,u_n} \ - \ 1 \bigg) \ u_i \ \geq \ - u_i. $$
 Therefore, we have $u_i (t) \geq u_i (0) \,e^{-t} > 0$ whenever $u_i (0) > 0$.
\end{proof} \\ \\
 Our last preliminary result indicates that, if the ability of species $i$ is strictly larger than the ability of species $j$ to exploit
 resources of either type, then species $j$ is driven to extinction when starting with a positive density of each species.
 The intuition behind this result is that the fitness of species $i$, which is a function of the configuration of species, is always strictly
 larger than the fitness of species $j$ \emph{regardless} of the configuration of the system.
\begin{lemma}
\label{lem:extinction}
 Assume that $a_{i,k} > a_{j,k}$ for all $k = 1, 2, \ldots, n$. Then
 $$ \lim_{t \to \ \infty} \ u_j (t) \ = \ 0 \ \ \hbox{whenever} \ \ u (0) = (u_1 (0), \ldots, u_n (0)) \in \tilde S_n. $$
\end{lemma}
\begin{proof}
 Since $\tilde S_n$ is positive invariant according to Lemma \ref{lem:invariant}, we have
 $$ \sum_{m = 1}^n \ u_m (t) \ = \ 1 \quad \hbox{and} \quad u_i (t) \in (0, 1) \ \hbox{for all} \ i = 1, 2, \ldots, n. $$
 Using in addition that each column of $M$ has at least one positive entry, we obtain
 $$ \sum_{m = 1}^n \ a_{m, k} \ u_m (t) \ > \ 0 \quad \hbox{and} \quad
    \min_{1 \leq m \leq n} \,\{a_{m,k} \} \ \leq \ \sum_{m = 1}^n \ a_{m, k} \ u_m (t) \ \leq \ \max_{1 \leq m \leq n} \,\{a_{m,k} \}. $$
 In particular, defining
 $$      F_{i,j} \ = \ \sum_{k = 1}^n \ \frac{(a_{i,k} - a_{j,k}) \,u_k}{a_{1,k} \,u_1 + \cdots + a_{n,k} \,u_n} \quad \hbox{and} \quad
    \Gamma_{i,j} \ = \ \min_{1 \leq k \leq n} \ \bigg\{\frac{a_{i,k} - a_{j,k}}{\max_{1 \leq m \leq n} \,\{a_{m,k} \}} \bigg\} $$
 and using that $a_{i,k} > a_{j,k}$ for all $k = 1, 2, \ldots, n$, we get
 $$ \frac{d}{dt} \,\bigg(\frac{u_i}{u_j} \bigg) \ = \
    \frac{u_i}{u_j} \ F_{i,j} \ \geq \
    \frac{u_i}{u_j} \ \bigg(\sum_{k = 1}^n  \ \frac{(a_{i,k} - a_{j,k}) \,u_k}{\max_{1 \leq m \leq n} \,\{a_{m,k} \}} \bigg) \ \geq \
    \frac{u_i}{u_j} \ \Gamma_{i,j}. $$
 Since $\Gamma_{i,j} > 0$, it follows that
 $$ \lim_{t \to \infty} \ u_j (t) \ \leq \ \frac{u_j (0)}{u_i (0)} \ \lim_{t \to \infty} \ \exp (- \Gamma_{i,j} t) \ = \ 0. $$
 This completes the proof.
\end{proof}

% % % % % % % % % % % % % % % % % % % % % % % % % % % % % % % % % % % % % % % % % % % % % % % % % % % % % % % % % % % % % % % % % % % % % 

\subsection*{The nonspatial two-type model.}
 Recall that in the presence of $n = 2$ species, there are only three possible types of interactions:
 cooperation, competition, and cheating.
 Interestingly, the nonspatial model also exhibits exactly three types of long-term behaviors which are directly connected to
 the type of relationship between the species.
 More precisely, we have the following results:
 cheating implies victory of the cheater, competition implies bistability, and cooperation implies coexistence.
 To prove these results, we first recall that
\begin{equation}
\label{eq:mean-field-2S}
\begin{array}{rcl}
 \displaystyle \frac{du_1}{dt} & = &
 \displaystyle \bigg(\frac{a_{1,2}}{a_{1,2} \,u_1 + a_{2,2} \,u_2} \ - \ \frac{a_{2,1}}{a_{1,1} \,u_1 + a_{2,1} \,u_2} \bigg) \ u_1 \,u_2 \vspace{8pt} \\
 \displaystyle \frac{du_2}{dt} & = &
 \displaystyle \bigg(\frac{a_{2,2}}{a_{1,2} \,u_1 + a_{2,2} \,u_2} \ - \ \frac{a_{1,2}}{a_{1,1} \,u_1 + a_{2,1} \,u_2} \bigg) \ u_1 \,u_2. \end{array}
\end{equation}
 By invoking the positive invariance of the set $S_2$, we obtain $u_2 = 1 - u_1$ so the model \eqref{eq:mean-field-2S} can be simply reduced to the
 one-dimensional ordinary differential equation
\begin{equation}
\label{eq:mean-field-1D}
 \frac{du_1}{dt} \ = \ \bigg(\frac{a_{1,2}}{a_{1,2} \,u_1 + a_{2,2} \,(1 - u_1)} \ - \ \frac{a_{2,1}}{a_{1,1} \,u_1 + a_{2,1} \,(1 - u_1)} \bigg) \ u_1 \,(1 - u_1).
\end{equation}
 By setting the right-hand side of \eqref{eq:mean-field-1D} equal to zero, we find three solutions for $u_1$, namely,
 the two trivial solutions $u_1 = 0$ and $u_1 = 1$, and the nontrivial solution
\begin{equation}
\label{eq:u1}
 \bar u_1 \ = \ \frac{a_{2,1} \,(a_{1,2} - a_{2,2})}{a_{1,2} \,(a_{2,1} - a_{1,1}) + a_{2,1} \,(a_{1,2} - a_{2,2})}.
\end{equation}
 To express the equilibria of \eqref{eq:mean-field-2S}, we also let
\begin{equation}
\label{eq:u2}
 \bar u_2 \ = \ 1 - \bar u_1 \ = \ \frac{a_{1,2} \,(a_{2,1} - a_{1,1})}{a_{1,2} \,(a_{2,1} - a_{1,1}) + a_{2,1} \,(a_{1,2} - a_{2,2})}
\end{equation}
 and observe that $e_{1,2} = (\bar u_1, \bar u_2) \in (0, 1)^2$ if and only if
\begin{equation}
\label{eq:conditions}
 (a_{1,1} < a_{2,1} \ \ \hbox{and} \ \ a_{2,2} < a_{1,2}) \quad \hbox{or} \quad (a_{1,1} > a_{2,1} \ \ \hbox{and} \ \ a_{2,2} > a_{1,2}).
\end{equation}
 Note that the first condition in \eqref{eq:conditions} indicates that both species cooperate whereas the second condition indicates that they compete.
 In conclusion, the nonspatial two-type model \eqref{eq:mean-field-2S} always has $e_1 = (1, 0)$ and $e_2 = (0, 1)$ as its two trivial equilibria.
 In the presence of either cooperation or competition, there is a third nontrivial equilibrium given by $e_{1,2} = (\bar u_1, \bar u_2)$.
 We now investigate the global stability of these three equilibria in details.
\begin{theor}[Global dynamics]
\label{th:global-2}
 For the mean-field model \eqref{eq:mean-field-2S}, we have
\begin{enumerate}
 \item Species 1 is a cheater:
  if $a_{1,1} > a_{2,1}$ and $a_{1,2} > a_{2,2}$ then $e_1$ is globally stable, $e_2$ is unstable, and there is no interior equilibrium. \vspace{4pt}
 \item Species 2 is a cheater:
  if $a_{1,1} < a_{2,1}$ and $a_{1,2} < a_{2,2}$ then $e_1$ is unstable, $e_2$ is globally stable, and there is no interior equilibrium. \vspace{4pt}
 \item Competition:
  if $a_{1,1} > a_{2,1}$ and $a_{2,2} > a_{1,2}$ then the interior equilibrium $e_{1,2}$ is unstable whereas the trivial equilibria $e_1$ and $e_2$ are
  locally asymptotically stable, and the system converges to $e_1$ whenever $u_1 (0) > \bar u_1$ and to $e_2$ whenever $u_1 (0) < \bar u_1$. \vspace{4pt}
 \item Cooperation:
  if $a_{1,1} < a_{2,1}$ and $a_{2,2} < a_{1,2}$ then the interior equilibrium $e_{1,2}$ is globally stable whereas the trivial equilibria $e_1$ and
  $e_2$ are unstable.
\end{enumerate}
\end{theor}
\begin{proof}
 The first two statements follow directly from Lemma \ref{lem:extinction} and the fact that none of the two conditions in \eqref{eq:conditions} holds.
 To prove the last two statements, we first observe that \eqref{eq:conditions} holds in both cases, therefore we have $e_{1,2} \in (0, 1)^2$ and the
 global stability of the equilibria can be determined by looking at the global stability of their analog for the system \eqref{eq:mean-field-1D}.
 Now, letting
 $$ h (u_1) \ = \ \bigg(\frac{a_{1,2}}{a_{1,2} \,u_1 + a_{2,2} \,(1 - u_1)} \ - \ \frac{a_{2,1}}{a_{1,1} \,u_1 + a_{2,1} \,(1 - u_1)} \bigg) \ u_1 \,(1 - u_1), $$
 the stability of $u_1 \in \{0, 1, \bar u_1 \}$ are determined by the sign of
 $$ \begin{array}{rcl}
     h' (0) & = & \displaystyle \frac{a_{1,2} - a_{2,2}}{a_{2,2}} \qquad h' (1) \ = \ \displaystyle \frac{a_{2,1} - a_{1,1}}{a_{1,1}} \qquad \hbox{and} \vspace{8pt} \\
     h' (\bar u_1) & = & \displaystyle
    \frac{(a_{2,1} - a_{1,1}) \,(a_{1,2} - a_{2,2}) \,[a_{1,2} \,(a_{1,1} - a_{2,1}) + a_{2,1} \,(a_{2,2} - a_{1,2})]}{(a_{1,1} \,a_{2,2} - a_{1,2} \,a_{2,1})^2} \end{array} $$
 respectively.
 It is then straightforward to conclude that
\begin{enumerate}
\item When $a_{1,1} > a_{2,1}$ and $a_{2,2} > a_{1,2}$, the trivial equilibria 0 and 1 are locally stable while the nontrivial equilibrium $\bar u_1$ is unstable.
 Returning to the mean-field model \eqref{eq:mean-field-2S} and using that the set $S_2$ is positive invariant, we deduce that
 $$ \lim_{t \to \infty} \ (u_1 (t), u_2 (t)) \ = \
    \left\{\hspace{-3pt} \begin{array}{lll} e_1 & \hbox{when} & u_1 (0) > \bar u_1 \vspace{2pt} \\ e_2 & \hbox{when} & u_1 (0) < \bar u_1. \end{array} \right. $$
\item When $a_{1,1} < a_{2,1}$ and $a_{2,2} < a_{1,2}$, the trivial equilibria 0 and 1 are unstable while the nontrivial equilibrium $\bar u_1$ is stable.
 Returning to the mean-field model \eqref{eq:mean-field-2S} and using that $S_2$ is positive invariant, we deduce that for any initial condition in $\tilde S_2$ we have
 $$ \lim_{t \to \infty} \ u_1 (t) \ = \ \bar u_1 \quad \hbox{and} \quad \lim_{t \to \infty} \ u_2 (t) \ = \ \bar u_2. $$
\end{enumerate}
 This completes the proof of the theorem.
\end{proof} \\ \\
 Note that the results of Theorem \ref{th:global-2} can simply be expressed in terms of the relative ability of each species to exploit the
 resource it produces, defined as
\begin{equation}
\label{eq:theta}
 \theta_1 \ = \ \frac{a_{1,1}}{a_{1,1} + a_{2,1}} \qquad \hbox{and} \qquad \theta_2 \ = \ \frac{a_{2,2}}{a_{1,2} + a_{2,2}}
\end{equation}
 since species $i$ is a cooperator whenever $\theta_i < 1/2$ and a defector whenever $\theta_i > 1/2$.
 For a summary of the results of Theorem \ref{th:global-2}, we refer the reader to the right-hand picture of Figure \ref{fig:two-type}
 where the phase diagram of the nonspatial two-type model is represented in the $\theta_1$-$\theta_2$ plane.

% % % % % % % % % % % % % % % % % % % % % % % % % % % % % % % % % % % % % % % % % % % % % % % % % % % % % % % % % % % % % % % % % % % % % 

\subsection*{The nonspatial three-type model.}
 Note that in the presence of three species, a species cannot simply be classified as a cooperator or a defector since it can behave
 as a cooperator with respect to one species but as a defector with respect to another species.
 This aspect gives rise to very rich dynamics.
 In particular, while cooperation is the only mechanism that leads to coexistence in communities with two species, there are additional
 mechanisms that promote coexistence when a third species is introduced.
 In this subsection, we first look at the local stability of the trivial boundary equilibria, and the nontrivial boundary equilibria when
 they exist.
 Based on the analysis of the local stability of the boundary equilibria, we conclude the subsection with general conditions
 for global coexistence of all three species.
 These analytical results will be used later as a guide to perform numerical simulations of the spatial model in specific parameter regions in
 order to understand the role of space on the global dynamics.
 For simplicity and to fix the ideas, we state some of our results assuming that some species play particular roles in the interactions, but
 additional results can be deduced by replacing $(1, 2, 3)$ with $(\sigma (1), \sigma (2), \sigma (3))$ where $\sigma$ is any of the
 six permutations of the group $\mathfrak S_3$.
 To start our analysis of the nonspatial three-type model, we recall that the mean-field model with three species is given by
\begin{equation}
\label{eq:mean-field-3S}
  \begin{array}{l}
    \displaystyle \frac{du_i}{dt} \ = \ F_i (u) \ = \ \bigg[
    \displaystyle \frac{a_{i,1} \,u_1}{a_{1,1} \,u_1 + a_{2,1} \,u_2 + a_{3,1} \,u_3} \ + \
    \displaystyle \frac{a_{i,2} \,u_2}{a_{1,2} \,u_1 + a_{2,2} \,u_2 + a_{3,2} \,u_3} \vspace{8pt} \\ \hspace{160pt} + \
    \displaystyle \frac{a_{i,3} \,u_3}{a_{1,3} \,u_1 + a_{2,3} \,u_2 + a_{3,3} \,u_3} \ - \ 1 \,\bigg] \ u_i \ = \ G_i (u) \,u_i \end{array} 
\end{equation}
 where $i = 1, 2, 3$.
 Note that, since the set $S_3$ is positive invariant, the mean-field model \eqref{eq:mean-field-3S} can be reduced, by letting
 $u_3 = 1 - u_1 - u_2$, to the following two-dimensional system:
\begin{equation}
\label{eq:mean-field-2D}
  \begin{array}{l}
    \displaystyle \frac{du_i}{dt} \ = \ F_i (u) \ = \ \bigg[
    \displaystyle \frac{a_{i,1} \,u_1}{a_{1,1} \,u_1 + a_{2,1} \,u_2 + a_{3,1} \,\left(1- u_1-u_2\right)} \vspace{8pt} \\ \hspace{100pt} + \
    \displaystyle \frac{a_{i,2} \,u_2}{a_{1,2} \,u_1 + a_{2,2} \,u_2 + a_{3,2} \,\left(1- u_1-u_2\right)} \vspace{8pt} \\ \hspace{120pt} + \
    \displaystyle \frac{a_{i,3} \,\left(1- u_1-u_2\right)}{a_{1,3} \,u_1 + a_{2,3} \,u_2 + a_{3,3} \,\left(1- u_1-u_2\right)} \ - \ 1 \,\bigg] \ u_i \ = \ G_i (u) \,u_i \end{array}
\end{equation}
 where $i = 1, 2$.
 Note also that the analysis of the two-type nonspatial model in the previous subsection directly gives the expression of the boundary
 equilibria of the three-type nonspatial model as well as the conditions under which these equilibria indeed exist.
 There are always three trivial boundary equilibria given by
 $$ e_1 = (1, 0, 0) \qquad e_2 = (0, 1, 0) \qquad \hbox{and} \qquad e_3 = (0, 0, 1) $$
 respectively.
 By setting the right-hand side of the three equations in \eqref{eq:mean-field-3S} equal to zero, we find three additional solutions, namely
 $$ e_{1,2} = (\bar u_{1,2}, \bar u_{2,1}, 0) \qquad e_{2,3} = (0, \bar u_{2,3}, \bar u_{3,2}) \qquad \hbox{and} \qquad e_{3,1} = (\bar u_{1,3}, 0, \bar u_{3,1}) $$
 where for all $i \neq j$, the density $\bar u_{i,j}$ is given by
 $$ \bar u_{i,j} \ = \ \frac{a_{j,i} \,(a_{i,j} - a_{j,j})}{a_{i,j} \,(a_{j,i} - a_{i,i}) + a_{j,i} \,(a_{i,j} - a_{j,j})}. $$
 According to \eqref{eq:conditions}, for $j \equiv i + 1 \ \hbox{mod} \ 3$, the solution $e_{i,j}$ is indeed a nontrivial boundary equilibrium
 in the sense that both densities $\bar u_{i,j}$ and $\bar u_{j,i}$ lie in the interval $(0, 1)$ if and only if
 $$ (a_{i,i} < a_{j,i} \ \ \hbox{and} \ \ a_{j,j} < a_{i,j}) \quad \hbox{or} \quad (a_{i,i} > a_{j,i} \ \ \hbox{and} \ \ a_{j,j} > a_{i,j}). $$
 The next theorem shows the following connection between the three-type and two-type nonspatial models:
 when one species, say species 1, has a strictly higher ability than another species, say species 3, to exploit resources of either type,
 species 3 goes extinct and species 1 and 2 interact as described by the two-type model \eqref{eq:mean-field-2S}, just as in the absence of type 3.
\begin{theor}
\label{th:global-trivial}
 Assume that $a_{1,k} > a_{3,k}$ for all $k = 1, 2, 3.$
 Then for any initial condition taken in the positive invariant set $\tilde S_3$ we have the following alternative.
\begin{enumerate}
 \item Species 1 is a cheater: if $a_{1,1} > a_{2,1}$ and $a_{1,2} > a_{2,2}$ then the trajectory converges to $e_1$. \vspace{4pt}
 \item Species 2 is a cheater: if $a_{1,1} < a_{2,1}$ and $a_{1,2} < a_{2,2}$ then the trajectory converges to $e_2$. \vspace{4pt}
 \item Competition: if $a_{1,1} > a_{2,1}$ and $a_{2,2} > a_{1,2}$ then the trajectory of \eqref{eq:mean-field-3S} converges to either $e_1$ or $e_2$
  depending on the initial condition. \vspace{4pt} 
 \item Cooperation: if $a_{1,1} < a_{2,1}$ and $a_{2,2} < a_{1,2}$ then the trajectory of \eqref{eq:mean-field-3S} converges to $e_{1,2}$.
\end{enumerate}
\end{theor}
\begin{proof}
 In view of Lemma \ref{lem:extinction} and the fact that the set $S_3$ is positive invariant, for any $\ep > 0$ there exists a time $T > 0$ that only
 depends on $\ep$ and the initial condition such that
 $$ u_3 (t) < \ep \quad \hbox{and} \quad u_1 (t) + u_2 (t) > 1 - \ep \quad \hbox{for all} \ t > T. $$
 In other words, any trajectory starting in $\tilde S_3$ is trapped after a finite time into the set
 $$ S_3^{\ep} \ = \ \{(u_1, u_2, u_3) \in [0, 1]^3 : u_1 + u_2 + u_3 = 1 \ \hbox{and} \ u_3 \leq \ep \}. $$
 The main idea of the proof is to use the fact that after entering the set $S_3^{\ep}$ the three-type system behaves similarly to the two-type
 system \eqref{eq:mean-field-2S} when the parameter $\ep$ is sufficiently small, and then to apply Theorem \ref{th:global-2}.
 Since all four cases can be treated similarly, we only provide the details of the proof for the first case.
 By applying the chain rule, we obtain
 $$ \begin{array}{rcl}
    \displaystyle \frac{d}{dt} \,\bigg(\frac{u_1}{u_2} \bigg) & = &
    \displaystyle \frac{u_1}{u_2} \ \sum_{k = 1}^3 \ \frac{(a_{1,k} - a_{2,k}) \,u_k}{a_{1,k} \,u_1 + a_{2,k} \,u_2 + a_{3,k} \,u_3} \vspace{8pt} \\ & \geq &
    \displaystyle \frac{u_1}{u_2} \ \bigg(\sum_{k = 1}^2 \ \frac{(a_{1,k} - a_{2,k}) \,u_k}{\max_{1 \leq m \leq 3} \{a_{m,k} \}} \ + \
    \frac{(a_{1,3} - a_{2,3}) \,u_3}{a_{1,3} \,u_1 + a_{2,3} \,u_2 + a_{3,3} \,u_3} \bigg). \end{array} $$
 Then, we distinguish two cases.
\begin{enumerate}
 \item If $a_{1,3} > a_{2,3}$ then Lemma \ref{lem:extinction} implies that $\lim_{t \to \infty} u_2 = 0$. \vspace{4pt}
 \item If $a_{1,3} < a_{2,3}$ then, using that $a_{1,k} > 0$ for all $k = 1, 2, 3$, we obtain 
 $$ a_{1,3} \,u_1 (t) + a_{2,3} \,u_2 (t) + a_{3,3} \,u_3 (t) \ \geq \ a_{1,3} \,u_1 (t) + a_{2,3} \,u_2 (t) \ > \ (1 - \epsilon) \,\min \{a_{1,3}, a_{2,3} \} $$
 for all times $t > T$.
 Therefore, we have
 $$ \frac{d}{dt} \,\bigg(\frac{u_1}{u_2} \bigg) \ \geq \
    \frac{u_1}{u_2} \ \bigg(\min_{k = 1, 2} \bigg\{\frac{(a_{1,k} - a_{2,k}) \,(1 - \ep)}{\max_{1 \leq m \leq 3} \,\{a_{m,k} \}} \bigg\} \ + \
    \frac{(a_{1,3} - a_{2,3}) \,\epsilon}{(1 - \ep) \,\min \,\{a_{1,3}, a_{2,3} \}} \bigg). $$
 Since $\ep$ can be arbitrary small and $a_{1,k} - a_{2,k} > 0$ for $k = 1, 2$, we have $\lim_{t \to \infty} u_2 = 0$.
\end{enumerate}
 In both cases, the trajectory converges to $e_1$.
 This completes the proof.
\end{proof} \\ \\
 To identify general conditions for tristability on the one hand, and coexistence of all three species on the
 other hand, we now look at the local stability of the trivial boundary equilibria.
\begin{theor}[Local stability of $e_1$]
\label{th:local-trivial}
 We have the following alternative:
\begin{enumerate}
 \item If $a_{1,1} > a_{2,1}$ and $a_{1,1} > a_{3,1}$ then $e_1$ is a sink, i.e., $e_1$ is locally asymptotically stable. \vspace{2pt}
 \item If $a_{1,1} < a_{2,1}$ and $a_{1,1} < a_{3,1}$ then $e_1$ is a source. \vspace{2pt}
 \item If $a_{2,1} < a_{1,1} < a_{3,1}$ or $a_{2,1} > a_{1,1} > a_{3,1}$ then $e_1$ is a saddle node.
\end{enumerate}
\end{theor}
\begin{proof}
 The Jacobian matrix of the system \eqref{eq:mean-field-3S} evaluated at equilibrium $e_1$ is given by
\begin{equation}
\label{Jacobian_e1}
 J_{e_1} \ = \ \left(\begin{array}{ccc}
     0 & \displaystyle \frac{a_{1,1} - a_{2,1}}{a_{1,1}} & \displaystyle \frac{a_{1,1} - a_{3,1}}{a_{1,1}} \vspace{6pt} \\
     0 & \displaystyle \frac{a_{2,1} - a_{1,1}}{a_{1,1}} & 0 \vspace{6pt} \\
     0 & 0 & \displaystyle \frac{a_{3,1} - a_{1,1}}{a_{1,1}} \end{array} \right).
\end{equation}
 The eigenvalues of \eqref{Jacobian_e1} are given by
 $$ \lambda_1 = 0 \qquad \lambda_2 = \frac{a_{2,1} - a_{1,1}}{a_{1,1}} \qquad \lambda_3 = \frac{a_{3,1} - a_{1,1}}{a_{1,1}} $$
 with respective eigenspaces
 $$ E_{\lambda_1} = \Span ((1, 0, 0)) \qquad E_{\lambda_2} = \Span ((-1, 1, 0)) \qquad E_{\lambda_3} = \Span ((-1, 0, 1)). $$
 In particular, for $j = 2, 3$, species $j$ can invade species 1 in the invariant manifold $u_1 + u_j = 1$ whenever $a_{1,1} < a_{j,1}$
 while it cannot whenever $a_{1,1} > a_{j,1}$.
 The three statements follow.
\end{proof} \\ \\
 From Theorem \ref{th:local-trivial}, one deduces immediately a general condition under which the nonspatial three-type system
 is tristable, which is given by Corollary \ref{cor:tristability} below.
 More importantly, this theorem allows to identify situations in which there is strong coexistence of all three species in terms of
 permanence of the system, whereas any two species cannot coexist in the absence of the third one.
 This indicates in particular that, in the presence of three species, coexistence is possible even in the absence of global cooperation,
 thus the existence of additional mechanisms that promote coexistence while the number of species increases.
 A more detailed discussion about this important result, which is stated below in Theorem \ref{th:heteroclinic}, is given at the end
 of this section.
\begin{corol}[Tristability]
\label{cor:tristability}
 All three trivial boundary equilibria $e_1$, $e_2$ and $e_3$ are simultaneously locally asymptotically stable
 whenever $a_{i,i} > a_{j,i}$ for all $i \neq j$.
\end{corol}
\begin{proof}
 This follows directly from Theorem \ref{th:local-trivial}.
\end{proof} \\ \\
 Numerical simulations further indicate that, under the conditions of Corollary \ref{cor:tristability}, for almost all initial
 conditions, i.e., all initial conditions excluding the ones that belong to a set with Lebesgue measure zero, 
 the system \eqref{eq:mean-field-3S} converges to one of the three trivial boundary equilibria.
 The limiting behavior of the system depends on both the parameters and the initial condition.

\begin{theor}[Heteroclinic cycle]
\label{th:heteroclinic} If the following inequalities hold
\begin{equation}
\label{hc}
 a_{3,1} < a_{1,1} < a_{2,1} \hspace{25pt}
 a_{1,2} < a_{2,2} < a_{3,2} \hspace{25pt}
 a_{2,3} < a_{3,3} < a_{1,3}
\end{equation}
 then there is a heteroclinic cycle connecting $e_1 \to e_2 \to e_3 \to e_1$.
 Moreover, the heteroclinic cycle is locally asymptotically stable whenever
\begin{equation}
\label{hc_s}
 a_{1,1} > \frac{a_{2,1} + a_{3,1}}{2} \qquad
 a_{2,2} > \frac{a_{1,2} + a_{3,2}}{2} \qquad
 a_{3,3} > \frac{a_{1,3} + a_{2,3}}{2}
\end{equation}
 while the heteroclinic cycle is repelling whenever
\begin{equation}
\label{hc_u}
 a_{1,1} < \frac{a_{2,1} + a_{3,1}}{2} \qquad
 a_{2,2} < \frac{a_{1,2} + a_{3,2}}{2} \qquad
 a_{3,3} < \frac{a_{1,3} + a_{2,3}}{2}
\end{equation}
 which implies that the system \eqref{eq:mean-field-3S} is permanent, i.e., there exists a compact interior attractor that attracts all
 points in the positive invariant set $\tilde S_3$.
\end{theor}
\begin{proof}
 By Theorem \ref{th:global-2}, the inequalities in \eqref{hc} indicate that the system \eqref{eq:mean-field-3S} only has the three trivial boundary
 equilibria as its boundary equilibria.
 Moreover, from the proof of Theorem \ref{th:local-trivial}, we know that the transversal eigenvalue at equilibrium $e_i$ in direction $u_j$ is given by 
 $$ \bigg(\frac{1}{u_j} \ \frac{du_j}{dt} \bigg) (e_i) \ = \ G_j (e_i) \ = \ \frac{a_{j,i} - a_{i,i}}{a_{i,i}}. $$
 In order to obtain a heteroclinic cycle $e_1 \to e_2 \to e_3 \to e_1$, each trivial equilibrium $e_i$ must have one positive and one negative transversal
 eigenvalue, and they must be arranged in a cyclic pattern.
 In particular, \eqref{hc} is a sufficient condition for the existence of a simple heteroclinic cycle.
 To prove the second part of the theorem, we let $\Gamma$ be a heteroclinic cycle which consists of the three trivial boundary equilibria $e_1, e_2, e_3$
 connected by heteroclinic orbits.
 Following the approach of Hofbauer and Sigmund \cite{hofbauer_sigmund_1998}, we define the \emph{characteristic matrix} of the heteroclinic cycle
 $\Gamma$ as the matrix whose entry in row $i$ and column $j$ is the external eigenvalue $G_j (e_i)$.
 This matrix is given by
 $$ C \ = \ \left(\begin{array}{ccc}
            \displaystyle 0 & \displaystyle \frac{a_{2,1} - a_{1,1}}{a_{1,1}} & \displaystyle \frac{a_{3,1} - a_{1,1}}{a_{1,1}} \vspace{6pt} \\
            \displaystyle \frac{a_{1,2} - a_{2,2}}{a_{2,2}} & \displaystyle 0 & \displaystyle \frac{a_{3,2} - a_{2,2}}{a_{2,2}} \vspace{6pt} \\
            \displaystyle \frac{a_{1,3} - a_{3,3}}{a_{3,3}} & \displaystyle \frac{a_{2,3} - a_{3,3}}{a_{3,3}} & \displaystyle 0  \end{array} \right). $$
 Then, defining $v = (1, 1, 1)$, we have
\begin{equation}
\label{hc_sign}
  C \cdot v^T \ = \
    \left(\begin{array}{c}
     G_1 (e_1) + G_2 (e_1) + G_3 (e_1) \vspace{4pt} \\
     G_1 (e_2) + G_2 (e_2) + G_3 (e_2) \vspace{4pt} \\
     G_1 (e_3) + G_2 (e_3) + G_3 (e_3) \end{array} \right) \ = \
    \left(\begin{array}{c}
    \displaystyle \frac{a_{2,1} + a_{3,1} - 2 a_{1,1}}{a_{1,1}} \vspace{6pt} \\
    \displaystyle \frac{a_{1,2} + a_{3,2} - 2 a_{2,2}}{a_{2,2}} \vspace{6pt} \\
    \displaystyle \frac{a_{2,3} + a_{1,3} - 2 a_{3,3}}{a_{3,3}} \end{array} \right)
\end{equation}
 By Theorem 17.5.1 in Hofbauer and Sigmund \cite{hofbauer_sigmund_1998}, the stability of the heteroclinic cycle can be determined based on the sign of the
 coordinates of the vector in \eqref{hc_sign}.
 More precisely,
\begin{enumerate}
 \item If all the coordinates of the vector \eqref{hc_sign} are positive, then the heteroclinic cycle $\Gamma$ is repelling.
  In this case, the system \eqref{eq:mean-field-3S} is permanent. \vspace{4pt}
 \item If all the coordinates of the vector \eqref{hc_sign} are negative, then the heteroclinic cycle $\Gamma$ is asymptotically stable in $S_3$.
  In this case, $\Gamma$ is an attractor of the system \eqref{eq:mean-field-3S}.
\end{enumerate}
 This completes the proof.
\end{proof} \\ \\
 Under the conditions in \eqref{hc}, there is a weak form of coexistence of the three species in the sense that none of the densities converges
 to zero.
 However, the trajectories may get closer and closer to the boundaries.
 Under the additional conditions in \eqref{hc_u} there is a strong form of coexistence in the sense that the theorem guarantees that the trajectories
 are eventually bounded away from the boundaries: the limit inferior of each of the three densities is larger than a positive constant.
 To find general conditions under which species coexist as in the two-type model due to a cooperative behavior, we now look at the stability of the
 nontrivial boundary equilibria.
\begin{theor}[Local stability of $e_{1,2}$]
\label{th:local-nontrivial}
 Assume that
\begin{equation}
\label{eq:conditions-bis}
 (a_{1,1} < a_{2,1} \ \hbox{and} \ a_{2,2} < a_{1,2}) \quad \hbox{or} \quad (a_{1,1} > a_{2,1} \ \hbox{and} \ a_{2,2} > a_{1,2})
\end{equation}
 so that $\bar u_{1,2} \in (0, 1)$ and $\bar u_{2,1} \in (0, 1)$ according to Theorem \ref{th:global-2}. Fix
 $$ \Delta_{1,2} \ = \ (2 a_{3,1} - a_{1,1} - a_{2,1}) (a_{1,2} - a_{2,2}) + (2 a_{3,2} - a_{1,2} - a_{2,2}) (a_{2,1} - a_{1,1}) $$
 a quantity that we call the invadibility of species 3.
 Then, we have the following.
\begin{enumerate}
 \item If $a_{1,1} < a_{2,1}$ and $a_{22} < a_{1,2}$ and $\Delta_{1,2} < 0$, then $e_{1,2}$ is locally asymptotically stable.
  In particular, species 3 cannot invade species 1 and 2 in their equilibrium. \vspace{4pt}
 \item If $a_{1,1} < a_{2,1}$ and $a_{22} < a_{1,2}$ and $\Delta_{1,2} > 0$, then $e_{1,2}$ is repelling.
  In particular, species 3 can invade species 1 and 2 in their equilibrium. \vspace{4pt}
 \item If $a_{1,1} > a_{2,1}$ and $a_{22} > a_{1,2}$ and $\Delta_{1,2} < 0$, then $e_{1,2}$ is a source.
  In particular, species 3 can invade species 1 and 2 in their equilibrium. \vspace{4pt}
 \item If $a_{1,1} > a_{2,1}$ and $a_{22} > a_{1,2}$ and $\Delta_{1,2} > 0$, then $e_{1,2}$ is a saddle node.
  In particular, species 3 cannot invade species 1 and 2 in their equilibrium.
\end{enumerate}
\end{theor}
\begin{proof}
 Recall that under the first condition in \eqref{eq:conditions-bis}, the nontrivial boundary equilibrium $e_{1,2}$ is globally stable (cooperation) in
 the two-dimensional manifold
 $$ \{(u_1, u_2, u_3) \in \R_+ : u_1, u_2 \in (0, 1) \ \hbox{and} \ u_1 + u_2 = 1 \} $$
 whereas under the second condition in \eqref{eq:conditions-bis} it is unstable (competition) in this manifold.
 To have a complete picture of the local stability of $e_{1,2}$ and deduce the invadability of species 3, we look at the Jacobian matrix associated
 with the system \eqref{eq:mean-field-3S} at point $e_{1,2}$
 $$ J_{e_{1,2}} \ = \ \left(\begin{array}{ccc}
                \displaystyle \frac{\partial F_1}{\partial u_1} \,(e_{1,2}) &
                \displaystyle \frac{\partial F_1}{\partial u_2} \,(e_{1,2}) &
                \displaystyle \frac{\partial F_1}{\partial u_3} \,(e_{1,2}) \vspace{8pt} \\
                \displaystyle \frac{\partial F_2}{\partial u_1} \,(e_{1,2}) &
                \displaystyle \frac{\partial F_2}{\partial u_2} \,(e_{1,2}) &
                \displaystyle \frac{\partial F_2}{\partial u_3} \,(e_{1,2}) \vspace{8pt} \\
                \displaystyle \frac{\partial F_3}{\partial u_1} \,(e_{1,2}) &
                \displaystyle \frac{\partial F_3}{\partial u_2} \,(e_{1,2}) &
                \displaystyle \frac{\partial F_3}{\partial u_3} \,(e_{1,2}) \end{array}\right). $$
 To express the coefficients of the matrix, we introduce
 $$ D_j \ = \ (a_{1,j} \,u_1 + a_{2,j} \,u_2 + a_{3,j} \,u_3) (e_{1,2}) \ = \ a_{1,j} \,\bar u_{1,2} + a_{2,j} \,\bar u_{2,1} \quad \hbox{for} \ j = 1, 2, 3. $$
 Since $e_{1,2}$ is a boundary equilibrium, $G_1 (e_{1,2}) = G_2 (e_{1,2}) = 0$. Therefore
 $$ \begin{array}{rclcrcl}
    \displaystyle \frac{\partial F_1}{\partial u_1} \,(e_{1,2}) & = &
    \displaystyle \Bigg[\frac{a_{1,1} \,a_{2,1}}{D_1^2} - \frac{a_{1,2}^2}{D_2^2} \Bigg] \ \bar u_{1,2} \,\bar u_{2,1} & \ \ &
    \displaystyle \frac{\partial F_1}{\partial u_2} \,(e_{1,2}) & = &
    \displaystyle - \ \Bigg[\frac{a_{1,1} \,a_{2,1}}{D_1^2} - \frac{a_{1,2}^2}{D_2^2} \Bigg] \ \bar u_{1,2}^2 \vspace{6pt} \\
    \displaystyle \frac{\partial F_2}{\partial u_1} \,(e_{1,2}) & = &
    \displaystyle \Bigg[\frac{a_{2,1}^2}{D_1^2} - \frac{a_{1,2} \,a_{2,2}}{D_2^2} \Bigg] \ \bar u_{2,1}^2 & \ \ &
    \displaystyle \frac{\partial F_2}{\partial u_2} \,(e_{1,2}) & = &
    \displaystyle  - \ \Bigg[\frac{a_{2,1}^2}{D_1^2} - \frac{a_{1,2} \,a_{2,2}}{D_2^2} \Bigg] \ \bar u_{1,2} \,\bar u_{2,1} \end{array} $$
 Using that the third coordinate of $e_{1,2}$ is equal to zero, we also have
 $$ \begin{array}{l}
    \displaystyle \frac{\partial F_3}{\partial u_1} \,(e_{1,2}) \ = \ \frac{\partial F_3}{\partial u_2} \,(e_{1,2}) \ = \ 0 \vspace{8pt} \\
    \displaystyle \frac{\partial F_3}{\partial u_3} \,(e_{1,2}) \ = \ G_3 (e_{1,2}) \ = \ \frac{a_{3,1} \,\bar u_{1,2}}{D_1} \ + \ \frac{a_{3,2} \,\bar u_{2,1}}{D_2} \ - \ 1 \end{array} $$
 Note that the determinant of the first $2 \times 2$ block of the Jacobian matrix $J_{e_{1,2}}$ is given by
 $$ \frac{\partial F_1}{\partial u_1} \,(e_{1,2}) \times \frac{\partial F_2}{\partial u_2} \,(e_{1,2}) \ - \
    \frac{\partial F_1}{\partial u_2} \,(e_{1,2}) \times \frac{\partial F_2}{\partial u_1} \,(e_{1,2}) \ = \ 0 $$
 while the trace of the matrix is given by
 $$ \frac{\partial F_1}{\partial u_1} \,(e_{1,2}) \ + \ \frac{\partial F_2}{\partial u_2} \,(e_{1,2}) \ = \
    \Bigg[\frac{a_{2,1} \,(a_{1,1} - a_{2,1})}{D_1^2} + \frac{a_{1,2} \,(a_{2,2} - a_{1,2})}{D_2^2} \Bigg] \ \bar u_{1,2} \,\bar u_{2,1} $$
 from which we deduce that the eigenvalues are given by $\lambda_1 = 0$,
 $$ \lambda_2 \ = \ \bigg[\frac{a_{2,1} \,(a_{1,1} - a_{2,1})}{D_1^2} + \frac{a_{1,2} \,(a_{2,2} - a_{1,2})}{D_2^2} \bigg] \,\bar u_{1,2} \,\bar u_{2,1} \quad \hbox{and} \quad
    \lambda_3 \ = \ \frac{a_{3,1} \,\bar u_{1,2}}{D_1} + \frac{a_{3,2} \,\bar u_{2,1}}{D_2} \ - \ 1 $$
 The eigenspace associated with the eigenvalue 0 is given by $\Span (e_{1,2})$.
 Since the second eigenvalue only depends on the first $2 \times 2$ block of the matrix, we expect that the associated eigenspace is spanned by $(1, -1, 0)$.
 To prove this assertion with a minimum of calculation, we observe that
 $$ \frac{\partial F_1}{\partial u_1} \,(e_{1,2}) \ + \ \frac{\partial F_2}{\partial u_1} \,(e_{1,2}) \ + \ \frac{\partial F_3}{\partial u_1} \,(e_{1,2}) \ = \
    \frac{\partial F_1}{\partial u_1} \,(e_{1,2}) \ + \ \frac{\partial F_2}{\partial u_1} \,(e_{1,2}) \ \equiv \ 0 $$
 since $F_1 + F_2 + F_3 \equiv 0$.
 It follows that
 $$ \frac{1}{\bar u_{2,1}} \ \frac{\partial F_1}{\partial u_1} \,(e_{1,2}) \ + \ \frac{1}{\bar u_{2,1}} \ \frac{\partial F_2}{\partial u_1} \,(e_{1,2}) \ = \
    \Bigg[\frac{a_{1,1} \,a_{2,1}}{D_1^2} - \frac{a_{1,2}^2}{D_2^2} \Bigg] \ \bar u_{1,2} \ + \
    \Bigg[\frac{a_{2,1}^2}{D_1^2} - \frac{a_{1,2} \,a_{2,2}}{D_2^2} \Bigg] \ \bar u_{2,1} \ = \ 0 $$
 In particular, recalling that $\lambda_2$ is the trace of the first $2 \times 2$ block,
 $$ \begin{array}{l}
    \displaystyle \frac{\partial F_1}{\partial u_1} \,(e_{1,2}) \ - \ \frac{\partial F_1}{\partial u_2} \,(e_{1,2}) \ = \
    \displaystyle \Bigg[\frac{a_{1,1} \,a_{2,1}}{D_1^2} - \frac{a_{1,2}^2}{D_2^2} \Bigg] \ \bar u_{1,2} \vspace{6pt} \\ \hspace{20pt} = \
    \displaystyle \Bigg[\frac{a_{1,1} \,a_{2,1}}{D_1^2} - \frac{a_{1,2}^2}{D_2^2} \Bigg] \ \bar u_{1,2} \,\bar u_{2,1} \ - \
    \displaystyle \Bigg[\frac{a_{2,1}^2}{D_1^2} - \frac{a_{1,2} \,a_{2,2}}{D_2^2} \Bigg] \ \bar u_{2,1} \,\bar u_{1,2} \ = \ \lambda_2 \end{array} $$
 Similarly, we prove that
 $$ \frac{\partial F_2}{\partial u_1} \,(e_{1,2}) \ - \ \frac{\partial F_2}{\partial u_2} \,(e_{1,2}) \ = \
     - \ \frac{\partial F_1}{\partial u_1} \,(e_{1,2}) \ - \ \frac{\partial F_2}{\partial u_2} \,(e_{1,2}) \ = \ - \lambda_2 $$
 therefore, the eigenspace associated with the eigenvalue $\lambda_2$ is spanned by $(1, -1, 0)$.
 The sign of this eigenvalue thus indicates the stability of $e_{1,2}$ in the invariant manifold $u_1 + u_2 = 1$ and as expected we
 find the same conditions as in the case $n = 2$ given in \eqref{eq:conditions-bis}.
 Rather than computing explicitly the third eigenspace, we notice that all three eigenvectors form a basis of the three dimensional Euclidean space.
 This together with the expression of the first two eigenvectors implies that the third coordinate of the eigenvector associated with the eigenvalue
 $\lambda_3$ is different from zero. In particular, the sign of $\lambda_3$ indicates whether $e_{1,2}$ is repelling or not in the direction of $e_3$.
 To study the sign of this eigenvalue, we first observe that, by definition of $e_{1,2}$,
 $$ a_{2,1} \,(a_{1,2} \,\bar u_{1,2} + a_{2,2} \,\bar u_{2,1}) \ = \ a_{1,2} \,(a_{1,1} \,\bar u_{1,2} + a_{2,1} \,\bar u_{2,1}) $$
 which allows to simplify the expression of $\lambda_3$ and obtain
 $$ \begin{array}{rcl}
    \sign (\lambda_3) & = &
    \sign \big[\,2 a_{1,2} \,a_{3,1} \,\bar u_{1,2} + 2 a_{2,1} \,a_{3,2} \,\bar u_{2,1} - a_{1,2} \,(a_{1,1} + a_{2,1}) \,\bar u_{1,2} - a_{2,1} \,(a_{1,2} + a_{2,2}) \,\bar u_{2,1} \,\big] \vspace{6pt} \\ & = &
    \sign \big[\,a_{1,2} \,(2 a_{3,1} - a_{1,1} - a_{2,1}) \,\bar u_{1,2} + a_{2,1} \,(2 a_{3,2} - a_{1,2} - a_{2,2}) \,\bar u_{2,1} \,\big] \vspace{6pt} \\ & = &
    \sign \big[\Delta_{1,2} \,[a_{1,2} \,(a_{2,1} - a_{1,1}) + a_{2,1} \,(a_{1,2} - a_{2,2})] \big]. \end{array} $$
 The result then follows directly from the previous equation and Theorem \ref{th:global-2}.
\end{proof} \\ \\
 Even though the four conditions in Theorem \ref{th:local-nontrivial} are somewhat complicated to understand due to the expression
 of $\Delta_{1,2}$, they can be used to deduce sufficient conditions that are more meaningful from an ecological point of view.
 Since species 3 can invade species 1 and 2 in their equilibrium when species 1 and 2 cooperate and $\Delta_{1,2} > 0$ or when
 species 1 and 2 compete and $\Delta_{1,2} < 0$, we obtain that a sufficient condition for invadability is given by
 $$ a_{3,1} \ > \ \frac{a_{1,1} + a_{2,1}}{2} \qquad \hbox{and} \qquad a_{3,2} \ > \ \frac{a_{1,2} + a_{2,2}}{2} $$
 In other words, species 3 can invade species 1 and 2 in their equilibrium whenever its ability to exploit resource of type 1, respectively 2,
 is larger than the average ability of the two other species to exploit this resource.
 The last theorem, which is partly based on Theorems \ref{th:heteroclinic}-\ref{th:local-nontrivial}, gives general sufficient
 conditions under which strong coexistence of all three species is possible in the sense that the three-type system is permanent.
 For more convenience, we distinguish four cases depending on the number of nontrivial boundary equilibria.
 We only give particular conditions but, as point out above, additional conditions can be deduced by replacing $(1, 2, 3)$ with $(\sigma (1), \sigma (2), \sigma (3))$ where $\sigma$ is any of the
 six permutations of the group $\mathfrak S_3$.

\begin{theor}[Coexistence]
\label{th:permanence}
 Let $\Delta$ denote the set of nontrivial boundary equilibria.
 If conditions in one of the following four cases hold, then the system is permanent.
\begin{enumerate}
 \item[0.] With $\Delta = \varnothing$ it suffices that conditions \eqref{hc} and \eqref{hc_u} hold. \vspace{4pt}
 \item[1.] With $\Delta = \{e_{1,2} \}$ it suffices that $\Delta_{1,2} > 0$ and
\begin{equation}
\label{1s}
 a_{3,1} < a_{1,1} < a_{2,1} \qquad a_{2,2} < \min (a_{3,2}, a_{1,2}) \qquad a_{2,3} < a_{3,3} < a_{1,3}
\end{equation}
 \item[2.] With $\Delta = \{e_{1,2}, e_{2,3} \}$ it suffices that $\min (\Delta_{1,2}, \Delta_{2,3}) > 0$ and
\begin{equation}
\label{2s}
 a_{3,1} < a_{1,1} < a_{2,1} \qquad a_{2,2} < \min (a_{3,2}, a_{1,2}) \qquad a_{3,3} < \min (a_{1,3}, a_{2,3})
\end{equation}
 \item[3.] With $\Delta = \{e_{1,2}, e_{2,3}, e_{3,1} \}$ it suffices that $\min (\Delta_{1,2}, \Delta_{2,3}, \Delta_{3,1}) > 0$ and
\begin{equation}
\label{3s}
 a_{1,1} < \min (a_{2,1}, a_{3,1}) \qquad a_{2,2} < \min (a_{3,2}, a_{1,2}) \qquad a_{3,3} < \min (a_{1,3}, a_{2,3})
\end{equation}
\end{enumerate}
\end{theor}
\begin{proof}
 Permanence of the system under conditions 0 is simply Theorem \ref{th:heteroclinic} that we repeat here to have a summary of
 all the coexistence results.
 Since permanence can be proved similarly under the other three conditions, we only focus on the last case when there are three nontrivial
 boundary equilibria.
 The proof is based on an application of Theorem 2.5 in Hutson \cite{hutson_1984}.
 First, we introduce the function $P (u_1, u_2, u_3) = u_1 \,u_2 \,u_3$ and observe that this function is equal to zero on the boundary
 of $S_3$ and strictly positive in the interior of $S_3$, which makes $P$ an average Lyapunov function of the system \eqref{eq:mean-field-3S}.
 From the expression of the derivatives in \eqref{eq:mean-field-3S}, we then obtain
 $$ \psi \ := \ \frac{1}{P} \ \frac{dP}{dt} \ = \
    \frac{1}{u_1 \,u_2 \,u_3} \ \bigg(\frac{du_1}{dt} \ u_2 \,u_3 + \frac{du_2}{dt} \ u_3 u_1 + \frac{du_3}{dt} \ u_1 u_2 \bigg) \ = \ G_1 + G_2 + G_3. $$
 By Theorem \ref{th:global-2}, the condition \eqref{3s} implies that the $\omega$-limit sets of the boundary of $S_3$ consists of the six
 trivial and nontrivial boundary equilibria.
 Therefore, in order to prove that the system is permanent, it suffices to prove according to Theorem 2.5 in Hutson \cite{hutson_1984} that the
 function $\psi$ evaluated at each of the six boundary equilibria is positive.
 We conclude by observing that
 $$ \begin{array}{rcl}
     (G_1 + G_2 + G_3) (e_i) & = & \displaystyle \frac{a_{j,i} + a_{k,i} - 2 a_{i,i}}{a_{i,i}} \ > \ 0 \vspace{8pt} \\
     (G_1 + G_2 + G_3) (e_{i,j}) & = & \displaystyle
    \frac{(2 a_{k,i} - a_{i,i} - a_{j,i}) (a_{i,j} - a_{j,j}) + (2 a_{k,j} - a_{i,j} - a_{j,j}) (a_{j,i} - a_{i,i})}{a_{i,j} \,a_{j,i} - a_{i,i} \,a_{j,j}} \vspace{4pt} \\ & = &
    \displaystyle \frac{\Delta_{i,j}}{a_{i,j} \,a_{j,i} - a_{i,i} \,a_{j,j}} \ > \ 0 \end{array} $$
 for all $i \neq j$ provided $\Delta_{i,j} > 0$ and the conditions in \eqref{3s} hold.
\end{proof} \\ \\
 Each of the four sets of conditions in Theorem \ref{th:permanence}, which describes the local dynamics at each boundary equilibrium,
 as well as all additional cases that can be deduced from a permutation of the three species imply that the system is permanent.
 However, the theorem does not specify whether there are indeed matrices that satisfy these conditions that might be conflicting.
 To have a complete picture of the coexistence region of the three-type nonspatial model, the last step is to identify four matrices that
 satisfy the conditions of Theorem \ref{th:permanence} and collect all the configurations of local dynamics at the boundary that might
 also produce permanence of the system.
 The four particular scenarios described in Theorem \ref{th:permanence} are illustrated in a schematic manner by the four pictures in the
 first row of Figure \ref{fig:coexistence}.
 The second row gives the solution curves of the system when the dynamics are described by each of the following four matrices, and it is
 straightforward to check that these matrices indeed satisfy each of the four sets of conditions of the theorem.
\begin{figure}[t!]
\centering
 \includegraphics[width = 440pt]{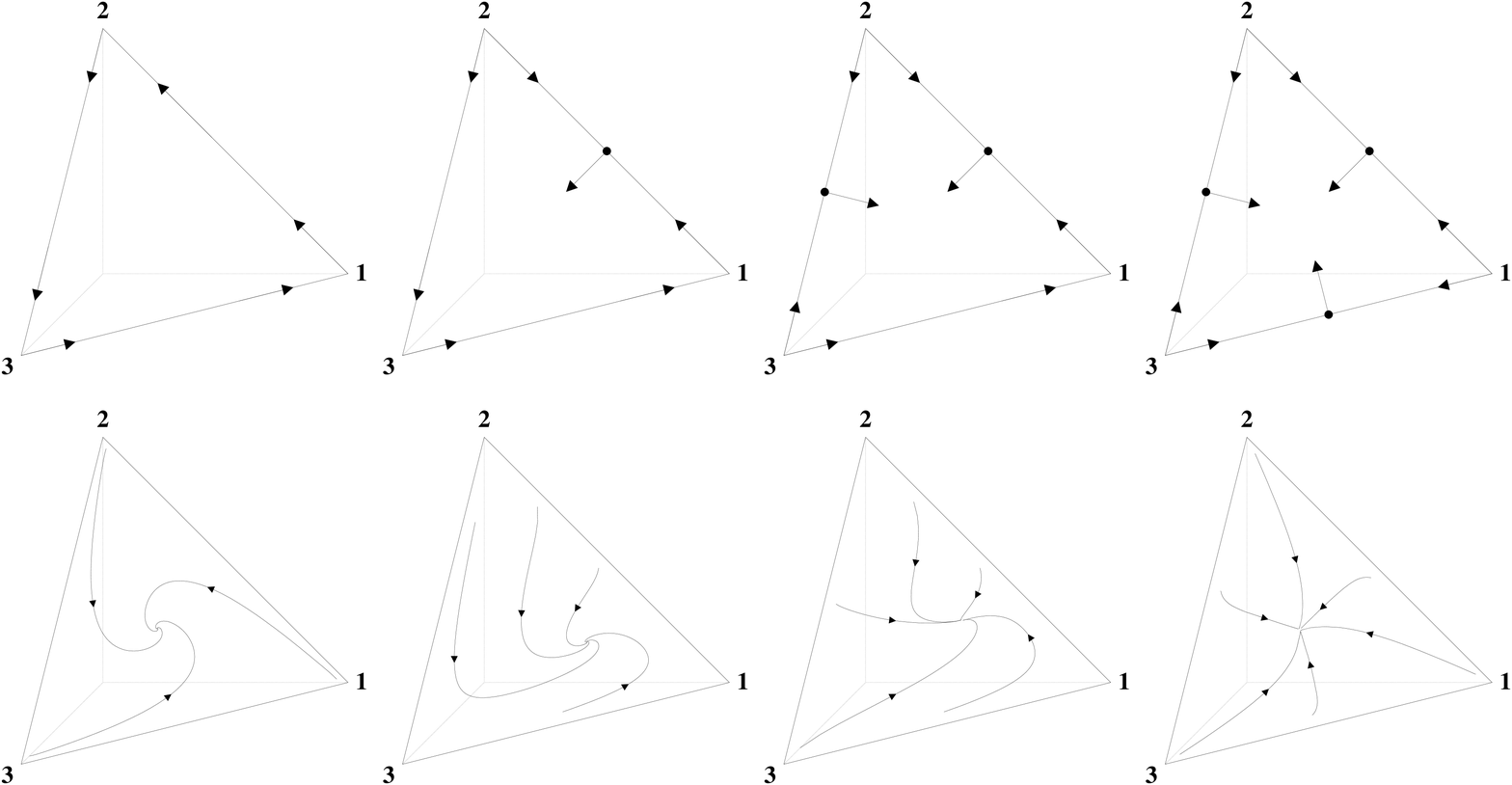}
\caption{\upshape The four possible scenarios.}
\label{fig:coexistence}
\end{figure}
 $$ M_0 = \left(\hspace{-3pt} \begin{array}{ccc} 1 & 0 & 4 \\ 4 & 1 & 0 \\ 0 & 4 & 1 \end{array} \hspace{-3pt} \right) \qquad
    M_1 = \left(\hspace{-3pt} \begin{array}{ccc} 1 & 1 & 2 \\ 2 & 0 & 0 \\ 0 & 8 & 1 \end{array} \hspace{-3pt} \right) \qquad
    M_2 = \left(\hspace{-3pt} \begin{array}{ccc} 1 & 1 & 2 \\ 2 & 0 & 1 \\ 0 & 4 & 0 \end{array} \hspace{-3pt} \right) \qquad
    M_3 = \left(\hspace{-3pt} \begin{array}{ccc} 1 & 2 & 2 \\ 2 & 1 & 2 \\ 2 & 2 & 1 \end{array} \hspace{-3pt} \right) $$
 To look for all the possible local stability at the boundary equilibria that might produce permanence of the system, we first observe that,
 whenever two species, say species 1 and 2, cooperate with each other but are defectors with respect to species 3, we have
 $$ a_{3,1} < a_{1,1} < a_{2,1} \qquad \hbox{and} \qquad a_{3,2} < a_{2,2} < a_{1,2} $$
 which directly implies that $\Delta_{1,2} < 0$.
 Therefore, species 3 cannot invade species 1 and 2 in their equilibrium.
 Similarly, when species 1 and 2 compete with each other but are cooperators with respect to species 3, we always have the condition
 $\Delta_{1,2} < 0$.
 This implies that the first two pictures of Figure \ref{fig:impossible} that describe the local stability of the two trivial boundary
 equilibria $e_1$ and $e_2$, and the nontrivial boundary equilibrium $e_{1,2}$ cannot occur.
 Using the two previous observations and the fact that it is necessary in order to have permanence of the system that none of the
 boundary equilibria is locally stable, simple graphical methods lead to only five possible scenarios, plus again the ones that can
 be deduced from a permutation of the three species.
 These five scenarios are the ones depicted in the first row of Figure \ref{fig:coexistence} and in the last picture
 of Figure \ref{fig:impossible} but we claim that the boundary dynamics shown in this last picture are conflicting.
 To establish this result, we observe that the local stability at the three trivial boundary equilibria leads to
\begin{equation}
\label{eq:trivial}
 a_{2,1} < a_{1,1} < a_{3,1} \qquad  a_{1,2} < a_{2,2} < a_{3,2} \qquad a_{3,3} < \min (a_{1,3}, a_{2,3}).
\end{equation}
 To check whether $e_{2,3}$ can be a saddle node, we compute
 $$ \begin{array}{rcl} \Delta_{2,3} & = &
    (2 a_{1,2} - a_{2,2} - a_{3,2}) (a_{2,3} - a_{3,3}) + (2 a_{1,3} - a_{2,3} - a_{3,3}) (a_{3,2} - a_{2,2}) \vspace{4pt} \\ & = &
    (2 a_{1,2} - a_{2,2} - a_{3,2}) (a_{2,3} - a_{3,3}) + (2 a_{1,3} - 2 a_{2,3} + a_{2,3} - a_{3,3}) (a_{3,2} - a_{2,2}) \vspace{4pt} \\ & = &
     2 \,(a_{2,3} - a_{3,3}) (a_{1,2} - a_{2,2}) + 2 \,(a_{3,2} - a_{2,2}) (a_{1,3} - a_{2,3}). \end{array} $$
 In particular, under the conditions \eqref{eq:trivial}, $e_{2,3}$ is a saddle node if and only if $\Delta_{2,3} > 0$ which,
 in view of the signs of the first three factors, requires that $a_{1,3} > a_{2,3}$.
 We prove similarly that, for the nontrivial boundary equilibrium $e_{3,1}$ to be a saddle node, it is necessary that
 $a_{1,3} < a_{2,3}$, therefore the two equilibria cannot both be saddle nodes.
 In conclusion, not only the four pictures in the first row of Figure \ref{fig:coexistence} are possible but also they are the
 only ones that lead to coexistence.
\begin{figure}[t!]
\centering
 \includegraphics[width = 440pt]{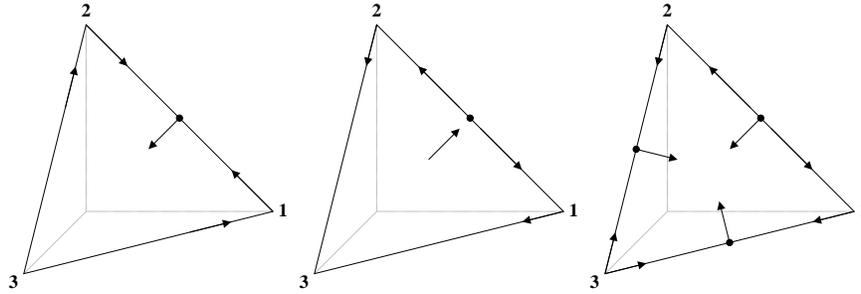}
\caption{\upshape Impossible pictures.}
\label{fig:impossible}
\end{figure}

% % % % % % % % % % % % % % % % % % % % % % % % % % % % % % % % % % % % % % % % % % % % % % % % % % % % % % % % % % % % % % % % % % % % % 

\subsection*{Summary of the results.}
 We conclude this section with a brief overview of the analytical results collected for the nonspatial two-type and three-type models.

\indent \emph{The two-type model.}
 In the presence of only two species, a defector always outcompetes a cooperator.
 More generally, regardless of the size of the community, a species that has a higher ability than all other species to exploit resources
 of either type outcompetes the other species, due to the fact that it always has the highest fitness \emph{regardless} of the global
 species distribution.
 When both species compete in a defector-defector relationship then the system is bistable indicating that coexistence is
 not possible and that the outcome of the competition strongly depends on the initial conditions.
 In contrast, two cooperators always coexist.
 Motivated by the properties of the two-type model, our analysis of the three-type model focuses on the two extreme cases that also
 appear later to be the most interesting ones: tristability and coexistence.

\indent \emph{Tristability.}
 In communities involving three species, each of the three trivial boundary equilibria is locally stable if and only if all three pairs
 of species are in defector-defector relationship.
 Numerical simulations of the nonspatial model suggest that, as in the case of bistability for the two-type model, for almost all
 initial conditions, one species outcompetes the other two ones.

\indent \emph{Coexistence and cooperation.}
 In contrast, if none of the three pairs of species is in a competition relationship, coexistence is possible.
 In particular, as in the two-type model, coexistence occurs in situations in which any two species cooperate.
 However, we point out that, whereas in the two-type model cooperation implies coexistence, global cooperation alone does not
 imply coexistence of three species since additional invadability conditions are required.
 Even though we omit the details of the proof here, it can be easily seen that if the cooperation between two species is significantly
 stronger than the cooperation between any of these two species and the third one, then one of the nontrivial boundary equilibrium
 becomes locally stable:
 indeed, it suffices to take the matrix with zeros on the diagonal, with $a_{1,2} = a_{2,1} > 2$, and with ones elsewhere.
 Numerical simulations further indicate that the two strong cooperators outcompete the third species.
 Hence, that any two species coexist in the absence of the third one does not imply coexistence of all three species.

\indent \emph{Coexistence of cheaters.}
 One of the most interesting aspects of the nonspatial model is that, whereas in the presence of two species cooperation is the only
 mechanism that promotes coexistence, in the presence of three species there are other such mechanisms.
 In particular, three cheaters can coexist.
 This happens when each species is a defector with respect to a second species but a cooperator with respect to the third species,
 a situation that we shall call \emph{rock-paper-scissors} type relationship to draw another analogy with game theory.
 When scissors are minoritary, paper beats rock and increases its density until the environment becomes suitable enough for scissors
 to expand again, and so on.
 In particular, that any two species cannot coexist does not imply that global coexistence is not possible.
 In rock-paper-scissors relationships, there is always a weak form of coexistence in the sense that none of the species is driven
 to extinction.
 However, the densities can be alternately arbitrarily low and, as in situations in which species cooperate, strong coexistence
 in the form of permanence of the system requires additional conditions.

\indent \emph{Coexistence and competition.}
 Finally, assuming that two species compete leads to the existence of at least one locally stable boundary equilibrium.
 In other words, whenever two species compete, all three species cannot coexist in the sense that the system is not permanent.
 There are only four boundary dynamics, along with the ones deduced from a permutation of the three species, that lead to
 strong coexistence: the four dynamics depicted in the first row of Figure \ref{fig:coexistence}.
 Therefore, coexistence can only occur when all pairs of species are either in a defector-cooperator relationship or
 cooperator-cooperator relationship, and the coexistence region covers all the cases where the number of cooperative pairs
 ranges from zero to three: regardless of the number of nontrivial boundary equilibria, permanence of the system is possible.
 In conclusion, coexistence occurs either in the presence of a global cooperation, or in the presence of a rock-paper-scissors
 type dynamics, or dynamics that consist of a mixture of these two extreme cases.

%%%%%%%%%%%%%%%%%%%%%%%%%%%%%%%%%%%%%%%%%%%%%%%%%%%%%%%%%%%%%%%%%%%%%%%%%%%%%%%%%%%%%%%%%%%%%%%%%%%%%%%%%%%%%%%%%%%%%%%%%%%%%%%%%%%%%%%%%%

\section{Analytical results for the spatial stochastic model}
\label{sec:stochastic}

\indent This section collects important analytical results for the two-type spatial model.
 We focus on the meaning of these results and only give an intuitive idea of the mathematical proofs.
 For a rigorous analysis of the two-type spatial model, we refer the reader to Lanchier \cite{lanchier_2010}.
 As previously, we explore the phase diagram in the $\theta_1$-$\theta_2$ plane where $\theta_i$ defined in \eqref{eq:theta}
 measures the relative ability of species $i$ to exploit the resource it produces, therefore species $i$ is a cooperator or
 a defector depending on whether the parameter $\theta_i$ is smaller or larger than one half.

% % % % % % % % % % % % % % % % % % % % % % % % % % % % % % % % % % % % % % % % % % % % % % % % % % % % % % % % % % % % % % % % % % % % % 

\subsection*{The one-dimensional two-type process.}
 To understand the one-dimensional model, we first initiate the process from the deterministic configuration in which
 all sites to the left of site 0, including site 0, are of type 1, and all the other sites are of type 2.
 The evolution rules imply that, at the times the rightmost site of type 1, respectively, leftmost site of type 2, is updated,
 the probability that the site to be updated switches to the other type is given by
 $$ \frac{a_{2,1}}{a_{1,1} + a_{2,1}} \ = \ 1 - \theta_1 \qquad \hbox{and} \qquad \frac{a_{1,2}}{a_{1,2} + a_{2,2}} \ = \ 1 - \theta_2 $$
 respectively.
 Since in addition all lattice sites are updated at the same rate, we deduce that the interface between both types drifts to the
 left or to the right depending on whether $\theta_1$ is smaller or larger than $\theta_2$, respectively.
 Therefore, when $\theta_1 < \theta_2$ species 2 wins, which is defined for the stochastic process as the fact that any lattice
 site is eventually of type 2, i.e.,
\begin{equation}
\label{eq:invasion}
 \lim_{t \to \infty} \ P \,(\eta_t (x) = 2) \ = \ 1 \quad \hbox{for all} \ x \in \Z.
\end{equation}
 Similarly, when $\theta_1 > \theta_2$ species 1 wins.
 Standard techniques further imply that the same conclusion holds when starting from a configuration with infinitely many
 individuals of both types, which includes all configurations in which the types at different sites are independent and identically
 distributed and in which both types occur with positive probability.

\indent To understand the neutral case $\theta_1 = \theta_2$, we initiate the stochastic process from the configuration in which sites
 are independently of type 1 and of type 2 with the same probability.
 We note that, excluding the case $\theta_1 = \theta_2 = 1$ in which the process is static, the set of interfaces, defined as the random
 set of points in the dual lattice $\Z + \frac{1}{2}$ that lie between two different types, evolves almost like a system of independent
 annihilating symmetric random walks.
 More precisely, each interface jumps one unit to the left or one unit to the right with equal probability except when two interfaces
 are distance one apart.
 Moreover, when an interface jumps to another interface, both interfaces annihilate.
 Using that symmetric random walks are recurrent in one dimension, it is straightforward to conclude that the set of interfaces
 goes extinct.
 This translates into a clustering of the process, which is defined as the fact that the probability that two lattice sites have
 different types tends to 0 as time goes to infinity, i.e.,
\begin{equation}
\label{eq:cluster}
 \lim_{t \to \infty} \ P \,(\eta_t (x) \neq \eta_t (y)) \ = \ 0 \quad \hbox{for all} \ x, y \in \Z.
\end{equation}
 The behavior of the one-dimensional process can thus be summarized as follows.
\begin{theor}
\label{th:spatial-1D}
 Assume that $d = 1$. Then,
\begin{enumerate}
 \item Excluding the static case $\theta_1 = \theta_2 = 1$, the process clusters. \vspace{2pt}
 \item Assume that $\theta_1 \neq \theta_2$. Then type $i$ with $\theta_i = \max (\theta_1, \theta_2)$ wins.
\end{enumerate}
\end{theor}
 We refer the reader to \cite{lanchier_2010}, Section 2, for a rigorous proof of these two statements.
 In contrast with the predictions based on the mean-field approximation, Theorem \ref{th:spatial-1D} indicates that, except in the
 trivial case when the process is static, coexistence is never possible in one dimension even when both species cooperate strongly
 by providing resources that increase the fitness of the other species.
 Moreover, while the mean-field model exhibits bistability when both species are defectors, the spatial model always has
 a ``dominant'' type when $\theta_1 \neq \theta_2$, that is a type that wins even when starting at very low density.
 This dominant species is always the least cooperative one.
 The most interesting behavior appears in the neutral case when the set of interfaces is roughly a collection of symmetric
 annihilating random walks.
 Results conjectured by Erd\H os and Ney \cite{erdos_ney_1974} and proved later by Schwartz \cite{schwartz_1978} show that such
 a system is site recurrent, meaning that, even though the set of interfaces goes extinct due to recurrent annihilations, each site
 of the dual lattice is visited infinitely often by an interface.
 This indicates that the type at each site alternates infinitely often, thus giving the impression of coexistence.
 However, coexistence is not, strictly speaking, possible in the sense that there is no equilibrium in which
 both types are present:
 except in the static case, there are only two stationary distributions, namely the ones that correspond to the configurations in
 which all sites are of type 1 or all sites are of type 2.

% % % % % % % % % % % % % % % % % % % % % % % % % % % % % % % % % % % % % % % % % % % % % % % % % % % % % % % % % % % % % % % % % % % % % 

\subsection*{The two-dimensional two-type process.}
 The analysis of the two-dimensional process is more difficult since, due to the geometry of the spatial structure, it cannot
 be reduced to a simple analysis of the interfaces.
 The idea is to look at key particular cases and then apply standard probabilistic techniques such as coupling argument, block
 construction and perturbation argument to obtain additional information about the process.
 Further insight for the two-dimensional process will be gained through numerical simulations, however, we point out that our
 analytical results easily extend to any dimension $d > 1$.
 Precisely, we focus on the two-dimensional two-type stochastic model when the parameters are given by each of the following
 four matrices
 $$ M_4 = \left(\hspace{-3pt} \begin{array}{cc} 1 - \ep & 0 \\ \ep & 1 \end{array} \hspace{-3pt} \right) \qquad
    M_5 = \left(\hspace{-3pt} \begin{array}{cc} 1 - \ep & 1 - \ep \\ 1 + \ep & 1 + \ep \end{array} \hspace{-3pt} \right) \qquad
    M_6 = \left(\hspace{-3pt} \begin{array}{cc} 0 & 1 \\ 1 & 0 \end{array} \hspace{-3pt} \right) \qquad
    M_7 = \left(\hspace{-3pt} \begin{array}{cc} 0 & 1 \\ 1 & 1 \end{array} \hspace{-3pt} \right) $$
 where the parameter $\ep$ will be typically chosen small and positive.
 These four matrices represent key points of the phase diagram from which we can deduce, using either monotonicity, symmetry
 or perturbation arguments, the behavior of the process in larger parameter regions.

\indent First, we assume that the local interactions are described by the matrix $M_4$.
 This case can be seen as a case where species 2 is an extreme defector as the second column indicates that this species
 produces resources that are exploited by its own species only, hence sites of type 2 never change their type.
 In contrast, each individual of type 1 has a positive probability to change its type at each update provided it has at
 least one neighbor of type 2.
 Thinking of site of type 2 as occupied and sites of type 1 as empty, the process becomes a pure birth process so it is obvious
 that type 2 wins.
 Standard arguments allow to extend the result to a certain topological neighborhood of the parameter matrix $M_4$ therefore we
 obtain the following result.
\begin{theor}
\label{th:spatial-pure-birth}
 For all $\theta_1 < 1$ there exists $\theta < 1$ such that type 2 wins whenever $\theta \leq \theta_2 \leq 1$.
\end{theor}
 We refer the reader to \cite{lanchier_2010}, Section 4, for more details about the proof.
 By symmetry, we also find a parameter region in which individuals of species 1 outcompete individuals of species 2.
 Note that Theorem \ref{th:spatial-pure-birth} holds even when species 2 starts at a very low density, which shows a new
 disagreement between spatial and nonspatial models: similarly to the one-dimensional process, there is a parameter region in
 which species 2 is the dominant species for the spatial model whereas the mean-field model is bistable.

\indent Assuming now that the dynamics are dictated by the matrix $M_5$ with $\ep = 0$, the resources are equally shared
 by both types so the new type at each update is simply chosen uniformly at random from the neighborhood.
 This process is known as the voter model introduced in \cite{clifford_sudbury_1973, holley_liggett_1975}.
 There exists a duality relationship between the voter model and coalescing random walks that allows to prove that clustering
 occurs in the sense of \eqref{eq:cluster} in one and two dimensions whereas there exists a stationary distribution in
 which both types coexist in higher dimensions.
 Taking $\ep > 0$ in the second matrix, a similar duality relationship between the process and a certain system of random walks
 allows to show that type 2 wins.
 This result is expected since in this case species 2 is a cheater: it has a higher ability than species 1 to exploit resources
 of either type.
 Since this holds for arbitrarily small parameter $\ep > 0$ and the probability that species 2 outcompetes species 1 is nonincreasing
 with respect to $\theta_1$ and nondecreasing with respect to $\theta_2$, we obtain the following result.
\begin{theor}
\label{th:spatial-invasion}
 Assume that $\theta_1 < \frac{1}{2} < \theta_2$. Then, in any dimension, type 2 wins.
\end{theor}
 We refer to Section 3 in \cite{lanchier_2010} for a rigorous proof of this theorem.
 By symmetry, an analogous result can be deduced by exchanging the role of species 1 and 2.
 In particular, Theorem \ref{th:spatial-invasion} shows that, in any spatial dimension, if one species is a cooperator and the
 other species a defector, which results in the presence of a cheater, then the cheater is always the dominant species.
 This agrees with the predictions based on the mean-field model.

\indent The analytical results obtained when the interactions are described by the last two matrices are more interesting
 from a biological perspective.
 First, we observe that, when the parameters are given by the matrix $M_6$, individuals of either type change their
 type at each update whenever there is at least one individual of the other type in their neighborhood which, in contrast with
 the voter model, leads to dynamics that somewhat favor the local minority.
 This corresponds to a case of extreme cooperation in which each type only provides resources to the other type.
 This process is known as the threshold voter model for which Liggett \cite{liggett_1994} has proved that coexistence occurs
 in any dimension $d > 1$.
 Relying on perturbation techniques, coexistence can be extended to a certain topological neighborhood of the matrix $M_6$.
\begin{theor}
\label{th:spatial-coexistence}
 If $d > 1$ then there is $\ep > 0$ such that coexistence occurs when $\max (\theta_1, \theta_2) < \ep$.
\end{theor}
 See Section 5 in \cite{lanchier_2010} for a proof of this theorem.
 Note that, while  Theorem \ref{th:spatial-coexistence} shows that increasing the spatial dimension allows two species to coexist,
 it also requires the cooperation to be strong enough, whereas results for the mean-field model indicate more generally that
 cooperation always leads to coexistence.

\indent Interestingly, our next result shows that the inclusion of local interactions indeed translates into a reduction of the
 coexistence region.
 Before stating this result, we first observe that, when the interactions are dictated by the matrix $M_7$, the process consists
 of a mixture of a voter model and a threshold voter model:
 at each update, type 1 is replaced by type 2 whenever it has at least one type 2 in its neighborhood whereas type 2 is simply
 replaced by a type chosen uniformly at random from its neighborhood.
 In this case, there is again a duality relationship between the process and a certain system of random walks whose analysis
 reveals that type 2 wins.
 Moreover, the techniques of the proof allow for the application of a perturbation argument, so the same conclusion holds in a
 certain topological neighborhood of the matrix $M_7$ which induces the existence of a parameter region for which coexistence
 occurs for the mean-field model but not the spatial model.
\begin{theor}
\label{th:spatial-reduction}
 There is $\ep > 0$ such that type 2 wins whenever $\theta_1 < \ep$ and $\theta_2 > \frac{1}{2} - \ep$.
\end{theor}
 The proof of this theorem can be found in \cite{lanchier_2010}, Section 6.
 This result indicates that even though two species cooperate, they cannot coexist in space if their relationship has too much
 asymmetry, that is when one species cooperates much more than the other.
 In such situation, the dominant species is again the least cooperative one.

%%%%%%%%%%%%%%%%%%%%%%%%%%%%%%%%%%%%%%%%%%%%%%%%%%%%%%%%%%%%%%%%%%%%%%%%%%%%%%%%%%%%%%%%%%%%%%%%%%%%%%%%%%%%%%%%%%%%%%%%%%%%%%%%%%%%%%%%%%

\section{The role of space}
\label{sec:comparison}

\indent In order to understand the role of local interactions (and stochasticity), we confront the analytical results obtained
 for the nonspatial and spatial models as well as additional numerical results provided in this section for the spatial model.
 Since numerical simulations of stochastic spatial models are somewhat difficult to interpret, we start by devoting few
 lines to mention important theoretical results that justify the approach we have followed to draw general conclusions from
 particular simulation results.
 Numerical and analytical results are used to first obtain a complete picture of the differences between both models when
 only two species are present.
 We note however that a complete understanding of the stochastic three-type model based on simulations is out of reach due to
 a too large number of parameters.
 Our approach is to use the results obtained for the nonspatial model and the spatial two-type model as a guide to
 find certain three-dimensional submanifolds of the set of parameters along which interesting behaviors emerge.
 Numerical simulations along these manifolds are then performed in order to obtain three-dimensional phase diagrams of the
 models when the set of parameters is restricted to particular subsets.

% % % % % % % % % % % % % % % % % % % % % % % % % % % % % % % % % % % % % % % % % % % % % % % % % % % % % % % % % % % % % % % % % % % % % 

\subsection*{Numerical simulations.}
 In the numerical simulations of the spatial model, the population evolves on a $400 \times 400$ lattice with periodic
 boundary conditions.
 That is, sites on the bottom row are neighbors of those on the top row, while sites on the left edge are neighbors of those
 on the right edge.
 In addition, the process starts from a uniform product measure, meaning that the types at different sites are initially
 independent and of either type with probability $1/n$ where $n$ denotes the total number of species, so sites are equally likely
 to be of either type.
 The process always halts after exactly $3.2 \times 10^8$ updates which corresponds to approximately 2000 units of time, which
 is simply obtained by dividing the number of updates by the number of sites and using the fact that each site is updated in
 average once per unit of time. 

\indent The first difficulty in interpreting correctly these spatial simulations is that the stochastic process on a finite
 connected graph always converges to a configuration in which all sites have the same type, therefore coexistence is not possible
 on a $400 \times 400$ lattice, whereas our analytical results show that two types can coexist on the infinite lattice.
 However, rigorous research in the topic of interacting particle systems indicates that, in the presence of a dominant type, this
 type typically invades space linearly in all the directions.
 In contrast, analytical results notably about the voter model on the torus \cite{cox_1989} show that, in the absence
 of a dominant type, the time to fixation of the process on large but finite graphs is sufficiently large that numerical simulations
 indeed reflect a transient behavior of the finite system that is symptomatic of the long-term behavior of its infinite analog.
 In particular, it is accurate to define a dominant type as a type which is able to invade our finite lattice in less
 than 2000 units of time while starting from a low density.

\indent The second difficulty is that the absence of a dominant type as defined above, namely the simultaneous survival of at
 least two types up to time 2000, is not always symptomatic of coexistence for the infinite analog.
 The reason is that, as proved in \cite{cox_1989} for the voter model, in case the infinite system clusters, the time to
 fixation of its finite analog is excessively large, which gives the impression of coexistence, whereas different types
 cannot coexist at equilibrium even for the infinite system.
 However, spatial correlations emerge quickly enough so that properties of the configuration of the finite system at
 time 2000 allow to speculate about the infinite system.
 More precisely, in the absence of a dominant type, the dichotomy between clustering and coexistence will be based on a
 quantity that we shall call \emph{clustering coefficient} and that we define as the percentage of edges that connect sites
 of the same type.
 Note that the clustering coefficient also gives a good approximation of the probability that two nearest neighbors are
 of the same type.
 To distinguish between clustering and coexistence, we invoke a result proved in \cite{cox_griffeath_1986} about the
 diffusive clustering of the two-dimensional voter model, which roughly indicates that the clustering of the process is
 extremely slow and motivates the use of the voter model as a test model.
 Numerical simulations of the voter model, which is obtained from our general model by choosing a matrix in which all the
 coefficients are equal, give the values 86\% and 81\% for the clustering coefficient of the two-type voter model and
 three-type voter model, respectively, at time 2000.
 Therefore, in the absence of a dominant type, we will assume that the infinite system clusters if the clustering
 coefficient of its finite analog is larger or equal to 86\% in the presence of two types and larger or equal to 81\%
 in the presence of three types, and that the process coexists otherwise.
 In conclusion, the long-term behavior of the infinite system is related to the clustering coefficient of its finite
 analog as follows:
\begin{enumerate}
 \item \emph{coexistence} occurs when the clustering coefficient of the finite system is smaller than that of the
  corresponding voter model, that is either 86\% or 81\%, \vspace{2pt}
 \item \emph{clustering} occurs when the clustering coefficient of the finite system is simultaneously larger than that of
  the corresponding voter model and different from 100\%, \vspace{2pt}
 \item existence of a \emph{dominant type} occurs when the clustering coefficient is equal to 100\%,
\end{enumerate}
 and we again point out that past research in the field of interacting particle systems allows to speculate with some
 confidence about the long-term behavior of the infinite system based on numerical simulations and the previous classification.

% % % % % % % % % % % % % % % % % % % % % % % % % % % % % % % % % % % % % % % % % % % % % % % % % % % % % % % % % % % % % % % % % % % % % 

\subsection*{The two-type models.}
 In the presence of two species, the dynamics only depends on two parameters, namely, the relative abilities $\theta_1$ and $\theta_2$
 of species 1 and 2, respectively, to exploit resources of their type.
 Therefore, the long term-behavior of the spatial and nonspatial models can be simply summarized in a two-dimensional phase diagram.
 The phase diagrams of the spatial and nonspatial models are depicted on the left-hand side and right-hand of Figure \ref{fig:two-type}.
 The phase diagram of the nonspatial model is completely understood analytically whereas the one of the spatial model is obtained from
 a combination of analytical and numerical results.
 Labels on the parameter regions surrounded with dashed lines refer to the theorems of Section \ref{sec:stochastic}.
 In both diagrams, the different phases are delineated with thick continuous lines, and we distinguish three cases.

\indent \emph{Presence of a cheater.}
 In defector-cooperator relationships, in which case we call the defector a cheater, analytical
 results for both the spatial and the nonspatial models indicate that the cheater outcompetes the other
 species (see Theorems \ref{th:global-2} and \ref{th:spatial-invasion}).
 The intuition behind this result is that, since the cheater has a better ability to exploit resources of either type, its fitness is
 always higher than the fitness of the other species \emph{regardless} of the configuration of species.
 Therefore, even when starting at very low density, the cheater always wins.

\indent \emph{Space and competition.}
 In the presence of competition, which we defined as the relationship between two defectors, analytical results for the
 nonspatial deterministic model indicate that the system is bistable (see Theorem \ref{th:global-2}).
 That is, the limiting behavior depends on the combination of both the parameter values and the initial densities.
 The parameters being fixed, provided the initial density of either species exceeds some critical threshold, this species
 outcompetes the other species.
 In contrast, analytical results for the spatial model (see Theorem \ref{th:spatial-pure-birth}) reveal the existence of
 parameter regions in which there is a dominant type that outcompetes the other species regardless of the initial configuration
 provided it starts with a positive density.
 Numerical results for the spatial model further indicate that when $\theta_1 \neq \theta_2$ there is a dominant type which is
 uniquely determined by the parameters of the system: the dominant type is always the least cooperative species (see the upper
 panel of Table \ref{tab:two-type}).
 Additional simulations have been performed to check that this type indeed outcompetes the other species even when it initially
 occupies only 5\% of the lattice sites.
 In the neutral case, the system clusters.
 The spatial correlations become stronger and stronger and the boundaries of the clusters sharper and sharper as the common value
 of the relative ability of both species to exploit their own resources increases.
 Observing in addition that under neutrality the number of lattice sites of a given type in the finite system is a Martingale,
 it can be proved that the probability that this type wins is simply equal to its initial density.

\begin{table}[t!]
  \centering
  \scalebox{0.32}{\input{tab-two-type.pstex_t}} \vspace{5pt}
  \caption{}
\label{tab:two-type}
\end{table}

\indent We now give a heuristic argument of the reason why competition translates into bistability in the absence of space whereas,
 excluding the neutral case, the dominant type is uniquely determined by the parameters of the system in the presence of a spatial
 structure.
 We first observe that if a species has a very high ability to exploit the resources it produces but a poor ability to exploit the
 other resources then, in the absence of space, the fitness of this species is roughly proportional to its density.
 Therefore, in the presence of competition, species at low density have a negative growth rate and species at high density a
 positive growth rate, which implies bistability.
 Including space in the form of local interactions drastically modifies the long-term behavior of the system due to the fact that
 individuals can only see their nearest neighbors rather than the whole system.
 More precisely, since an individual of either type can only appear next to a site already occupied by this type (the parent site),
 spatial correlations build up in such a way that the density of individuals of a given type seen by an individual of the same type
 is in average significantly larger in the interacting particle system than in the mean-field approximation.
 Regardless of the global densities, this fraction of homologous neighbors is roughly the same for both species at the interface
 between adjacent clusters.
 Since in addition the resources produced at a site are only available for the nearest neighbors, the global densities become
 irrelevant in determining the fitness of each individual and interfaces expand linearly in favor of the least cooperative species:
 regardless of its initial density, this species is the dominant type.

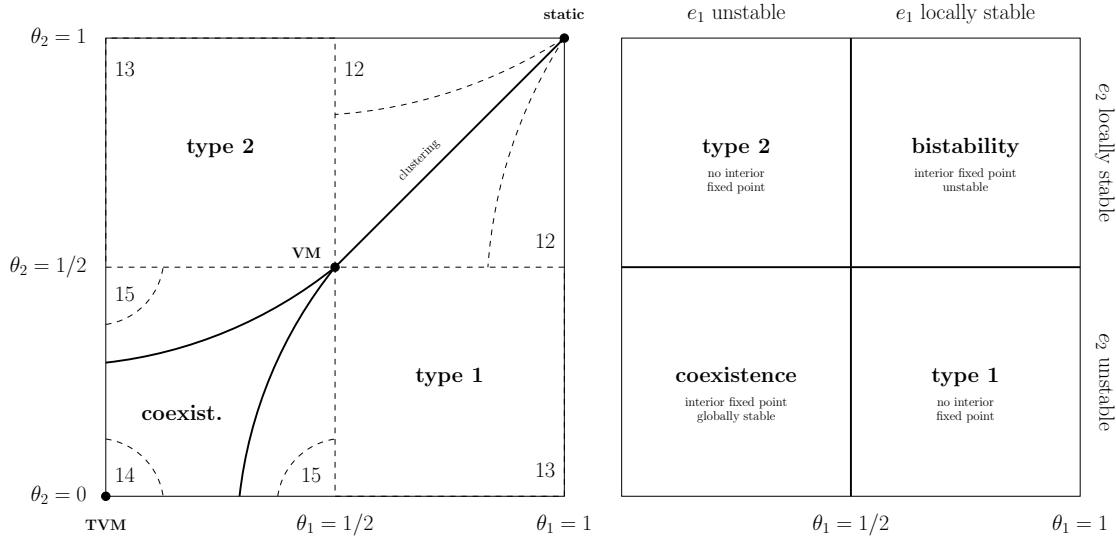
\begin{figure}[t!]
  \centering
  \scalebox{0.40}{\input{diagrams-two-type.pstex_t}}
  \caption{\upshape Phase diagrams of the spatial and nonspatial two-type models.}
\label{fig:two-type}
\end{figure}

\indent \emph{Space and cooperation.}
 In the presence of cooperation, analytical results for the nonspatial model predict coexistence of both
 species (see Theorem \ref{th:global-2}).
 In contrast, analytical results for the spatial model indicate that coexistence is never possible in one
 dimension (see Theorem \ref{th:spatial-1D}).
 In higher dimensions, coexistence becomes possible although there are regions in which coexistence
 occurs for the nonspatial model but not for the spatial model (see Theorems \ref{th:spatial-coexistence} and \ref{th:spatial-reduction}).
 Numerical simulations of the two-dimensional finite system suggest that two cooperative species coexist in the neutral
 case so the coexistence region of the spatial model stretches up to the point that corresponds to voter model dynamics.
 As previously mentioned, analytical results for the voter model show that this critical point is included in the coexistence region if and
 only if the spatial dimension is strictly larger than two.
 There is however a wide range of parameters for which coexistence is not possible though both species cooperate (see the
 lower panel of Table \ref{tab:two-type}).
 Therefore, the inclusion of space translates into a reduction of the coexistence region in any spatial dimension.
 This effect is more pronounced in low dimensions and we strongly believe that the coexistence region of the spatial model approaches the
 one of the nonspatial model as the dimension tends to infinity.

\indent The intuition behind the fact that the inclusion of a spatial structure reduces the ability of species to coexist follows
 the lines of the heuristic argument introduced above.
 In cooperative mean-field systems, a species at very low density turns out to be in a very favorable environment due to the high
 density of the other species and so a large amount of resources to exploit:
 the smaller the density of a species, the larger the fitness of this species, which is a basic mechanism that leads to coexistence
 of both types.
 In contrast, the inclusion of local interactions translates as previously into the presence of spatial correlations that make
 the density of individuals of a given type seen by an individual of the other type significantly smaller in the interacting
 particle system than in the mean-field approximation.
 Therefore, in the presence of a spatial structure, cooperative species cannot fully benefit from the resources produced by the
 other species simply because these resources are geographically out of reach.
 Under neutrality, both species suffer local interactions equally so coexistence is possible, but when the parameters get closer
 to the region where one species is a cheater, individuals of the other species are unable to expand whenever they have at least
 one neighbor of their own type: large clusters of this species tend to shrink quickly while small clusters of this species
 dissolve eventually due to stochasticity.
 Note also that the number of neighbors is smaller, and so the fraction of neighbors of the same type larger, in low dimensions
 hence the harmful effect of space on coexistence is more pronounced in low dimensions.

% % % % % % % % % % % % % % % % % % % % % % % % % % % % % % % % % % % % % % % % % % % % % % % % % % % % % % % % % % % % % % % % % % % % % 

\subsection*{The three-type models.}
 The analysis of the mean-field approximation reveals that, in the presence of three species, the nonspatial model exhibits
 a wide variety of regimes.
 However, analytical and numerical results for the spatial two-type model indicate that the most interesting behaviors in terms
 of the inclusion of space emerge in parameter regions for which there is either bistability or coexistence for the mean-field
 approximation.
 Therefore, we mainly focus on the analogous cases for the three-type models by investigating the effect of space in regions
 for which there is either tristability or coexistence of all three species when space is absent.
 We note that, whereas in the presence of two species cooperation is a necessary and sufficient condition for coexistence in
 a nonspatial universe, in the presence of three species, there are additional strategies that promote global coexistence.
 For instance, coexistence occurs for what we shall call rock-paper-scissors dynamics in which any two species cannot coexist
 but all three species can.
 Interestingly, numerical simulations of the spatial three-type model reveal that the inclusion of space does not have the same
 effect on coexistence depending on the strategies species follow, suggesting that some of the strategies that promote
 coexistence are much more stable than others.

\indent In the presence of two species, each species is either a defector or a cooperator, which results in only three types of
 relationship: cheating, competition, and cooperation.
 Introducing an additional species increases significantly the number of types of relationships since a species can be a defector
 with respect to a second species but a cooperator with respect to the third one as in rock-paper-scissors dynamics.
 In order to extend to the three-type models results obtained for the two-type models, we first assume that any two species have
 the same ability to exploit resources produced by the third one so that each species is either a ``global'' defector or
 a ``global'' cooperator.
 Under this assumption, the dynamics only depend on three parameters being described by matrices that belong to the
 three-dimensional manifold that consists of the matrices of the form
 $$ M_8 \ = \ \left(\hspace{-3pt} \begin{array}{ccc}
      2 \theta_1 & 1 - \theta_2 & 1 - \theta_3 \\ 1 - \theta_1 & 2 \theta_2 & 1 - \theta_3 \\ 1 - \theta_1 & 1 - \theta_2 & 2 \theta_3 \end{array} \hspace{-3pt} \right). $$
 Note that species $i$ is a defector if $\theta_i > 1/3$ and a cooperator if $\theta_i < 1/3$.
 Following the same approach as for the two-type models and motivated by the analytical results for the nonspatial three-type
 model, we distinguish four cases depending on the number of defectors.

\indent \emph{Presence of a cheater.}
 In case one species is a global defector and the other two species are global cooperators, so that the defector is a cheater, analytical
 results for the nonspatial models indicate that the cheater outcompetes the other two species (see Theorem \ref{th:global-trivial}).
 Analytical results for the spatial two-type model easily extend to prove that the same holds including space in the
 form of local interactions.
 The intuition is the exact same as for the two-type model.
 The ability of the cheater to exploit resources of either type is larger or equal to the ability of the other two species
 to exploit the same resource.
 Although the cheater has the same ability of any of the other species to exploit the resources produced by the third one, since
 it has a better ability to exploit the resources it produces, its fitness is always strictly larger than that of the other
 species \emph{regardless} of the configuration of the system.

\indent \emph{Presence of two cheaters.}
 This case occurs when there are two global defectors and one global cooperator.
 That space is present or not, the altruist species is driven to extinction.
 This can be understood based again on the same heuristic argument, which indicates that, \emph{regardless} of the configuration
 of the system, the cooperator always has a strictly smaller fitness than the two defectors.
 The long-term behavior is therefore similar to that of the two-type models in the presence of competition.
 In the mean-field model, one of the two defectors outcompetes the other one, and the limiting densities depend on the
 combination of the parameter values and the initial densities: the system is bistable (see Theorem \ref{th:global-trivial}).
 In the interacting particle system, the dominant type is uniquely determined by the parameter values and is always the least
 cooperative species.
 In case of a symmetric behavior between both defectors, the community clusters.

\indent \emph{Space and competition.}
 Similarly to the two-type model, the most interesting behaviors emerge when all three species compete or all three species
 cooperate.
 In the presence of competition, that is when all three species are global defectors, the mean-field model is tristable:
 the trajectories converge to one of the three trivial equilibria in which only one species is present, and the limiting
 densities again depend on both the parameter values and the initial densities.
 Numerical simulations of the spatial model (see the upper panel of Table \ref{tab:matrix_8}), indicate that the dominant
 type is again the least cooperative species.
 In case of equality, additional behaviors can appear:
 clustering of the two species that have the (same) highest relative ability to exploit the resources they produce, or clustering
 of all three species in the neutral case.
 The disagreement between both models in the presence of competition follows from the same argument as for the two-type models.
 The density of individuals of a given type seen by an individual of the same type is in average significantly larger in the
 interacting particle system than in the mean-field approximation in such a way that global densities become irrelevant and
 interfaces expand linearly in favor of the species that has the largest relative ability to exploit its own resources, i.e.,
 the least cooperative one.

\begin{table}[t!]
  \centering
  \scalebox{0.32}{\input{tab-matrix_8.pstex_t}} \vspace{5pt}
  \caption{}
\label{tab:matrix_8}
\end{table}

\indent \emph{Coexistence and cooperation.}
 Recall that in the presence of three species, global cooperation alone doe not imply coexistence for the nonspatial model,
 since additional conditions about the nontrivial boundary equilibria are required (see statement 3 in Theorem \ref{th:permanence}).
 However, when the dynamics are described by the matrix $M_8$ one has
 $$ \begin{array}{rcl}
    \Delta_{1,2} & = & (2 a_{3,1} - a_{1,1} - a_{2,1}) (a_{1,2} - a_{2,2}) + (2 a_{3,2} - a_{1,2} - a_{2,2}) (a_{2,1} - a_{1,1}) \vspace{4pt} \\ & = &
     (1 - 3 \theta_1) (1 - 3 \theta_2) + (1 - 3 \theta_2) (1 - 3 \theta_1) \ = \ 2 \,(1 - 3 \theta_1) (1 - 3 \theta_2) \end{array} $$
 and similar expressions for $\Delta_{2,3}$ and $\Delta_{3,1}$.
 Therefore, in this special case, coexistence occurs in the nonspatial mean-field model if and only if species cooperate.
 In contrast, in the presence of global cooperation, the interacting particle system has no less than seven possible regimes that
 are all stable under small perturbations of the three parameters.
 The existence of all these regimes follows heuristically from the same argument that explains the reduction of the coexistence
 region for the spatial two-type model: the density of individuals of a given type seen by an individual of the other type is in
 average significantly smaller in the interacting particle system than in the mean-field approximation, which makes species
 significantly more cooperative than others unable to survive in the presence of space.
 When cooperation is strong or fair among all three species, then global coexistence occurs even in the presence of space.
 However, when two species are significantly less cooperative than the third species, these two species can coexist but outcompete
 the third one (see the middle panel of Table \ref{tab:matrix_8}).
 Similarly, when a single species is significantly less cooperative than the other two species, this species is the dominant
 type: it outcompetes the other two species even when starting from a low density (see the bottom panel of Table \ref{tab:matrix_8}).

\begin{figure}[t!]
  \centering
  \scalebox{0.36}{\input{diagrams-three-type.pstex_t}}
  \caption{\upshape Phase diagram of the spatial three-type spatial model restricted to matrices of the form $M_8$.}
\label{fig:three-type}
\end{figure}
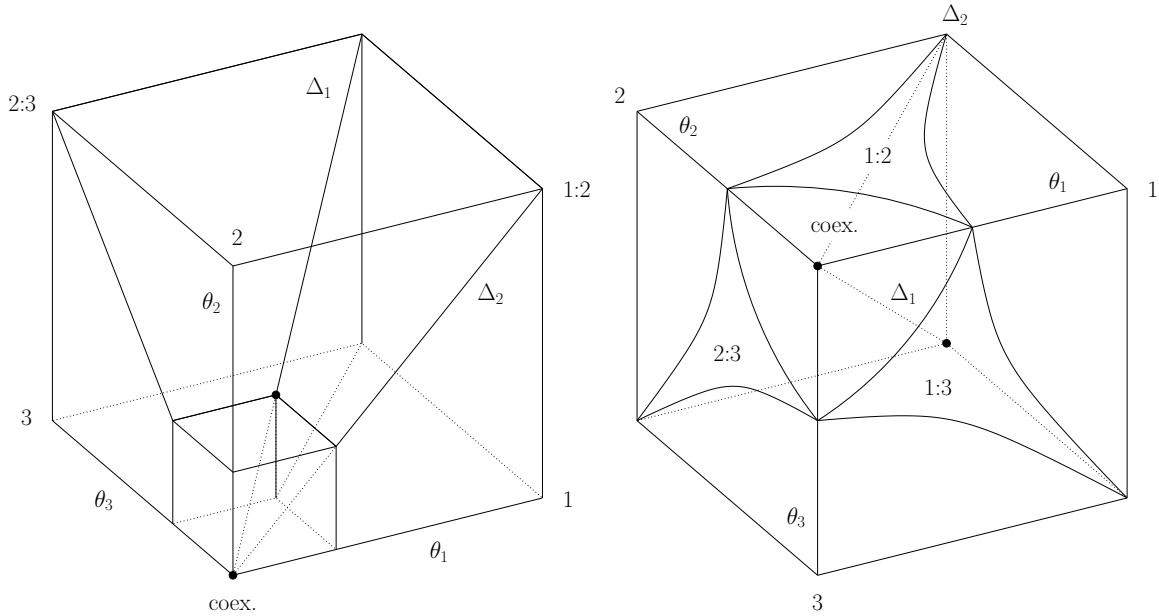

\begin{figure}[t!]
\centering
 \includegraphics[width = 440pt]{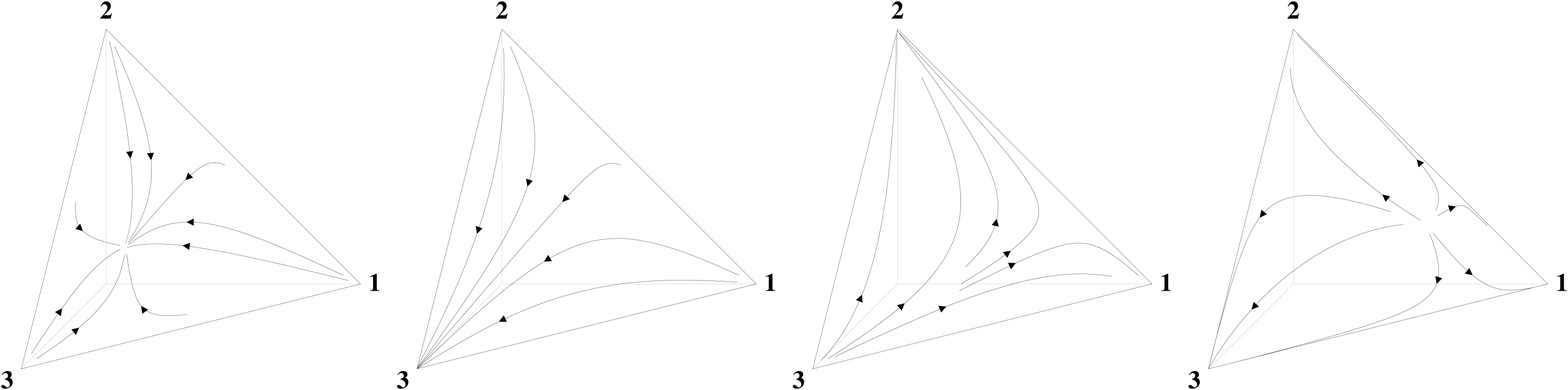}
\caption{\upshape Solution curves when the dynamics are described by $M_8$.}
%  picture 1 -- theta_1 = 0.15, theta_2 = 0.20, theta_3 = 0.25
%  picture 2 -- theta_1 = 0.15, theta_2 = 0.20, theta_3 = 0.50
%  picture 3 -- theta_1 = 0.40, theta_2 = 0.50, theta_3 = 0.30
%  picture 4 -- theta_1 = 0.50, theta_2 = 0.60, theta_3 = 0.70
\label{fig:curves}
\end{figure}

\indent In conclusion, the phase diagram of the nonspatial model when the set of parameters is restricted to matrices of the
 form $M_8$ can be deduced from Theorems \ref{th:global-trivial} and \ref{th:permanence}, and Corollary \ref{cor:tristability}.
 The unit cube representing the possible values of $\theta_1$, $\theta_2$ and $\theta_3$ is divided into eight regions which are
 delimited by the three two-dimensional hyperplanes with equation $\theta_i = 1/3$.
 In the presence of three cooperators, coexistence occurs.
 In the presence of two cooperators and one defector, the defector wins.
 In the presence of one cooperator and two defectors, the system is bistable.
 Finally, in the presence of three defectors, the system is tristable.
 We refer to Figure \ref{fig:curves} for the solution curves of the nonspatial model in these four cases.
 Finally, the three-dimensional phase diagram of the spatial model is depicted in Figure \ref{fig:three-type}.
 Excluding the cube with length edge $1/3$ at the bottom of the left picture, the phase diagram is divided into three regions
 delimited by three hyperplanes that intersect along the segment line, denoted by $\Delta_1$, that connects $(1/3, 1/3, 1/3)$ and $(1, 1, 1)$.
 Along this segment line, all three species cluster, along the hyperplanes two of the species outcompete the third one and cluster,
 and in the three regions delimited by the three hyperplanes there is a dominant type which is uniquely determined by the parameter values.
 The right picture shows an enlargement of the coexistence region which is included in the cube with length edge $1/3$ that corresponds
 to the region where all three pairs of species cooperate.
 This picture is drawn from the simulation results of Table \ref{tab:matrix_8}.
 Note that the spatial model exhibits seven regimes in this single cubic region.
 The first row of Figure \ref{fig:matrix_8} shows three snapshots of the spatial model when parameters are taken along the straight
 line $\Delta_2$.
 As the common value of $\theta_1$ and $\theta_2$ increases, the system crosses three different regimes:
 first all three species coexist, then species 1 and 2 coexist but outcompete species 3, and finally not only species 3 is driven
 to extinction but also species 1 and 2 cluster.
 The second row of the same figure shows three snapshots of the spatial model when parameters are taken along the straight line $\Delta_1$.
 In this neutral case, cooperation allows all three species to coexist while switching to competitive relationships, the system clusters,
 as shown by the last two snapshots.

\indent \emph{Coexistence of cheaters.}
 Recall that, according to Theorem \ref{th:permanence} and the discussion following the theorem, coexistence is possible in the
 absence of space regardless of the number of pairs of species that are in cooperative relationship or in defector-cooperator relationship,
 whereas coexistence cannot occur whenever one pair is competing.
 After looking at the first extreme case where all pairs are in a cooperative relationship, in which case the inclusion of space reduces
 the coexistence region, we now look at the other extreme case where all three pairs are in a defector-cooperator relationship that
 results in rock-paper-scissors dynamics.
 To understand the role of space in this case, we have performed numerical simulations of the spatial three-type model when the
 local interactions are described by matrices of the form
 $$ M_9 \ = \ \left(\hspace{-3pt} \begin{array}{ccc} \theta_1 & 0 & 1 \\ 1 & \theta_2 & 0 \\ 0 & 1 & \theta_3 \end{array} \hspace{-3pt} \right). $$
 Recall that, according to Theorem \ref{th:heteroclinic}, when the dynamics are described by the matrix $M_9$ one obtains a heteroclinic
 cycle by taking $\theta_i < 1$ for all $i = 1, 2, 3$.
 Recall also that, under the additional conditions \eqref{hc_s} the heteroclinic cycle is locally asymptotically stable whereas it is repelling,
 inducing permanence of the system, under the conditions \eqref{hc_u}.
 Numerical simulations of the nonspatial model further indicate that there is a unique interior equilibrium.
 When starting from any condition in $\tilde S_3$ excluding this interior equilibrium and under the assumptions \eqref{hc_s}, the
 trajectories get closer and closer to the boundaries of the system indicating that for large times at least one of the densities is
 arbitrarily close to zero, whereas under the assumptions \eqref{hc_u}, the trajectories converge to the interior equilibrium.
 In contrast, numerical simulations of the spatial three-type model indicate that taking $\theta_i < 1$ for all $i = 1, 2, 3$, suffices
 to obtain coexistence of all three species in space.
 These simulation results are reported in Table \ref{tab:matrix_9} where $\theta_3 = 0.80$.
 Note that, even when both parameters $\theta_1$ and $\theta_2$ are strictly larger than one half, in which case \eqref{hc_s} is satisfied and so
 the nonspatial model is not permanent, all three densities are bounded away from zero and the cluster coefficient
 is strictly smaller than that of the three-type voter model.
 In conclusion, whereas the inclusion of local interactions reduces the coexistence region when all three pairs of
 species are in a cooperator-cooperator relationship, it increases the coexistence region when on the contrary all three pairs of species
 are in a defector-cooperator relationship.
 The three snapshots of the spatial model in Figure \ref{fig:matrix_9} indicate that this also holds when starting with two
 species at low density.

\begin{table}[t!]
  \centering
  \scalebox{0.32}{\input{tab-matrix_9.pstex_t}} \vspace{5pt}
  \caption{}
\label{tab:matrix_9}
\end{table}

\begin{figure}[h!]
\centering
 \mbox{\subfigure[$\theta_1 = 0.04$]{\epsfig{figure = 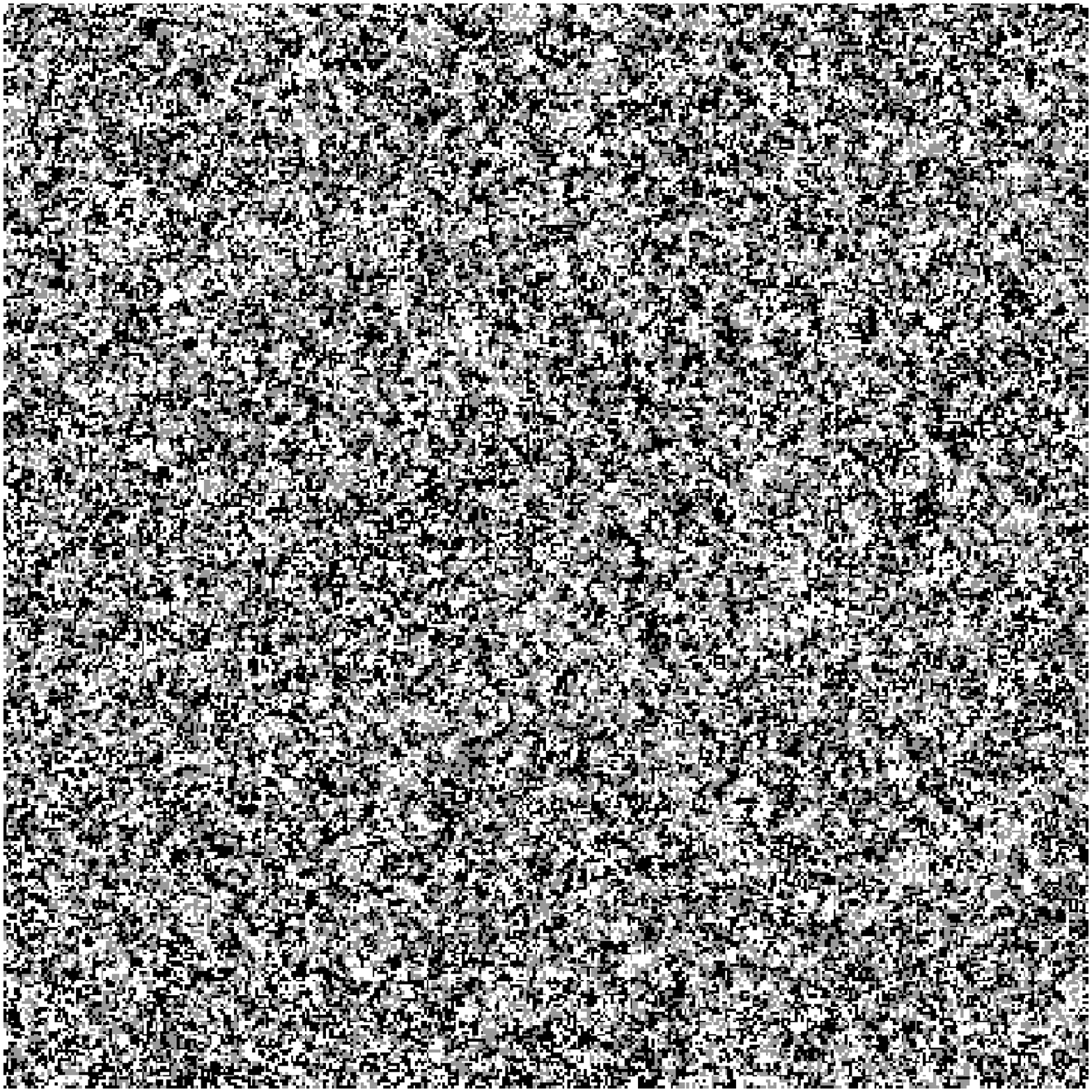, width = 140pt}}} \hspace{5pt}
 \mbox{\subfigure[$\theta_1 = 0.25$]{\epsfig{figure = 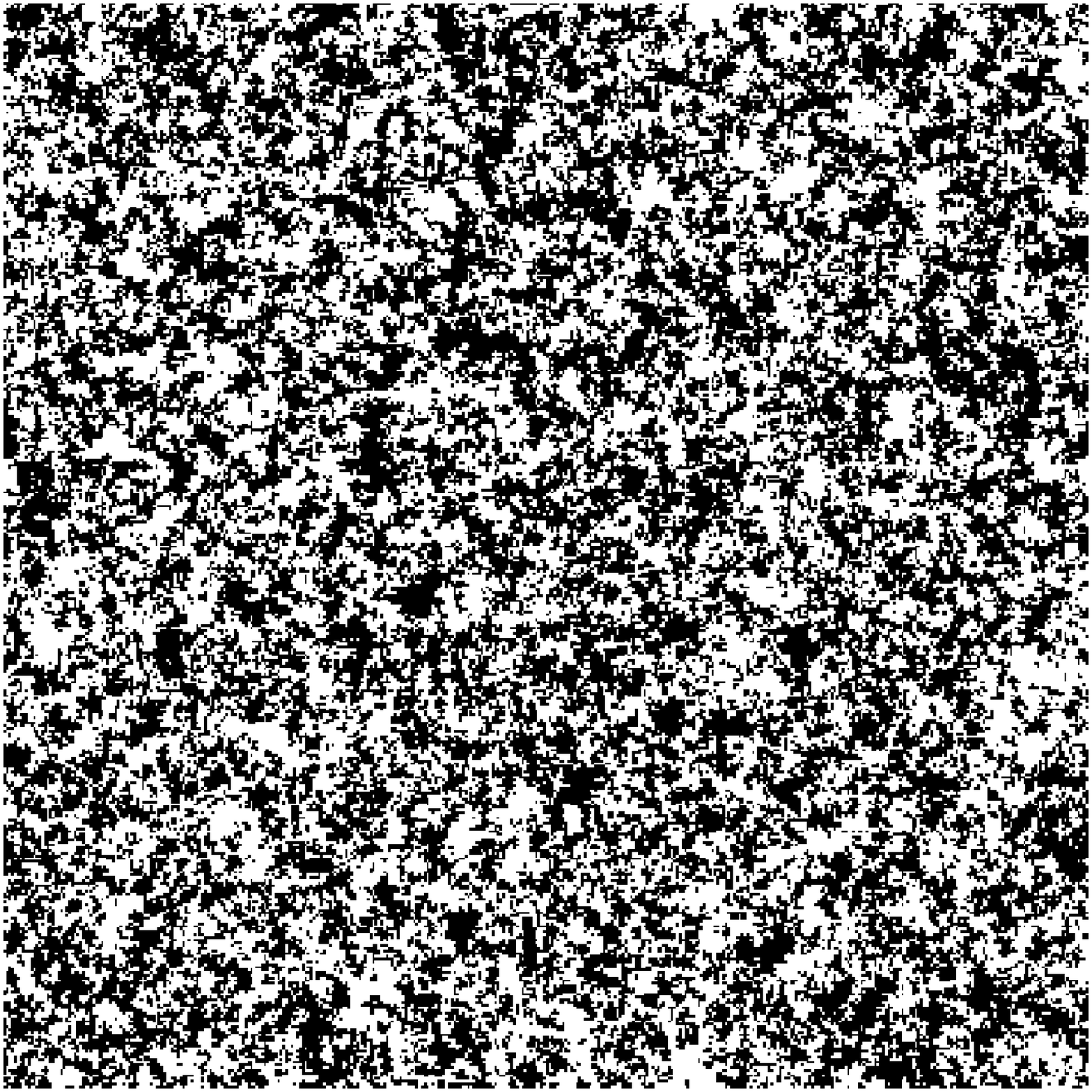, width = 140pt}}} \hspace{5pt}
 \mbox{\subfigure[$\theta_1 = 0.80$]{\epsfig{figure = 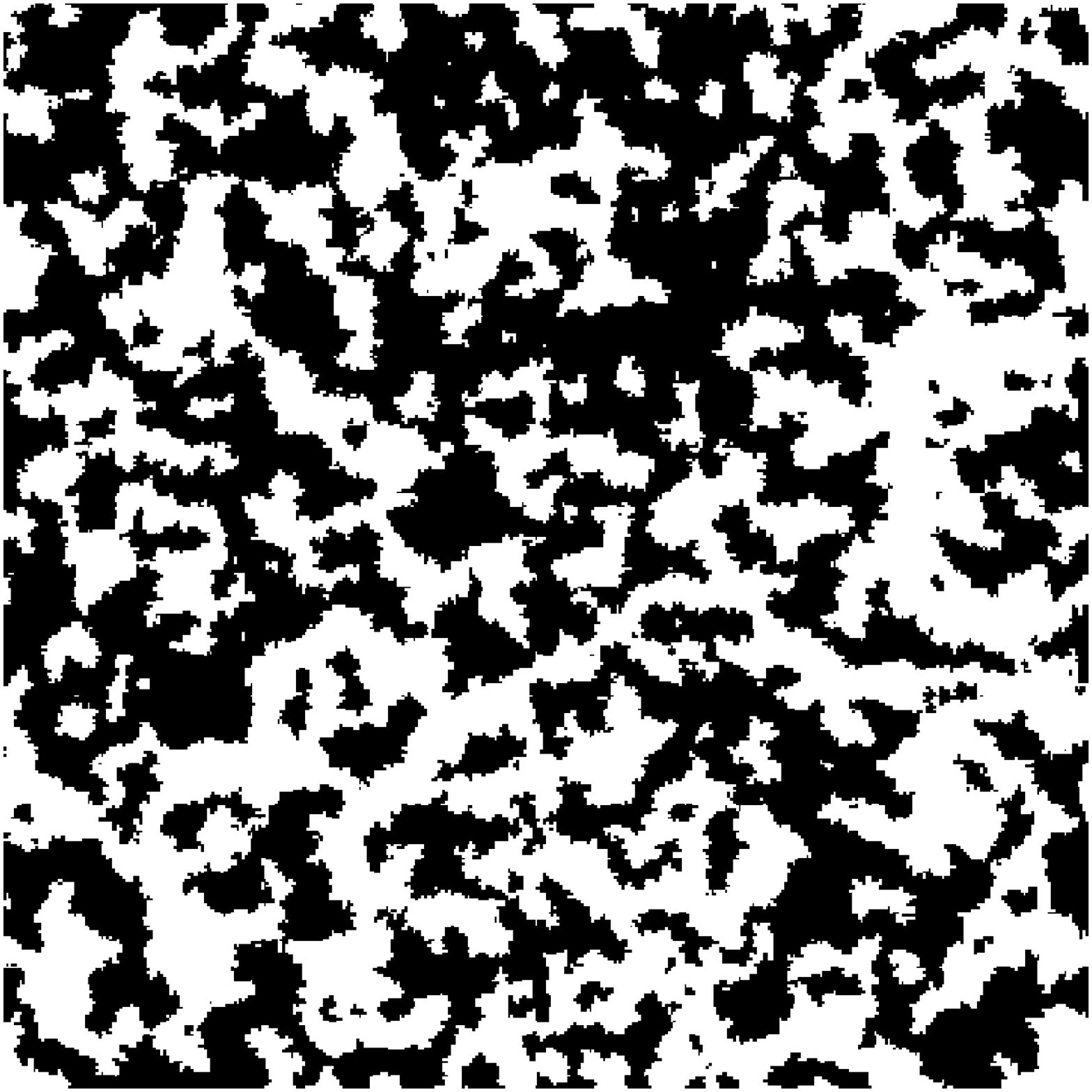, width = 140pt}}} \\
 \mbox{\subfigure[$\theta_1 = 0.25$]{\epsfig{figure = 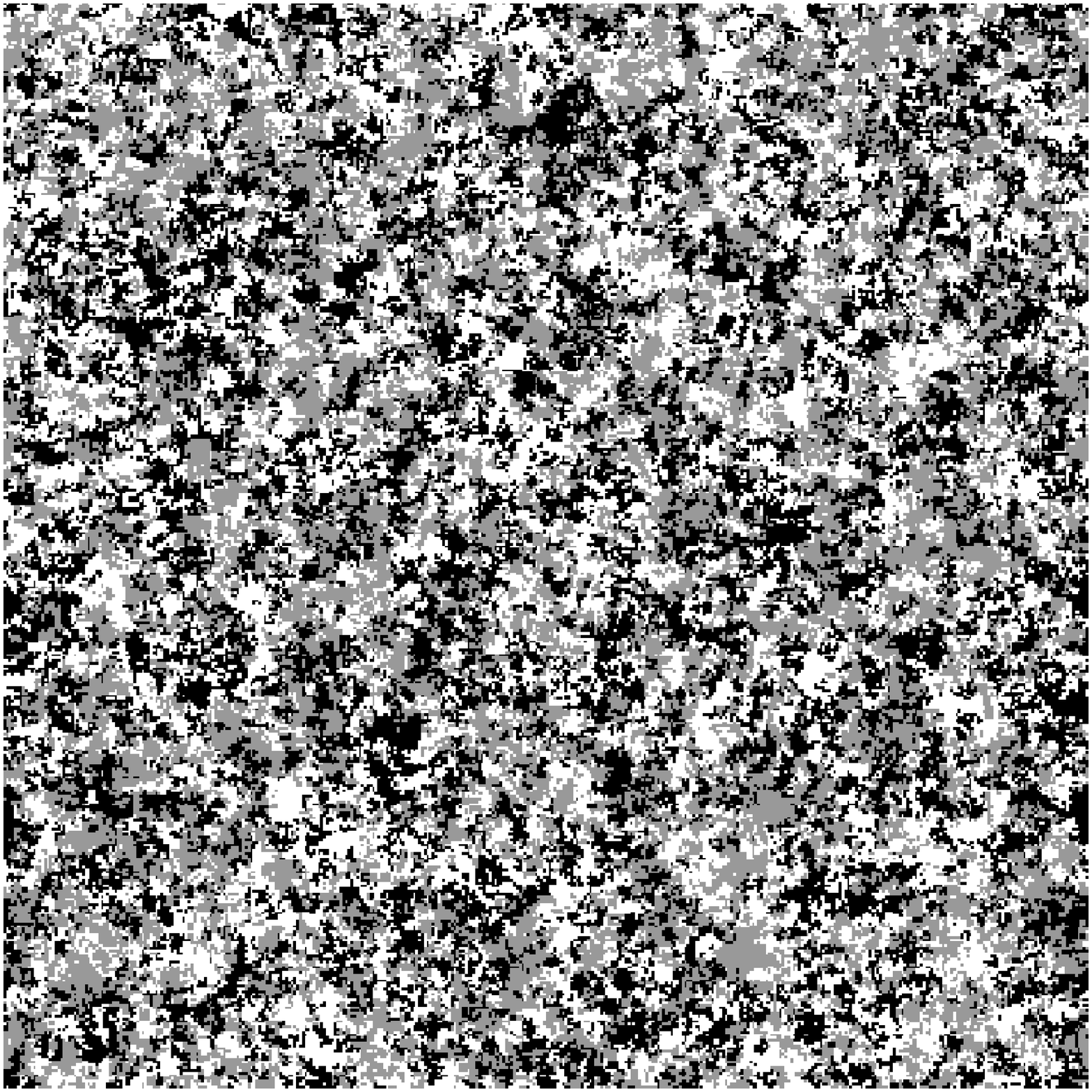, width = 140pt}}} \hspace{5pt}
 \mbox{\subfigure[$\theta_1 = 0.50$]{\epsfig{figure = 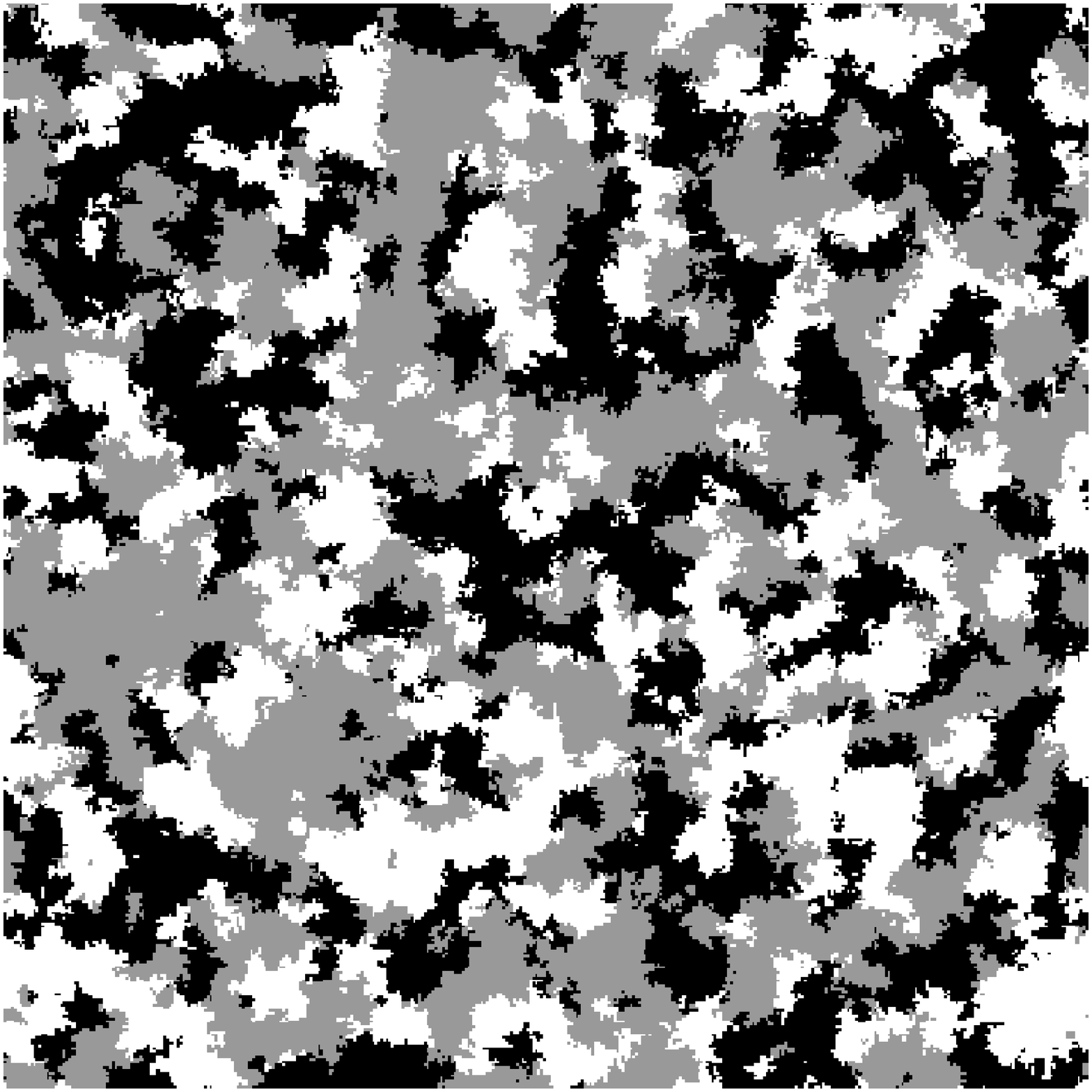, width = 140pt}}} \hspace{5pt}
 \mbox{\subfigure[$\theta_1 = 0.75$]{\epsfig{figure = 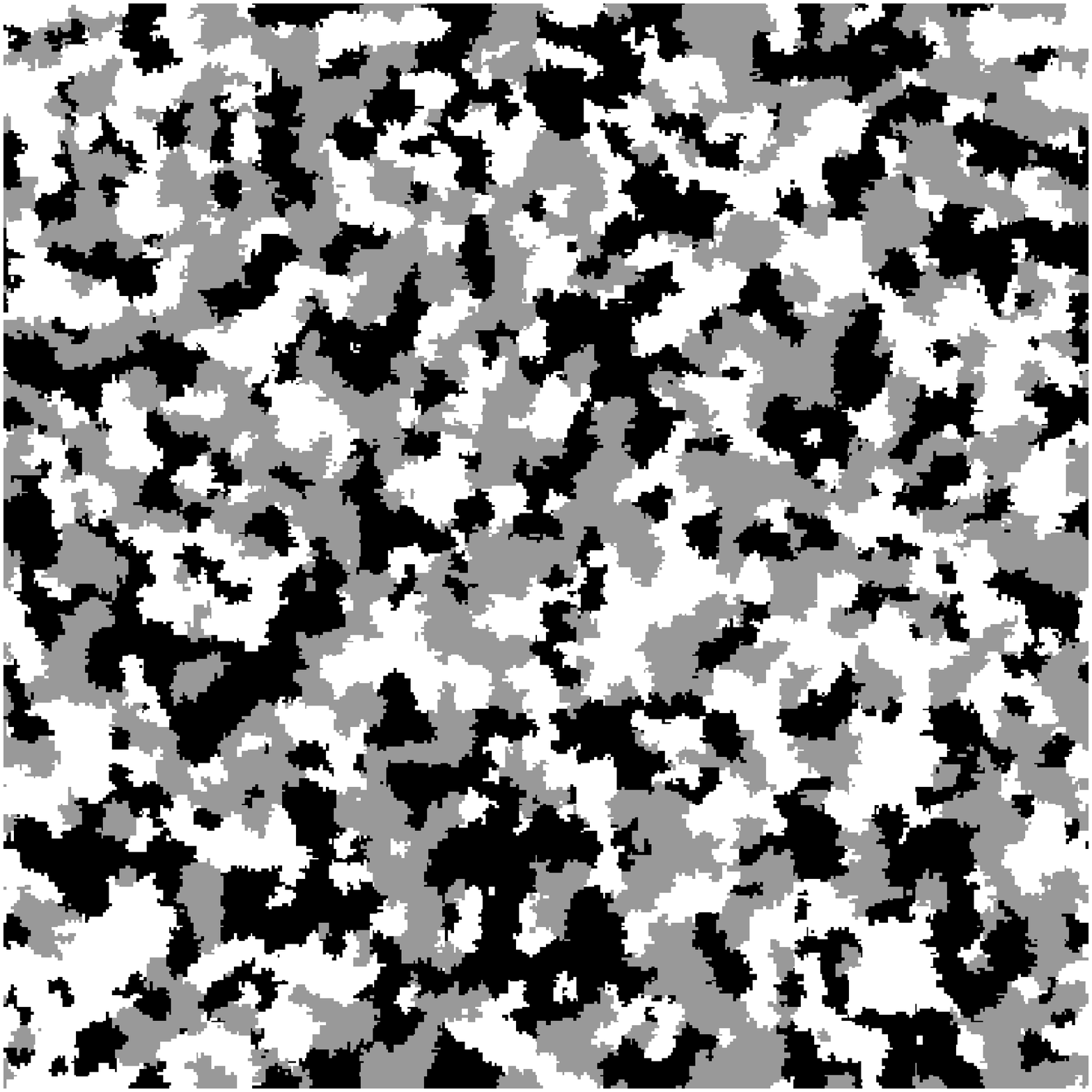, width = 140pt}}}
\caption{\upshape Snapshots of the three-type model at time 200 when starting from a product measure in which sites are equally likely
 to be of either type and when the local interactions are described by the matrix $M_9$.
 In the first row, the parameters are $\theta_1 = \theta_2$ and $\theta_3 = 0$, while in the second row,
 $\theta_1 = \theta_2 = \theta_3$.} \vspace{10pt}
\label{fig:matrix_8}
\centering
 \mbox{\subfigure[time  40]{\epsfig{figure = 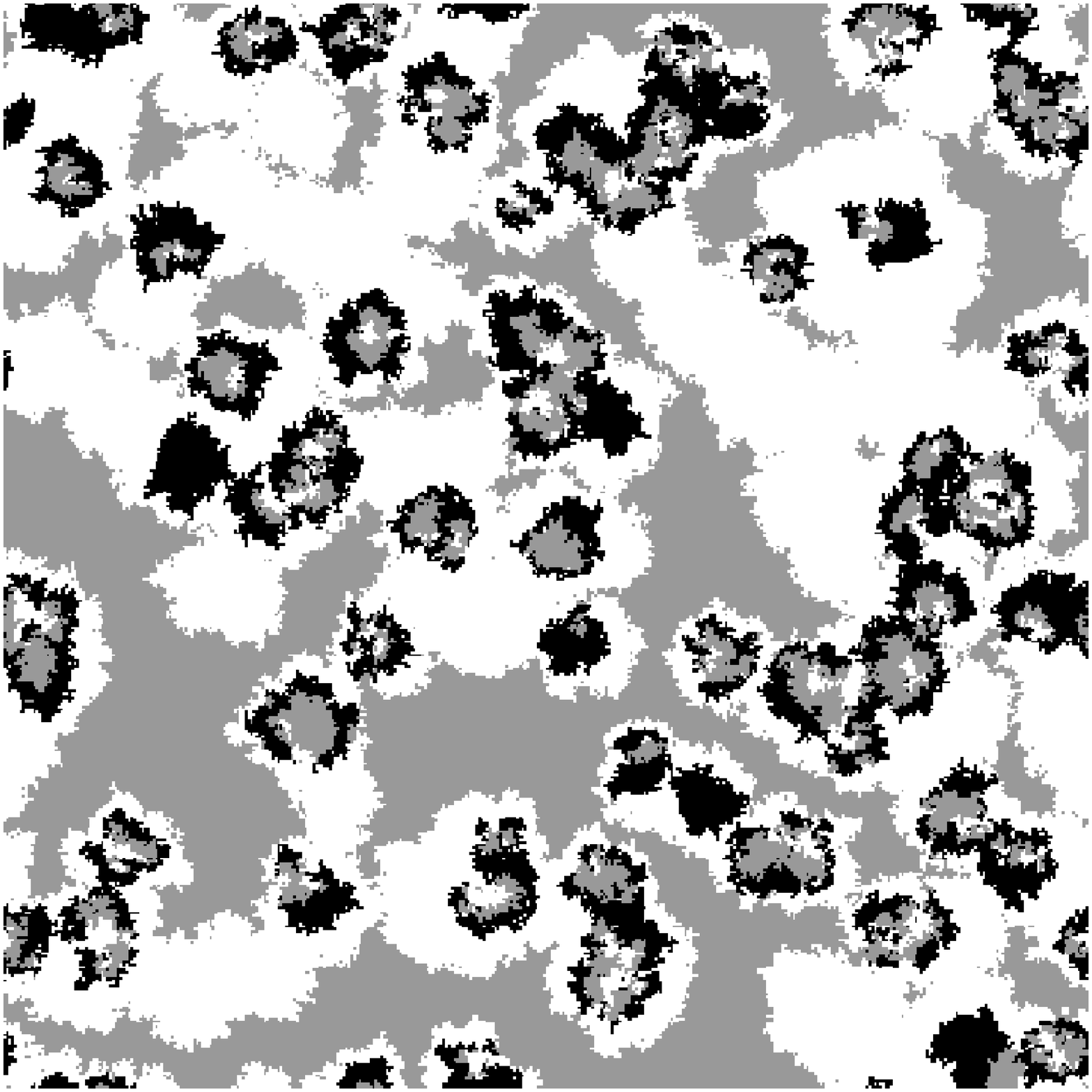, width = 140pt}}} \hspace{5pt}
 \mbox{\subfigure[time  80]{\epsfig{figure = 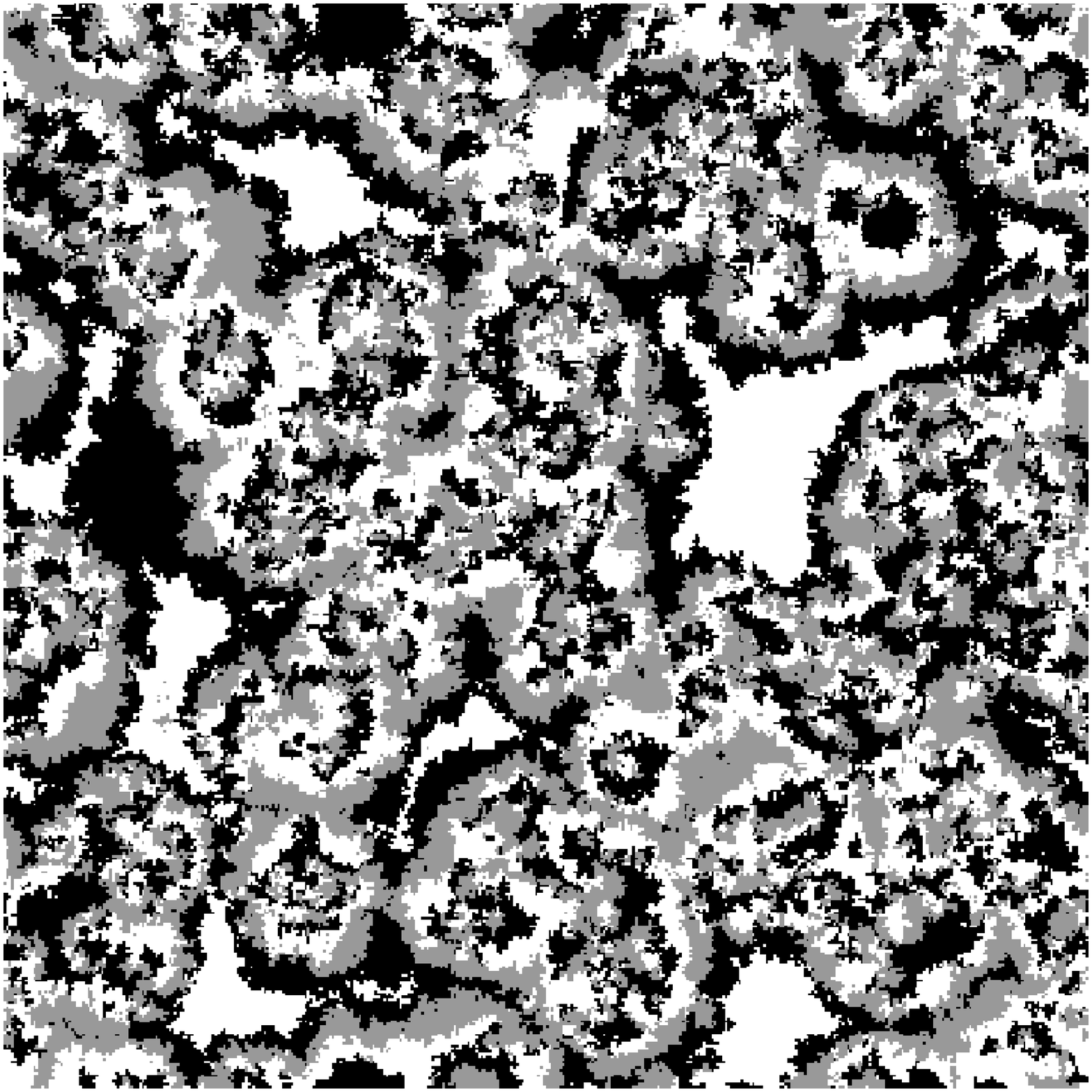, width = 140pt}}} \hspace{5pt}
 \mbox{\subfigure[time 120]{\epsfig{figure = 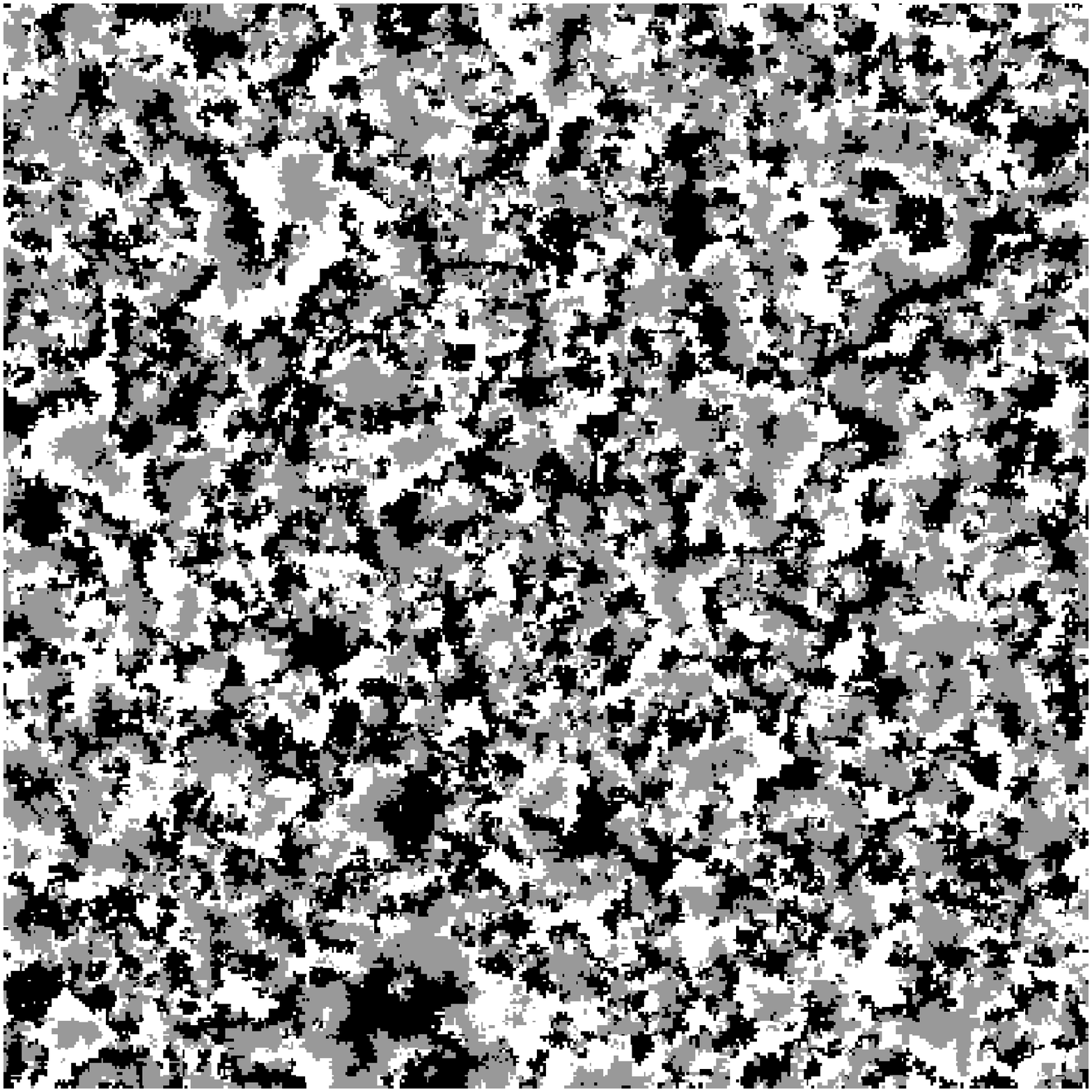, width = 140pt}}}
\caption{\upshape Snapshots of the spatial model starting with 90\% of rocks (black) and 5\% of papers and scissors (grey and
 white) when the local interactions are described by the matrix $M_9$ with $\theta_1 = \theta_2 = \theta_3 = 0.80$.}
\label{fig:matrix_9}
\end{figure}

\indent We finally explain the reason why local interactions promote coexistence in the case of rock-paper-scissors dynamics.
 The dynamical picture of the spatial model at equilibrium shows that small clusters keep forming very quickly but then shrink again
 immediately: residuals of one species, say paper, invade nearby clusters of rocks very quickly resulting in clusters of papers
 which, when they reach a critical size, are invaded and destroyed by nearby residuals of scissors, so the system exhibits local
 traveling waves.
 Focusing on a reasonably small piece of the lattice, one sees a majority of rocks, which are suddenly invaded by a wave of papers that
 become majoritary, which themselves are suddenly invaded by a wave of scissors, and so on.
 In particular, the dynamics predicted by the nonspatial model under conditions \eqref{hc}-\eqref{hc_s} indeed reflects the dynamics of
 the spatial model viewed in a small region of the lattice.
 However, due to the presence of local interactions, the global densities become irrelevant and small pieces of the lattice that are far
 from each other are almost independent in such a way that the nonspatial dynamics are only symptomatic of what happens in small spatial
 regions.
 In addition, due to stochasticity, there is a strong asynchrony among nearby regions so, even if a region is void in one type, this
 type can quickly reappear.

%%%%%%%%%%%%%%%%%%%%%%%%%%%%%%%%%%%%%%%%%%%%%%%%%%%%%%%%%%%%%%%%%%%%%%%%%%%%%%%%%%%%%%%%%%%%%%%%%%%%%%%%%%%%%%%%%%%%%%%%%%%%%%%%%%%%%%%%%%

\section{Discussion}
\label{sec:discussion}

\indent The long-term behavior of ecological communities involving competitive defector-defector interactions, cooperative
 interactions, and defector-cooperator interactions are difficult to predict mainly because the fitness of each individual
 depends on the environment, which induces a constant feedback between the set fitnesses and the configuration of the community.
 In situations where one species has a higher ability than all the other species to exploit resources of either type, the fitness
 of the first species is always strictly larger than that of the other ones regardless of the configuration of the community.
 Excluding such situations, the interacting particle system and its mean-field approximation generally disagree.
 This is due to the fact that the environment seen by each individual strongly differs in the presence and in the
 absence of a spatial structure.

\indent Regardless of the size of the community, when all pairs of species are in a competitive relationship, the mean-field
 model predicts multistability of the system: one type eventually takes over and this dominant type is determined by the
 combination of both the parameters of the system and the initial densities.
 In contrast, in the presence of local interactions, the dominant type is uniquely determined by the parameters.
 The reason is that, regardless of the global densities, the fractions of species seen from the spatial interface
 between two types are roughly the same for both types which results in traveling waves that expand in favor of the dominant type.

\indent In the presence of only two species, cooperation is the only mechanism that promotes coexistence: in the absence
 of space, coexistence occurs if and only if cooperation occurs.
 Introducing additional species has two important consequences.
 First, even in the mean-field approximation, global cooperation alone does not result in the coexistence of all species
 since a too strong cooperative behavior between two species can drive the other species to extinction.
 In particular, that any two species can coexist does not imply that species all together can coexist.
 Second, there are additional mechanisms that promote coexistence.
 Rock-paper-scissors dynamics in which all pairs of species are in a defector-cooperator relationship is an example of such
 a mechanism.
 In particular, that any two species cannot coexist in the absence of the other species does not imply that species all
 together cannot coexist.
 Another important aspect is that, at least in the presence of at most three species, whenever two species are in a
 competitive defector-defector relationship, global coexistence is never possible.
 Therefore, the collection of mechanisms that promote coexistence when three species are present consist of global cooperative
 behavior, rock-paper-scissors type dynamics, and mixture of these two extreme cases.

\indent The effect of local interactions in the presence of either global cooperation or rock-paper-scissors
 dynamics is probably the most interesting aspect of our study:
 in the presence of mutual cooperative relationships, the inclusion of local interactions translates into a reduction of
 the coexistence region, whereas in the presence of rock-paper-scissors dynamics, the inclusion of local interactions
 translates into an expansion of the coexistence region.
 The intuition behind the first statement is that the fraction of individuals of a given species seen by an individual
 of another species is significantly smaller in the interacting particle system than in the mean-field model indicating
 that, in the presence of local interactions, cooperative species cannot fully benefit from the resources produced by the
 other species.
 This effect is more pronounced in low spatial dimensions.
 Finally, the intuition behind the second statement is that the behavior predicted by the mean-field model in the presence
 of rock-paper-scissors dynamics is only symptomatic of the local behavior of the interacting particle system seen in
 a fairly small region.
 Since in addition stochasticity induces a strong asynchrony among disjoint spatial regions, when a species is driven
 to extinction locally, it is quickly reintroduced from nearby regions.

%%%%%%%%%%%%%%%%%%%%%%%%%%%%%%%%%%%%%%%%%%%%%%%%%%%%%%%%%%%%%%%%%%%%%%%%%%%%%%%%%%%%%%%%%%%%%%%%%%%%%%%%%%%%%%%%%%%%%%%%%%%%%%%%%%%%%%%%%%

\noindent\textbf{Acknowledgment}.
 The authors would like to thank an anonymous referee for important comments that helped to improve the article.

%%%%%%%%%%%%%%%%%%%%%%%%%%%%%%%%%%%%%%%%%%%%%%%%%%%%%%%%%%%%%%%%%%%%%%%%%%%%%%%%%%%%%%%%%%%%%%%%%%%%%%%%%%%%%%%%%%%%%%%%%%%%%%%%%%%%%%%%%%

\end{document}

%% file: tab-two-type.pstex_t
\begin{picture}(0,0)%
\includegraphics{tab-two-type.pstex}%
\end{picture}%
\setlength{\unitlength}{3947sp}%
\begingroup\makeatletter\ifx\SetFigFontNFSS\undefined%
\gdef\SetFigFontNFSS#1#2#3#4#5{%
  \reset@font\fontsize{#1}{#2pt}%
  \fontfamily{#3}\fontseries{#4}\fontshape{#5}%
  \selectfont}%
\fi\endgroup%
\begin{picture}(23009,15069)(-3086,-13243)
\put(8401,-5311){\makebox(0,0)[b]{\smash{{\SetFigFontNFSS{20}{24.0}{\familydefault}{\mddefault}{\updefault}Densities and clustering coefficient for the two-type model when $\theta_1, \theta_2 \geq 1/2$}}}}
\put(8401,-13111){\makebox(0,0)[b]{\smash{{\SetFigFontNFSS{20}{24.0}{\familydefault}{\mddefault}{\updefault}Densities and clustering coefficient for the two-type model when $\theta_1, \theta_2 < 1/2$}}}}
\put(  1,1439){\makebox(0,0)[b]{\smash{{\SetFigFontNFSS{20}{24.0}{\familydefault}{\mddefault}{\updefault}$\theta_1 = 0.50$}}}}
\put(2101,1439){\makebox(0,0)[b]{\smash{{\SetFigFontNFSS{20}{24.0}{\familydefault}{\mddefault}{\updefault}$\theta_1 = 0.55$}}}}
\put(4201,1439){\makebox(0,0)[b]{\smash{{\SetFigFontNFSS{20}{24.0}{\familydefault}{\mddefault}{\updefault}$\theta_1 = 0.60$}}}}
\put(6301,1439){\makebox(0,0)[b]{\smash{{\SetFigFontNFSS{20}{24.0}{\familydefault}{\mddefault}{\updefault}$\theta_1 = 0.65$}}}}
\put(8401,1439){\makebox(0,0)[b]{\smash{{\SetFigFontNFSS{20}{24.0}{\familydefault}{\mddefault}{\updefault}$\theta_1 = 0.70$}}}}
\put(10501,1439){\makebox(0,0)[b]{\smash{{\SetFigFontNFSS{20}{24.0}{\familydefault}{\mddefault}{\updefault}$\theta_1 = 0.75$}}}}
\put(12601,1439){\makebox(0,0)[b]{\smash{{\SetFigFontNFSS{20}{24.0}{\familydefault}{\mddefault}{\updefault}$\theta_1 = 0.80$}}}}
\put(14701,1439){\makebox(0,0)[b]{\smash{{\SetFigFontNFSS{20}{24.0}{\familydefault}{\mddefault}{\updefault}$\theta_1 = 0.85$}}}}
\put(16801,1439){\makebox(0,0)[b]{\smash{{\SetFigFontNFSS{20}{24.0}{\familydefault}{\mddefault}{\updefault}$\theta_1 = 0.90$}}}}
\put(18901,1439){\makebox(0,0)[b]{\smash{{\SetFigFontNFSS{20}{24.0}{\familydefault}{\mddefault}{\updefault}$\theta_1 = 0.95$}}}}
\put(-2099,839){\makebox(0,0)[b]{\smash{{\SetFigFontNFSS{20}{24.0}{\familydefault}{\mddefault}{\updefault}$\theta_2 = 0.50$}}}}
\put(-2099,239){\makebox(0,0)[b]{\smash{{\SetFigFontNFSS{20}{24.0}{\familydefault}{\mddefault}{\updefault}$\theta_2 = 0.55$}}}}
\put(-2099,-361){\makebox(0,0)[b]{\smash{{\SetFigFontNFSS{20}{24.0}{\familydefault}{\mddefault}{\updefault}$\theta_2 = 0.60$}}}}
\put(-2099,-961){\makebox(0,0)[b]{\smash{{\SetFigFontNFSS{20}{24.0}{\familydefault}{\mddefault}{\updefault}$\theta_2 = 0.65$}}}}
\put(-2099,-1561){\makebox(0,0)[b]{\smash{{\SetFigFontNFSS{20}{24.0}{\familydefault}{\mddefault}{\updefault}$\theta_2 = 0.70$}}}}
\put(-2099,-2161){\makebox(0,0)[b]{\smash{{\SetFigFontNFSS{20}{24.0}{\familydefault}{\mddefault}{\updefault}$\theta_2 = 0.75$}}}}
\put(-2099,-2761){\makebox(0,0)[b]{\smash{{\SetFigFontNFSS{20}{24.0}{\familydefault}{\mddefault}{\updefault}$\theta_2 = 0.80$}}}}
\put(-2099,-3361){\makebox(0,0)[b]{\smash{{\SetFigFontNFSS{20}{24.0}{\familydefault}{\mddefault}{\updefault}$\theta_2 = 0.85$}}}}
\put(-2099,-3961){\makebox(0,0)[b]{\smash{{\SetFigFontNFSS{20}{24.0}{\familydefault}{\mddefault}{\updefault}$\theta_2 = 0.90$}}}}
\put(-2099,-4561){\makebox(0,0)[b]{\smash{{\SetFigFontNFSS{20}{24.0}{\familydefault}{\mddefault}{\updefault}$\theta_2 = 0.95$}}}}
\put(  1,239){\makebox(0,0)[b]{\smash{{\SetFigFontNFSS{20}{24.0}{\familydefault}{\mddefault}{\updefault}0:100 (100)}}}}
\put(  1,-361){\makebox(0,0)[b]{\smash{{\SetFigFontNFSS{20}{24.0}{\familydefault}{\mddefault}{\updefault}0:100 (100)}}}}
\put(  1,-961){\makebox(0,0)[b]{\smash{{\SetFigFontNFSS{20}{24.0}{\familydefault}{\mddefault}{\updefault}0:100 (100)}}}}
\put(  1,-1561){\makebox(0,0)[b]{\smash{{\SetFigFontNFSS{20}{24.0}{\familydefault}{\mddefault}{\updefault}0:100 (100)}}}}
\put(  1,-2161){\makebox(0,0)[b]{\smash{{\SetFigFontNFSS{20}{24.0}{\familydefault}{\mddefault}{\updefault}0:100 (100)}}}}
\put(  1,-2761){\makebox(0,0)[b]{\smash{{\SetFigFontNFSS{20}{24.0}{\familydefault}{\mddefault}{\updefault}0:100 (100)}}}}
\put(  1,-3361){\makebox(0,0)[b]{\smash{{\SetFigFontNFSS{20}{24.0}{\familydefault}{\mddefault}{\updefault}0:100 (100)}}}}
\put(  1,-3961){\makebox(0,0)[b]{\smash{{\SetFigFontNFSS{20}{24.0}{\familydefault}{\mddefault}{\updefault}0:100 (100)}}}}
\put(  1,-4561){\makebox(0,0)[b]{\smash{{\SetFigFontNFSS{20}{24.0}{\familydefault}{\mddefault}{\updefault}0:100 (100)}}}}
\put(2101,839){\makebox(0,0)[b]{\smash{{\SetFigFontNFSS{20}{24.0}{\familydefault}{\mddefault}{\updefault}100:0 (100)}}}}
\put(2101,239){\makebox(0,0)[b]{\smash{{\SetFigFontNFSS{20}{24.0}{\familydefault}{\mddefault}{\updefault}48:51 (95)}}}}
\put(2101,-361){\makebox(0,0)[b]{\smash{{\SetFigFontNFSS{20}{24.0}{\familydefault}{\mddefault}{\updefault}0:100 (100)}}}}
\put(2101,-961){\makebox(0,0)[b]{\smash{{\SetFigFontNFSS{20}{24.0}{\familydefault}{\mddefault}{\updefault}0:100 (100)}}}}
\put(2101,-1561){\makebox(0,0)[b]{\smash{{\SetFigFontNFSS{20}{24.0}{\familydefault}{\mddefault}{\updefault}0:100 (100)}}}}
\put(2101,-2161){\makebox(0,0)[b]{\smash{{\SetFigFontNFSS{20}{24.0}{\familydefault}{\mddefault}{\updefault}0:100 (100)}}}}
\put(2101,-2761){\makebox(0,0)[b]{\smash{{\SetFigFontNFSS{20}{24.0}{\familydefault}{\mddefault}{\updefault}0:100 (100)}}}}
\put(2101,-3361){\makebox(0,0)[b]{\smash{{\SetFigFontNFSS{20}{24.0}{\familydefault}{\mddefault}{\updefault}0:100 (100)}}}}
\put(2101,-3961){\makebox(0,0)[b]{\smash{{\SetFigFontNFSS{20}{24.0}{\familydefault}{\mddefault}{\updefault}0:100 (100)}}}}
\put(2101,-4561){\makebox(0,0)[b]{\smash{{\SetFigFontNFSS{20}{24.0}{\familydefault}{\mddefault}{\updefault}0:100 (100)}}}}
\put(4201,839){\makebox(0,0)[b]{\smash{{\SetFigFontNFSS{20}{24.0}{\familydefault}{\mddefault}{\updefault}100:0 (100)}}}}
\put(4201,239){\makebox(0,0)[b]{\smash{{\SetFigFontNFSS{20}{24.0}{\familydefault}{\mddefault}{\updefault}100:0 (100)}}}}
\put(4201,-361){\makebox(0,0)[b]{\smash{{\SetFigFontNFSS{20}{24.0}{\familydefault}{\mddefault}{\updefault}49:50 (96)}}}}
\put(4201,-961){\makebox(0,0)[b]{\smash{{\SetFigFontNFSS{20}{24.0}{\familydefault}{\mddefault}{\updefault}0:100 (100)}}}}
\put(4201,-1561){\makebox(0,0)[b]{\smash{{\SetFigFontNFSS{20}{24.0}{\familydefault}{\mddefault}{\updefault}0:100 (100)}}}}
\put(4201,-2161){\makebox(0,0)[b]{\smash{{\SetFigFontNFSS{20}{24.0}{\familydefault}{\mddefault}{\updefault}0:100 (100)}}}}
\put(4201,-2761){\makebox(0,0)[b]{\smash{{\SetFigFontNFSS{20}{24.0}{\familydefault}{\mddefault}{\updefault}0:100 (100)}}}}
\put(4201,-3361){\makebox(0,0)[b]{\smash{{\SetFigFontNFSS{20}{24.0}{\familydefault}{\mddefault}{\updefault}0:100 (100)}}}}
\put(4201,-3961){\makebox(0,0)[b]{\smash{{\SetFigFontNFSS{20}{24.0}{\familydefault}{\mddefault}{\updefault}0:100 (100)}}}}
\put(4201,-4561){\makebox(0,0)[b]{\smash{{\SetFigFontNFSS{20}{24.0}{\familydefault}{\mddefault}{\updefault}0:100 (100)}}}}
\put(6301,839){\makebox(0,0)[b]{\smash{{\SetFigFontNFSS{20}{24.0}{\familydefault}{\mddefault}{\updefault}100:0 (100)}}}}
\put(6301,239){\makebox(0,0)[b]{\smash{{\SetFigFontNFSS{20}{24.0}{\familydefault}{\mddefault}{\updefault}100:0 (100)}}}}
\put(6301,-361){\makebox(0,0)[b]{\smash{{\SetFigFontNFSS{20}{24.0}{\familydefault}{\mddefault}{\updefault}100:0 (100)}}}}
\put(6301,-961){\makebox(0,0)[b]{\smash{{\SetFigFontNFSS{20}{24.0}{\familydefault}{\mddefault}{\updefault}50:49 (96)}}}}
\put(6301,-1561){\makebox(0,0)[b]{\smash{{\SetFigFontNFSS{20}{24.0}{\familydefault}{\mddefault}{\updefault}0:100 (100)}}}}
\put(6301,-2161){\makebox(0,0)[b]{\smash{{\SetFigFontNFSS{20}{24.0}{\familydefault}{\mddefault}{\updefault}0:100 (100)}}}}
\put(6301,-2761){\makebox(0,0)[b]{\smash{{\SetFigFontNFSS{20}{24.0}{\familydefault}{\mddefault}{\updefault}0:100 (100)}}}}
\put(6301,-3361){\makebox(0,0)[b]{\smash{{\SetFigFontNFSS{20}{24.0}{\familydefault}{\mddefault}{\updefault}0:100 (100)}}}}
\put(6301,-3961){\makebox(0,0)[b]{\smash{{\SetFigFontNFSS{20}{24.0}{\familydefault}{\mddefault}{\updefault}0:100 (100)}}}}
\put(6301,-4561){\makebox(0,0)[b]{\smash{{\SetFigFontNFSS{20}{24.0}{\familydefault}{\mddefault}{\updefault}0:100 (100)}}}}
\put(8401,839){\makebox(0,0)[b]{\smash{{\SetFigFontNFSS{20}{24.0}{\familydefault}{\mddefault}{\updefault}100:0 (100)}}}}
\put(8401,239){\makebox(0,0)[b]{\smash{{\SetFigFontNFSS{20}{24.0}{\familydefault}{\mddefault}{\updefault}100:0 (100)}}}}
\put(8401,-361){\makebox(0,0)[b]{\smash{{\SetFigFontNFSS{20}{24.0}{\familydefault}{\mddefault}{\updefault}100:0 (100)}}}}
\put(8401,-961){\makebox(0,0)[b]{\smash{{\SetFigFontNFSS{20}{24.0}{\familydefault}{\mddefault}{\updefault}100:0 (100)}}}}
\put(8401,-1561){\makebox(0,0)[b]{\smash{{\SetFigFontNFSS{20}{24.0}{\familydefault}{\mddefault}{\updefault}48:51 (97)}}}}
\put(8401,-2161){\makebox(0,0)[b]{\smash{{\SetFigFontNFSS{20}{24.0}{\familydefault}{\mddefault}{\updefault}0:100 (100)}}}}
\put(8401,-2761){\makebox(0,0)[b]{\smash{{\SetFigFontNFSS{20}{24.0}{\familydefault}{\mddefault}{\updefault}0:100 (100)}}}}
\put(8401,-3361){\makebox(0,0)[b]{\smash{{\SetFigFontNFSS{20}{24.0}{\familydefault}{\mddefault}{\updefault}0:100 (100)}}}}
\put(8401,-3961){\makebox(0,0)[b]{\smash{{\SetFigFontNFSS{20}{24.0}{\familydefault}{\mddefault}{\updefault}0:100 (100)}}}}
\put(8401,-4561){\makebox(0,0)[b]{\smash{{\SetFigFontNFSS{20}{24.0}{\familydefault}{\mddefault}{\updefault}0:100 (100)}}}}
\put(10501,839){\makebox(0,0)[b]{\smash{{\SetFigFontNFSS{20}{24.0}{\familydefault}{\mddefault}{\updefault}100:0 (100)}}}}
\put(10501,239){\makebox(0,0)[b]{\smash{{\SetFigFontNFSS{20}{24.0}{\familydefault}{\mddefault}{\updefault}100:0 (100)}}}}
\put(10501,-361){\makebox(0,0)[b]{\smash{{\SetFigFontNFSS{20}{24.0}{\familydefault}{\mddefault}{\updefault}100:0 (100)}}}}
\put(10501,-961){\makebox(0,0)[b]{\smash{{\SetFigFontNFSS{20}{24.0}{\familydefault}{\mddefault}{\updefault}100:0 (100)}}}}
\put(10501,-1561){\makebox(0,0)[b]{\smash{{\SetFigFontNFSS{20}{24.0}{\familydefault}{\mddefault}{\updefault}100:0 (100)}}}}
\put(10501,-2161){\makebox(0,0)[b]{\smash{{\SetFigFontNFSS{20}{24.0}{\familydefault}{\mddefault}{\updefault}55:44 (97)}}}}
\put(10501,-2761){\makebox(0,0)[b]{\smash{{\SetFigFontNFSS{20}{24.0}{\familydefault}{\mddefault}{\updefault}0:100 (100)}}}}
\put(10501,-3361){\makebox(0,0)[b]{\smash{{\SetFigFontNFSS{20}{24.0}{\familydefault}{\mddefault}{\updefault}0:100 (100)}}}}
\put(10501,-3961){\makebox(0,0)[b]{\smash{{\SetFigFontNFSS{20}{24.0}{\familydefault}{\mddefault}{\updefault}0:100 (100)}}}}
\put(10501,-4561){\makebox(0,0)[b]{\smash{{\SetFigFontNFSS{20}{24.0}{\familydefault}{\mddefault}{\updefault}0:100 (100)}}}}
\put(12601,839){\makebox(0,0)[b]{\smash{{\SetFigFontNFSS{20}{24.0}{\familydefault}{\mddefault}{\updefault}100:0 (100)}}}}
\put(12601,239){\makebox(0,0)[b]{\smash{{\SetFigFontNFSS{20}{24.0}{\familydefault}{\mddefault}{\updefault}100:0 (100)}}}}
\put(12601,-361){\makebox(0,0)[b]{\smash{{\SetFigFontNFSS{20}{24.0}{\familydefault}{\mddefault}{\updefault}100:0 (100)}}}}
\put(12601,-961){\makebox(0,0)[b]{\smash{{\SetFigFontNFSS{20}{24.0}{\familydefault}{\mddefault}{\updefault}100:0 (100)}}}}
\put(12601,-1561){\makebox(0,0)[b]{\smash{{\SetFigFontNFSS{20}{24.0}{\familydefault}{\mddefault}{\updefault}100:0 (100)}}}}
\put(12601,-2161){\makebox(0,0)[b]{\smash{{\SetFigFontNFSS{20}{24.0}{\familydefault}{\mddefault}{\updefault}100:0 (100)}}}}
\put(12601,-2761){\makebox(0,0)[b]{\smash{{\SetFigFontNFSS{20}{24.0}{\familydefault}{\mddefault}{\updefault}46:53 (97)}}}}
\put(12601,-3361){\makebox(0,0)[b]{\smash{{\SetFigFontNFSS{20}{24.0}{\familydefault}{\mddefault}{\updefault}0:100 (100)}}}}
\put(12601,-3961){\makebox(0,0)[b]{\smash{{\SetFigFontNFSS{20}{24.0}{\familydefault}{\mddefault}{\updefault}0:100 (100)}}}}
\put(12601,-4561){\makebox(0,0)[b]{\smash{{\SetFigFontNFSS{20}{24.0}{\familydefault}{\mddefault}{\updefault}0:100 (100)}}}}
\put(14701,839){\makebox(0,0)[b]{\smash{{\SetFigFontNFSS{20}{24.0}{\familydefault}{\mddefault}{\updefault}100:0 (100)}}}}
\put(14701,239){\makebox(0,0)[b]{\smash{{\SetFigFontNFSS{20}{24.0}{\familydefault}{\mddefault}{\updefault}100:0 (100)}}}}
\put(14701,-361){\makebox(0,0)[b]{\smash{{\SetFigFontNFSS{20}{24.0}{\familydefault}{\mddefault}{\updefault}100:0 (100)}}}}
\put(14701,-961){\makebox(0,0)[b]{\smash{{\SetFigFontNFSS{20}{24.0}{\familydefault}{\mddefault}{\updefault}100:0 (100)}}}}
\put(14701,-1561){\makebox(0,0)[b]{\smash{{\SetFigFontNFSS{20}{24.0}{\familydefault}{\mddefault}{\updefault}100:0 (100)}}}}
\put(14701,-2161){\makebox(0,0)[b]{\smash{{\SetFigFontNFSS{20}{24.0}{\familydefault}{\mddefault}{\updefault}100:0 (100)}}}}
\put(14701,-2761){\makebox(0,0)[b]{\smash{{\SetFigFontNFSS{20}{24.0}{\familydefault}{\mddefault}{\updefault}100:0 (100)}}}}
\put(14701,-3361){\makebox(0,0)[b]{\smash{{\SetFigFontNFSS{20}{24.0}{\familydefault}{\mddefault}{\updefault}50:49 (97)}}}}
\put(14701,-3961){\makebox(0,0)[b]{\smash{{\SetFigFontNFSS{20}{24.0}{\familydefault}{\mddefault}{\updefault}0:100 (100)}}}}
\put(14701,-4561){\makebox(0,0)[b]{\smash{{\SetFigFontNFSS{20}{24.0}{\familydefault}{\mddefault}{\updefault}0:100 (100)}}}}
\put(16801,839){\makebox(0,0)[b]{\smash{{\SetFigFontNFSS{20}{24.0}{\familydefault}{\mddefault}{\updefault}100:0 (100)}}}}
\put(16801,239){\makebox(0,0)[b]{\smash{{\SetFigFontNFSS{20}{24.0}{\familydefault}{\mddefault}{\updefault}100:0 (100)}}}}
\put(16801,-361){\makebox(0,0)[b]{\smash{{\SetFigFontNFSS{20}{24.0}{\familydefault}{\mddefault}{\updefault}100:0 (100)}}}}
\put(16801,-961){\makebox(0,0)[b]{\smash{{\SetFigFontNFSS{20}{24.0}{\familydefault}{\mddefault}{\updefault}100:0 (100)}}}}
\put(16801,-1561){\makebox(0,0)[b]{\smash{{\SetFigFontNFSS{20}{24.0}{\familydefault}{\mddefault}{\updefault}100:0 (100)}}}}
\put(16801,-2161){\makebox(0,0)[b]{\smash{{\SetFigFontNFSS{20}{24.0}{\familydefault}{\mddefault}{\updefault}100:0 (100)}}}}
\put(16801,-2761){\makebox(0,0)[b]{\smash{{\SetFigFontNFSS{20}{24.0}{\familydefault}{\mddefault}{\updefault}100:0 (100)}}}}
\put(16801,-3361){\makebox(0,0)[b]{\smash{{\SetFigFontNFSS{20}{24.0}{\familydefault}{\mddefault}{\updefault}100:0 (100)}}}}
\put(16801,-3961){\makebox(0,0)[b]{\smash{{\SetFigFontNFSS{20}{24.0}{\familydefault}{\mddefault}{\updefault}55:44 (97)}}}}
\put(16801,-4561){\makebox(0,0)[b]{\smash{{\SetFigFontNFSS{20}{24.0}{\familydefault}{\mddefault}{\updefault}0:100 (100)}}}}
\put(18901,839){\makebox(0,0)[b]{\smash{{\SetFigFontNFSS{20}{24.0}{\familydefault}{\mddefault}{\updefault}100:0 (100)}}}}
\put(18901,239){\makebox(0,0)[b]{\smash{{\SetFigFontNFSS{20}{24.0}{\familydefault}{\mddefault}{\updefault}100:0 (100)}}}}
\put(18901,-361){\makebox(0,0)[b]{\smash{{\SetFigFontNFSS{20}{24.0}{\familydefault}{\mddefault}{\updefault}100:0 (100)}}}}
\put(18901,-961){\makebox(0,0)[b]{\smash{{\SetFigFontNFSS{20}{24.0}{\familydefault}{\mddefault}{\updefault}100:0 (100)}}}}
\put(18901,-1561){\makebox(0,0)[b]{\smash{{\SetFigFontNFSS{20}{24.0}{\familydefault}{\mddefault}{\updefault}100:0 (100)}}}}
\put(18901,-2161){\makebox(0,0)[b]{\smash{{\SetFigFontNFSS{20}{24.0}{\familydefault}{\mddefault}{\updefault}100:0 (100)}}}}
\put(18901,-2761){\makebox(0,0)[b]{\smash{{\SetFigFontNFSS{20}{24.0}{\familydefault}{\mddefault}{\updefault}100:0 (100)}}}}
\put(18901,-3361){\makebox(0,0)[b]{\smash{{\SetFigFontNFSS{20}{24.0}{\familydefault}{\mddefault}{\updefault}100:0 (100)}}}}
\put(18901,-3961){\makebox(0,0)[b]{\smash{{\SetFigFontNFSS{20}{24.0}{\familydefault}{\mddefault}{\updefault}100:0 (100)}}}}
\put(18901,-4561){\makebox(0,0)[b]{\smash{{\SetFigFontNFSS{20}{24.0}{\familydefault}{\mddefault}{\updefault}51:48 (96)}}}}
\put(  1,-6361){\makebox(0,0)[b]{\smash{{\SetFigFontNFSS{20}{24.0}{\familydefault}{\mddefault}{\updefault}$\theta_1 = 0.00$}}}}
\put(2101,-6361){\makebox(0,0)[b]{\smash{{\SetFigFontNFSS{20}{24.0}{\familydefault}{\mddefault}{\updefault}$\theta_1 = 0.05$}}}}
\put(4201,-6361){\makebox(0,0)[b]{\smash{{\SetFigFontNFSS{20}{24.0}{\familydefault}{\mddefault}{\updefault}$\theta_1 = 0.10$}}}}
\put(6301,-6361){\makebox(0,0)[b]{\smash{{\SetFigFontNFSS{20}{24.0}{\familydefault}{\mddefault}{\updefault}$\theta_1 = 0.15$}}}}
\put(8401,-6361){\makebox(0,0)[b]{\smash{{\SetFigFontNFSS{20}{24.0}{\familydefault}{\mddefault}{\updefault}$\theta_1 = 0.20$}}}}
\put(10501,-6361){\makebox(0,0)[b]{\smash{{\SetFigFontNFSS{20}{24.0}{\familydefault}{\mddefault}{\updefault}$\theta_1 = 0.25$}}}}
\put(12601,-6361){\makebox(0,0)[b]{\smash{{\SetFigFontNFSS{20}{24.0}{\familydefault}{\mddefault}{\updefault}$\theta_1 = 0.30$}}}}
\put(14701,-6361){\makebox(0,0)[b]{\smash{{\SetFigFontNFSS{20}{24.0}{\familydefault}{\mddefault}{\updefault}$\theta_1 = 0.35$}}}}
\put(16801,-6361){\makebox(0,0)[b]{\smash{{\SetFigFontNFSS{20}{24.0}{\familydefault}{\mddefault}{\updefault}$\theta_1 = 0.40$}}}}
\put(18901,-6361){\makebox(0,0)[b]{\smash{{\SetFigFontNFSS{20}{24.0}{\familydefault}{\mddefault}{\updefault}$\theta_1 = 0.45$}}}}
\put(-2099,-6961){\makebox(0,0)[b]{\smash{{\SetFigFontNFSS{20}{24.0}{\familydefault}{\mddefault}{\updefault}$\theta_2 = 0.00$}}}}
\put(-2099,-7561){\makebox(0,0)[b]{\smash{{\SetFigFontNFSS{20}{24.0}{\familydefault}{\mddefault}{\updefault}$\theta_2 = 0.05$}}}}
\put(-2099,-8161){\makebox(0,0)[b]{\smash{{\SetFigFontNFSS{20}{24.0}{\familydefault}{\mddefault}{\updefault}$\theta_2 = 0.10$}}}}
\put(-2099,-8761){\makebox(0,0)[b]{\smash{{\SetFigFontNFSS{20}{24.0}{\familydefault}{\mddefault}{\updefault}$\theta_2 = 0.15$}}}}
\put(-2099,-9361){\makebox(0,0)[b]{\smash{{\SetFigFontNFSS{20}{24.0}{\familydefault}{\mddefault}{\updefault}$\theta_2 = 0.20$}}}}
\put(-2099,-9961){\makebox(0,0)[b]{\smash{{\SetFigFontNFSS{20}{24.0}{\familydefault}{\mddefault}{\updefault}$\theta_2 = 0.25$}}}}
\put(-2099,-10561){\makebox(0,0)[b]{\smash{{\SetFigFontNFSS{20}{24.0}{\familydefault}{\mddefault}{\updefault}$\theta_2 = 0.30$}}}}
\put(-2099,-11161){\makebox(0,0)[b]{\smash{{\SetFigFontNFSS{20}{24.0}{\familydefault}{\mddefault}{\updefault}$\theta_2 = 0.35$}}}}
\put(-2099,-11761){\makebox(0,0)[b]{\smash{{\SetFigFontNFSS{20}{24.0}{\familydefault}{\mddefault}{\updefault}$\theta_2 = 0.40$}}}}
\put(-2099,-12361){\makebox(0,0)[b]{\smash{{\SetFigFontNFSS{20}{24.0}{\familydefault}{\mddefault}{\updefault}$\theta_2 = 0.45$}}}}
\put(  1,-6961){\makebox(0,0)[b]{\smash{{\SetFigFontNFSS{20}{24.0}{\familydefault}{\mddefault}{\updefault}50:49 (55)}}}}
\put(  1,-7561){\makebox(0,0)[b]{\smash{{\SetFigFontNFSS{20}{24.0}{\familydefault}{\mddefault}{\updefault}46:53 (56)}}}}
\put(  1,-8161){\makebox(0,0)[b]{\smash{{\SetFigFontNFSS{20}{24.0}{\familydefault}{\mddefault}{\updefault}43:56 (57)}}}}
\put(  1,-8761){\makebox(0,0)[b]{\smash{{\SetFigFontNFSS{20}{24.0}{\familydefault}{\mddefault}{\updefault}38:61 (60)}}}}
\put(  1,-9361){\makebox(0,0)[b]{\smash{{\SetFigFontNFSS{20}{24.0}{\familydefault}{\mddefault}{\updefault}32:67 (65)}}}}
\put(  1,-9961){\makebox(0,0)[b]{\smash{{\SetFigFontNFSS{20}{24.0}{\familydefault}{\mddefault}{\updefault}23:76 (73)}}}}
\put(  1,-10561){\makebox(0,0)[b]{\smash{{\SetFigFontNFSS{20}{24.0}{\familydefault}{\mddefault}{\updefault}4:95 (94)}}}}
\put(  1,-11161){\makebox(0,0)[b]{\smash{{\SetFigFontNFSS{20}{24.0}{\familydefault}{\mddefault}{\updefault}0:100 (100)}}}}
\put(  1,-11761){\makebox(0,0)[b]{\smash{{\SetFigFontNFSS{20}{24.0}{\familydefault}{\mddefault}{\updefault}0:100 (100)}}}}
\put(  1,-12361){\makebox(0,0)[b]{\smash{{\SetFigFontNFSS{20}{24.0}{\familydefault}{\mddefault}{\updefault}0:100 (100)}}}}
\put(2101,-6961){\makebox(0,0)[b]{\smash{{\SetFigFontNFSS{20}{24.0}{\familydefault}{\mddefault}{\updefault}52:47 (56)}}}}
\put(2101,-7561){\makebox(0,0)[b]{\smash{{\SetFigFontNFSS{20}{24.0}{\familydefault}{\mddefault}{\updefault}50:49 (56)}}}}
\put(2101,-8161){\makebox(0,0)[b]{\smash{{\SetFigFontNFSS{20}{24.0}{\familydefault}{\mddefault}{\updefault}46:53 (58)}}}}
\put(2101,-8761){\makebox(0,0)[b]{\smash{{\SetFigFontNFSS{20}{24.0}{\familydefault}{\mddefault}{\updefault}41:58 (60)}}}}
\put(2101,-9361){\makebox(0,0)[b]{\smash{{\SetFigFontNFSS{20}{24.0}{\familydefault}{\mddefault}{\updefault}35:64 (63)}}}}
\put(2101,-9961){\makebox(0,0)[b]{\smash{{\SetFigFontNFSS{20}{24.0}{\familydefault}{\mddefault}{\updefault}26:73 (70)}}}}
\put(2101,-10561){\makebox(0,0)[b]{\smash{{\SetFigFontNFSS{20}{24.0}{\familydefault}{\mddefault}{\updefault}11:88 (86)}}}}
\put(2101,-11161){\makebox(0,0)[b]{\smash{{\SetFigFontNFSS{20}{24.0}{\familydefault}{\mddefault}{\updefault}0:100 (100)}}}}
\put(2101,-11761){\makebox(0,0)[b]{\smash{{\SetFigFontNFSS{20}{24.0}{\familydefault}{\mddefault}{\updefault}0:100 (100)}}}}
\put(2101,-12361){\makebox(0,0)[b]{\smash{{\SetFigFontNFSS{20}{24.0}{\familydefault}{\mddefault}{\updefault}0:100 (100)}}}}
\put(4201,-6961){\makebox(0,0)[b]{\smash{{\SetFigFontNFSS{20}{24.0}{\familydefault}{\mddefault}{\updefault}56:43 (58)}}}}
\put(4201,-7561){\makebox(0,0)[b]{\smash{{\SetFigFontNFSS{20}{24.0}{\familydefault}{\mddefault}{\updefault}53:46 (57)}}}}
\put(4201,-8161){\makebox(0,0)[b]{\smash{{\SetFigFontNFSS{20}{24.0}{\familydefault}{\mddefault}{\updefault}50:49 (58)}}}}
\put(4201,-8761){\makebox(0,0)[b]{\smash{{\SetFigFontNFSS{20}{24.0}{\familydefault}{\mddefault}{\updefault}45:54 (59)}}}}
\put(4201,-9361){\makebox(0,0)[b]{\smash{{\SetFigFontNFSS{20}{24.0}{\familydefault}{\mddefault}{\updefault}39:60 (62)}}}}
\put(4201,-9961){\makebox(0,0)[b]{\smash{{\SetFigFontNFSS{20}{24.0}{\familydefault}{\mddefault}{\updefault}31:68 (67)}}}}
\put(4201,-10561){\makebox(0,0)[b]{\smash{{\SetFigFontNFSS{20}{24.0}{\familydefault}{\mddefault}{\updefault}17:82 (80)}}}}
\put(4201,-11161){\makebox(0,0)[b]{\smash{{\SetFigFontNFSS{20}{24.0}{\familydefault}{\mddefault}{\updefault}0:100 (100)}}}}
\put(4201,-11761){\makebox(0,0)[b]{\smash{{\SetFigFontNFSS{20}{24.0}{\familydefault}{\mddefault}{\updefault}0:100 (100)}}}}
\put(4201,-12361){\makebox(0,0)[b]{\smash{{\SetFigFontNFSS{20}{24.0}{\familydefault}{\mddefault}{\updefault}0:100 (100)}}}}
\put(6301,-6961){\makebox(0,0)[b]{\smash{{\SetFigFontNFSS{20}{24.0}{\familydefault}{\mddefault}{\updefault}61:38 (60)}}}}
\put(6301,-7561){\makebox(0,0)[b]{\smash{{\SetFigFontNFSS{20}{24.0}{\familydefault}{\mddefault}{\updefault}58:41 (60)}}}}
\put(6301,-8161){\makebox(0,0)[b]{\smash{{\SetFigFontNFSS{20}{24.0}{\familydefault}{\mddefault}{\updefault}54:45 (59)}}}}
\put(6301,-8761){\makebox(0,0)[b]{\smash{{\SetFigFontNFSS{20}{24.0}{\familydefault}{\mddefault}{\updefault}49:50 (60)}}}}
\put(6301,-9361){\makebox(0,0)[b]{\smash{{\SetFigFontNFSS{20}{24.0}{\familydefault}{\mddefault}{\updefault}43:56 (61)}}}}
\put(6301,-9961){\makebox(0,0)[b]{\smash{{\SetFigFontNFSS{20}{24.0}{\familydefault}{\mddefault}{\updefault}36:63 (65)}}}}
\put(6301,-10561){\makebox(0,0)[b]{\smash{{\SetFigFontNFSS{20}{24.0}{\familydefault}{\mddefault}{\updefault}24:75 (74)}}}}
\put(6301,-11161){\makebox(0,0)[b]{\smash{{\SetFigFontNFSS{20}{24.0}{\familydefault}{\mddefault}{\updefault}0:100 (100)}}}}
\put(6301,-11761){\makebox(0,0)[b]{\smash{{\SetFigFontNFSS{20}{24.0}{\familydefault}{\mddefault}{\updefault}0:100 (100)}}}}
\put(6301,-12361){\makebox(0,0)[b]{\smash{{\SetFigFontNFSS{20}{24.0}{\familydefault}{\mddefault}{\updefault}0:100 (100)}}}}
\put(8401,-6961){\makebox(0,0)[b]{\smash{{\SetFigFontNFSS{20}{24.0}{\familydefault}{\mddefault}{\updefault}67:32 (65)}}}}
\put(8401,-7561){\makebox(0,0)[b]{\smash{{\SetFigFontNFSS{20}{24.0}{\familydefault}{\mddefault}{\updefault}64:35 (63)}}}}
\put(8401,-8161){\makebox(0,0)[b]{\smash{{\SetFigFontNFSS{20}{24.0}{\familydefault}{\mddefault}{\updefault}60:39 (62)}}}}
\put(8401,-8761){\makebox(0,0)[b]{\smash{{\SetFigFontNFSS{20}{24.0}{\familydefault}{\mddefault}{\updefault}56:43 (62)}}}}
\put(8401,-9361){\makebox(0,0)[b]{\smash{{\SetFigFontNFSS{20}{24.0}{\familydefault}{\mddefault}{\updefault}50:49 (62)}}}}
\put(8401,-9961){\makebox(0,0)[b]{\smash{{\SetFigFontNFSS{20}{24.0}{\familydefault}{\mddefault}{\updefault}41:58 (64)}}}}
\put(8401,-10561){\makebox(0,0)[b]{\smash{{\SetFigFontNFSS{20}{24.0}{\familydefault}{\mddefault}{\updefault}30:69 (70)}}}}
\put(8401,-11161){\makebox(0,0)[b]{\smash{{\SetFigFontNFSS{20}{24.0}{\familydefault}{\mddefault}{\updefault}6:93 (92)}}}}
\put(8401,-11761){\makebox(0,0)[b]{\smash{{\SetFigFontNFSS{20}{24.0}{\familydefault}{\mddefault}{\updefault}0:100 (100)}}}}
\put(8401,-12361){\makebox(0,0)[b]{\smash{{\SetFigFontNFSS{20}{24.0}{\familydefault}{\mddefault}{\updefault}0:100 (100)}}}}
\put(10501,-6961){\makebox(0,0)[b]{\smash{{\SetFigFontNFSS{20}{24.0}{\familydefault}{\mddefault}{\updefault}76:23 (73)}}}}
\put(10501,-7561){\makebox(0,0)[b]{\smash{{\SetFigFontNFSS{20}{24.0}{\familydefault}{\mddefault}{\updefault}73:26 (70)}}}}
\put(10501,-8161){\makebox(0,0)[b]{\smash{{\SetFigFontNFSS{20}{24.0}{\familydefault}{\mddefault}{\updefault}68:31 (67)}}}}
\put(10501,-8761){\makebox(0,0)[b]{\smash{{\SetFigFontNFSS{20}{24.0}{\familydefault}{\mddefault}{\updefault}63:36 (65)}}}}
\put(10501,-9361){\makebox(0,0)[b]{\smash{{\SetFigFontNFSS{20}{24.0}{\familydefault}{\mddefault}{\updefault}57:42 (64)}}}}
\put(10501,-9961){\makebox(0,0)[b]{\smash{{\SetFigFontNFSS{20}{24.0}{\familydefault}{\mddefault}{\updefault}49:50 (64)}}}}
\put(10501,-10561){\makebox(0,0)[b]{\smash{{\SetFigFontNFSS{20}{24.0}{\familydefault}{\mddefault}{\updefault}38:61 (67)}}}}
\put(10501,-11161){\makebox(0,0)[b]{\smash{{\SetFigFontNFSS{20}{24.0}{\familydefault}{\mddefault}{\updefault}20:79 (80)}}}}
\put(10501,-11761){\makebox(0,0)[b]{\smash{{\SetFigFontNFSS{20}{24.0}{\familydefault}{\mddefault}{\updefault}0:100 (100)}}}}
\put(10501,-12361){\makebox(0,0)[b]{\smash{{\SetFigFontNFSS{20}{24.0}{\familydefault}{\mddefault}{\updefault}0:100 (100)}}}}
\put(12601,-6961){\makebox(0,0)[b]{\smash{{\SetFigFontNFSS{20}{24.0}{\familydefault}{\mddefault}{\updefault}95:4 (94)}}}}
\put(12601,-7561){\makebox(0,0)[b]{\smash{{\SetFigFontNFSS{20}{24.0}{\familydefault}{\mddefault}{\updefault}88:11 (86)}}}}
\put(12601,-8161){\makebox(0,0)[b]{\smash{{\SetFigFontNFSS{20}{24.0}{\familydefault}{\mddefault}{\updefault}82:17 (80)}}}}
\put(12601,-8761){\makebox(0,0)[b]{\smash{{\SetFigFontNFSS{20}{24.0}{\familydefault}{\mddefault}{\updefault}76:23 (75)}}}}
\put(12601,-9361){\makebox(0,0)[b]{\smash{{\SetFigFontNFSS{20}{24.0}{\familydefault}{\mddefault}{\updefault}69:30 (70)}}}}
\put(12601,-9961){\makebox(0,0)[b]{\smash{{\SetFigFontNFSS{20}{24.0}{\familydefault}{\mddefault}{\updefault}61:38 (67)}}}}
\put(12601,-10561){\makebox(0,0)[b]{\smash{{\SetFigFontNFSS{20}{24.0}{\familydefault}{\mddefault}{\updefault}50:49 (66)}}}}
\put(12601,-11161){\makebox(0,0)[b]{\smash{{\SetFigFontNFSS{20}{24.0}{\familydefault}{\mddefault}{\updefault}33:66 (71)}}}}
\put(12601,-11761){\makebox(0,0)[b]{\smash{{\SetFigFontNFSS{20}{24.0}{\familydefault}{\mddefault}{\updefault}0:100 (100)}}}}
\put(12601,-12361){\makebox(0,0)[b]{\smash{{\SetFigFontNFSS{20}{24.0}{\familydefault}{\mddefault}{\updefault}0:100 (100)}}}}
\put(14701,-6961){\makebox(0,0)[b]{\smash{{\SetFigFontNFSS{20}{24.0}{\familydefault}{\mddefault}{\updefault}100:0 (100)}}}}
\put(14701,-7561){\makebox(0,0)[b]{\smash{{\SetFigFontNFSS{20}{24.0}{\familydefault}{\mddefault}{\updefault}100:0 (100)}}}}
\put(14701,-8161){\makebox(0,0)[b]{\smash{{\SetFigFontNFSS{20}{24.0}{\familydefault}{\mddefault}{\updefault}100:0 (100)}}}}
\put(14701,-8761){\makebox(0,0)[b]{\smash{{\SetFigFontNFSS{20}{24.0}{\familydefault}{\mddefault}{\updefault}100:0 (100)}}}}
\put(14701,-9361){\makebox(0,0)[b]{\smash{{\SetFigFontNFSS{20}{24.0}{\familydefault}{\mddefault}{\updefault}93:6 (93)}}}}
\put(14701,-9961){\makebox(0,0)[b]{\smash{{\SetFigFontNFSS{20}{24.0}{\familydefault}{\mddefault}{\updefault}79:20 (79)}}}}
\put(14701,-10561){\makebox(0,0)[b]{\smash{{\SetFigFontNFSS{20}{24.0}{\familydefault}{\mddefault}{\updefault}66:33 (71)}}}}
\put(14701,-11161){\makebox(0,0)[b]{\smash{{\SetFigFontNFSS{20}{24.0}{\familydefault}{\mddefault}{\updefault}50:49 (68)}}}}
\put(14701,-11761){\makebox(0,0)[b]{\smash{{\SetFigFontNFSS{20}{24.0}{\familydefault}{\mddefault}{\updefault}19:80 (82)}}}}
\put(14701,-12361){\makebox(0,0)[b]{\smash{{\SetFigFontNFSS{20}{24.0}{\familydefault}{\mddefault}{\updefault}0:100 (100)}}}}
\put(16801,-6961){\makebox(0,0)[b]{\smash{{\SetFigFontNFSS{20}{24.0}{\familydefault}{\mddefault}{\updefault}100:0 (100)}}}}
\put(16801,-7561){\makebox(0,0)[b]{\smash{{\SetFigFontNFSS{20}{24.0}{\familydefault}{\mddefault}{\updefault}100:0 (100)}}}}
\put(16801,-8161){\makebox(0,0)[b]{\smash{{\SetFigFontNFSS{20}{24.0}{\familydefault}{\mddefault}{\updefault}100:0 (100)}}}}
\put(16801,-8761){\makebox(0,0)[b]{\smash{{\SetFigFontNFSS{20}{24.0}{\familydefault}{\mddefault}{\updefault}100:0 (100)}}}}
\put(16801,-9361){\makebox(0,0)[b]{\smash{{\SetFigFontNFSS{20}{24.0}{\familydefault}{\mddefault}{\updefault}100:0 (100)}}}}
\put(16801,-9961){\makebox(0,0)[b]{\smash{{\SetFigFontNFSS{20}{24.0}{\familydefault}{\mddefault}{\updefault}100:0 (100)}}}}
\put(16801,-10561){\makebox(0,0)[b]{\smash{{\SetFigFontNFSS{20}{24.0}{\familydefault}{\mddefault}{\updefault}100:0 (100)}}}}
\put(16801,-11161){\makebox(0,0)[b]{\smash{{\SetFigFontNFSS{20}{24.0}{\familydefault}{\mddefault}{\updefault}81:18 (83)}}}}
\put(16801,-11761){\makebox(0,0)[b]{\smash{{\SetFigFontNFSS{20}{24.0}{\familydefault}{\mddefault}{\updefault}49:50 (72)}}}}
\put(16801,-12361){\makebox(0,0)[b]{\smash{{\SetFigFontNFSS{20}{24.0}{\familydefault}{\mddefault}{\updefault}0:100 (100)}}}}
\put(18901,-6961){\makebox(0,0)[b]{\smash{{\SetFigFontNFSS{20}{24.0}{\familydefault}{\mddefault}{\updefault}100:0 (100)}}}}
\put(18901,-7561){\makebox(0,0)[b]{\smash{{\SetFigFontNFSS{20}{24.0}{\familydefault}{\mddefault}{\updefault}100:0 (100)}}}}
\put(18901,-8161){\makebox(0,0)[b]{\smash{{\SetFigFontNFSS{20}{24.0}{\familydefault}{\mddefault}{\updefault}100:0 (100)}}}}
\put(18901,-8761){\makebox(0,0)[b]{\smash{{\SetFigFontNFSS{20}{24.0}{\familydefault}{\mddefault}{\updefault}100:0 (100)}}}}
\put(18901,-9361){\makebox(0,0)[b]{\smash{{\SetFigFontNFSS{20}{24.0}{\familydefault}{\mddefault}{\updefault}100:0 (100)}}}}
\put(18901,-9961){\makebox(0,0)[b]{\smash{{\SetFigFontNFSS{20}{24.0}{\familydefault}{\mddefault}{\updefault}100:0 (100)}}}}
\put(18901,-10561){\makebox(0,0)[b]{\smash{{\SetFigFontNFSS{20}{24.0}{\familydefault}{\mddefault}{\updefault}100:0 (100)}}}}
\put(18901,-11161){\makebox(0,0)[b]{\smash{{\SetFigFontNFSS{20}{24.0}{\familydefault}{\mddefault}{\updefault}100:0 (100)}}}}
\put(18901,-11761){\makebox(0,0)[b]{\smash{{\SetFigFontNFSS{20}{24.0}{\familydefault}{\mddefault}{\updefault}100:0 (100)}}}}
\put(18901,-12361){\makebox(0,0)[b]{\smash{{\SetFigFontNFSS{20}{24.0}{\familydefault}{\mddefault}{\updefault}51:48 (76)}}}}
\put(  1,839){\makebox(0,0)[b]{\smash{{\SetFigFontNFSS{20}{24.0}{\familydefault}{\mddefault}{\updefault}51:48 (86)}}}}
\end{picture}%

%% file: diagrams-two-type.pstex_t
\begin{picture}(0,0)%
\includegraphics{diagrams-two-type.pstex}%
\end{picture}%
\setlength{\unitlength}{3947sp}%
\begingroup\makeatletter\ifx\SetFigFontNFSS\undefined%
\gdef\SetFigFontNFSS#1#2#3#4#5{%
  \reset@font\fontsize{#1}{#2pt}%
  \fontfamily{#3}\fontseries{#4}\fontshape{#5}%
  \selectfont}%
\fi\endgroup%
\begin{picture}(16211,8392)(-344,-6986)
\put(  1,-6886){\makebox(0,0)[b]{\smash{{\SetFigFontNFSS{14}{16.8}{\rmdefault}{\bfdefault}{\updefault}TVM}}}}
\put(7201,1139){\makebox(0,0)[b]{\smash{{\SetFigFontNFSS{14}{16.8}{\rmdefault}{\bfdefault}{\updefault}static}}}}
\put(1201,-5161){\makebox(0,0)[b]{\smash{{\SetFigFontNFSS{20}{24.0}{\rmdefault}{\bfdefault}{\updefault}coexist.}}}}
\put(4951,-1111){\rotatebox{45.0}{\makebox(0,0)[b]{\smash{{\SetFigFontNFSS{12}{14.4}{\familydefault}{\mddefault}{\updefault}clustering}}}}}
\put(5401,-4561){\makebox(0,0)[b]{\smash{{\SetFigFontNFSS{20}{24.0}{\rmdefault}{\bfdefault}{\updefault}type 1}}}}
\put(3901,239){\makebox(0,0)[b]{\smash{{\SetFigFontNFSS{20}{24.0}{\familydefault}{\mddefault}{\updefault}\ref{th:spatial-pure-birth}}}}}
\put(7201,-6886){\makebox(0,0)[b]{\smash{{\SetFigFontNFSS{20}{24.0}{\familydefault}{\mddefault}{\updefault}$\theta_1 = 1$}}}}
\put(-299,764){\makebox(0,0)[rb]{\smash{{\SetFigFontNFSS{20}{24.0}{\familydefault}{\mddefault}{\updefault}$\theta_2 = 1$}}}}
\put(-299,-2836){\makebox(0,0)[rb]{\smash{{\SetFigFontNFSS{20}{24.0}{\familydefault}{\mddefault}{\updefault}$\theta_2 = 1/2$}}}}
\put(-299,-6436){\makebox(0,0)[rb]{\smash{{\SetFigFontNFSS{20}{24.0}{\familydefault}{\mddefault}{\updefault}$\theta_2 = 0$}}}}
\put(3601,-6886){\makebox(0,0)[b]{\smash{{\SetFigFontNFSS{20}{24.0}{\familydefault}{\mddefault}{\updefault}$\theta_1 = 1/2$}}}}
\put(1801,-961){\makebox(0,0)[b]{\smash{{\SetFigFontNFSS{20}{24.0}{\rmdefault}{\bfdefault}{\updefault}type 2}}}}
\put(13501,-1561){\makebox(0,0)[b]{\smash{{\SetFigFontNFSS{12}{14.4}{\familydefault}{\mddefault}{\updefault}unstable}}}}
\put(13501,-1336){\makebox(0,0)[b]{\smash{{\SetFigFontNFSS{12}{14.4}{\familydefault}{\mddefault}{\updefault}interior fixed point}}}}
\put(9901,-4936){\makebox(0,0)[b]{\smash{{\SetFigFontNFSS{12}{14.4}{\familydefault}{\mddefault}{\updefault}interior fixed point}}}}
\put(9901,-1336){\makebox(0,0)[b]{\smash{{\SetFigFontNFSS{12}{14.4}{\familydefault}{\mddefault}{\updefault}no interior}}}}
\put(9901,-1561){\makebox(0,0)[b]{\smash{{\SetFigFontNFSS{12}{14.4}{\familydefault}{\mddefault}{\updefault}fixed point}}}}
\put(13501,-4936){\makebox(0,0)[b]{\smash{{\SetFigFontNFSS{12}{14.4}{\familydefault}{\mddefault}{\updefault}no interior}}}}
\put(13501,-5161){\makebox(0,0)[b]{\smash{{\SetFigFontNFSS{12}{14.4}{\familydefault}{\mddefault}{\updefault}fixed point}}}}
\put(9901,1139){\makebox(0,0)[b]{\smash{{\SetFigFontNFSS{20}{24.0}{\familydefault}{\mddefault}{\updefault}$e_1$ unstable}}}}
\put(13501,1139){\makebox(0,0)[b]{\smash{{\SetFigFontNFSS{20}{24.0}{\familydefault}{\mddefault}{\updefault}$e_1$ locally stable}}}}
\put(9901,-5161){\makebox(0,0)[b]{\smash{{\SetFigFontNFSS{12}{14.4}{\familydefault}{\mddefault}{\updefault}globally stable}}}}
\put(15601,-961){\rotatebox{270.0}{\makebox(0,0)[b]{\smash{{\SetFigFontNFSS{20}{24.0}{\familydefault}{\mddefault}{\updefault}$e_2$ locally stable}}}}}
\put(15601,-4561){\rotatebox{270.0}{\makebox(0,0)[b]{\smash{{\SetFigFontNFSS{20}{24.0}{\familydefault}{\mddefault}{\updefault}$e_2$ unstable}}}}}
\put(9901,-961){\makebox(0,0)[b]{\smash{{\SetFigFontNFSS{20}{24.0}{\rmdefault}{\bfdefault}{\updefault}type 2}}}}
\put(13501,-4561){\makebox(0,0)[b]{\smash{{\SetFigFontNFSS{20}{24.0}{\rmdefault}{\bfdefault}{\updefault}type 1}}}}
\put(13501,-961){\makebox(0,0)[b]{\smash{{\SetFigFontNFSS{20}{24.0}{\rmdefault}{\bfdefault}{\updefault}bistability}}}}
\put(11701,-6886){\makebox(0,0)[b]{\smash{{\SetFigFontNFSS{20}{24.0}{\familydefault}{\mddefault}{\updefault}$\theta_1 = 1/2$}}}}
\put(15301,-6886){\makebox(0,0)[b]{\smash{{\SetFigFontNFSS{20}{24.0}{\familydefault}{\mddefault}{\updefault}$\theta_1 = 1$}}}}
\put(9901,-4561){\makebox(0,0)[b]{\smash{{\SetFigFontNFSS{20}{24.0}{\rmdefault}{\bfdefault}{\updefault}coexistence}}}}
\put(3151,-2611){\makebox(0,0)[b]{\smash{{\SetFigFontNFSS{14}{16.8}{\rmdefault}{\bfdefault}{\updefault}VM}}}}
\put(301,-3286){\makebox(0,0)[b]{\smash{{\SetFigFontNFSS{20}{24.0}{\familydefault}{\mddefault}{\updefault}\ref{th:spatial-reduction}}}}}
\put(3226,-6136){\makebox(0,0)[b]{\smash{{\SetFigFontNFSS{20}{24.0}{\familydefault}{\mddefault}{\updefault}\ref{th:spatial-reduction}}}}}
\put(301,239){\makebox(0,0)[b]{\smash{{\SetFigFontNFSS{20}{24.0}{\familydefault}{\mddefault}{\updefault}\ref{th:spatial-invasion}}}}}
\put(6901,-6061){\makebox(0,0)[b]{\smash{{\SetFigFontNFSS{20}{24.0}{\familydefault}{\mddefault}{\updefault}\ref{th:spatial-invasion}}}}}
\put(6901,-2461){\makebox(0,0)[b]{\smash{{\SetFigFontNFSS{20}{24.0}{\familydefault}{\mddefault}{\updefault}\ref{th:spatial-pure-birth}}}}}
\put(301,-6136){\makebox(0,0)[b]{\smash{{\SetFigFontNFSS{20}{24.0}{\familydefault}{\mddefault}{\updefault}\ref{th:spatial-coexistence}}}}}
\end{picture}%

%% file: tab-matrix_8.pstex_t
\begin{picture}(0,0)%
\includegraphics{tab-matrix_8.pstex}%
\end{picture}%
\setlength{\unitlength}{3947sp}%
\begingroup\makeatletter\ifx\SetFigFontNFSS\undefined%
\gdef\SetFigFontNFSS#1#2#3#4#5{%
  \reset@font\fontsize{#1}{#2pt}%
  \fontfamily{#3}\fontseries{#4}\fontshape{#5}%
  \selectfont}%
\fi\endgroup%
\begin{picture}(23178,22869)(-3086,-21043)
\put(8401,-13111){\makebox(0,0)[b]{\smash{{\SetFigFontNFSS{20}{24.0}{\familydefault}{\mddefault}{\updefault}Densities and clustering coefficient when the interactions are described by the matrix $M_8$ with $\theta_3 = 0$}}}}
\put(8401,-20911){\makebox(0,0)[b]{\smash{{\SetFigFontNFSS{20}{24.0}{\familydefault}{\mddefault}{\updefault}Densities and clustering coefficient when the interactions are described by the matrix $M_8$ with $\theta_3 = 8/27$}}}}
\put(  1,-14161){\makebox(0,0)[b]{\smash{{\SetFigFontNFSS{20}{24.0}{\familydefault}{\mddefault}{\updefault}$\theta_1 = 0/27$}}}}
\put(2101,-14161){\makebox(0,0)[b]{\smash{{\SetFigFontNFSS{20}{24.0}{\familydefault}{\mddefault}{\updefault}$\theta_1 = 1/27$}}}}
\put(4201,-14161){\makebox(0,0)[b]{\smash{{\SetFigFontNFSS{20}{24.0}{\familydefault}{\mddefault}{\updefault}$\theta_1 = 2/27$}}}}
\put(6301,-14161){\makebox(0,0)[b]{\smash{{\SetFigFontNFSS{20}{24.0}{\familydefault}{\mddefault}{\updefault}$\theta_1 = 3/27$}}}}
\put(8401,-14161){\makebox(0,0)[b]{\smash{{\SetFigFontNFSS{20}{24.0}{\familydefault}{\mddefault}{\updefault}$\theta_1 = 4/27$}}}}
\put(10501,-14161){\makebox(0,0)[b]{\smash{{\SetFigFontNFSS{20}{24.0}{\familydefault}{\mddefault}{\updefault}$\theta_1 = 5/27$}}}}
\put(12601,-14161){\makebox(0,0)[b]{\smash{{\SetFigFontNFSS{20}{24.0}{\familydefault}{\mddefault}{\updefault}$\theta_1 = 6/27$}}}}
\put(14701,-14161){\makebox(0,0)[b]{\smash{{\SetFigFontNFSS{20}{24.0}{\familydefault}{\mddefault}{\updefault}$\theta_1 = 7/27$}}}}
\put(16801,-14161){\makebox(0,0)[b]{\smash{{\SetFigFontNFSS{20}{24.0}{\familydefault}{\mddefault}{\updefault}$\theta_1 = 8/27$}}}}
\put(18901,-14161){\makebox(0,0)[b]{\smash{{\SetFigFontNFSS{20}{24.0}{\familydefault}{\mddefault}{\updefault}$\theta_1 = 9/27$}}}}
\put(-2099,-14761){\makebox(0,0)[b]{\smash{{\SetFigFontNFSS{20}{24.0}{\familydefault}{\mddefault}{\updefault}$\theta_2 = 0/27$}}}}
\put(-2099,-15361){\makebox(0,0)[b]{\smash{{\SetFigFontNFSS{20}{24.0}{\familydefault}{\mddefault}{\updefault}$\theta_2 = 1/27$}}}}
\put(-2099,-15961){\makebox(0,0)[b]{\smash{{\SetFigFontNFSS{20}{24.0}{\familydefault}{\mddefault}{\updefault}$\theta_2 = 2/27$}}}}
\put(-2099,-16561){\makebox(0,0)[b]{\smash{{\SetFigFontNFSS{20}{24.0}{\familydefault}{\mddefault}{\updefault}$\theta_2 = 3/27$}}}}
\put(-2099,-17161){\makebox(0,0)[b]{\smash{{\SetFigFontNFSS{20}{24.0}{\familydefault}{\mddefault}{\updefault}$\theta_2 = 4/27$}}}}
\put(-2099,-17761){\makebox(0,0)[b]{\smash{{\SetFigFontNFSS{20}{24.0}{\familydefault}{\mddefault}{\updefault}$\theta_2 = 5/27$}}}}
\put(-2099,-18361){\makebox(0,0)[b]{\smash{{\SetFigFontNFSS{20}{24.0}{\familydefault}{\mddefault}{\updefault}$\theta_2 = 6/27$}}}}
\put(-2099,-18961){\makebox(0,0)[b]{\smash{{\SetFigFontNFSS{20}{24.0}{\familydefault}{\mddefault}{\updefault}$\theta_2 = 7/27$}}}}
\put(-2099,-19561){\makebox(0,0)[b]{\smash{{\SetFigFontNFSS{20}{24.0}{\familydefault}{\mddefault}{\updefault}$\theta_2 = 8/27$}}}}
\put(-2099,-20161){\makebox(0,0)[b]{\smash{{\SetFigFontNFSS{20}{24.0}{\familydefault}{\mddefault}{\updefault}$\theta_2 = 9/27$}}}}
\put(  1,-14761){\makebox(0,0)[b]{\smash{{\SetFigFontNFSS{20}{24.0}{\familydefault}{\mddefault}{\updefault}0:0:100 (100)}}}}
\put(  1,-15361){\makebox(0,0)[b]{\smash{{\SetFigFontNFSS{20}{24.0}{\familydefault}{\mddefault}{\updefault}0:0:100 (100)}}}}
\put(  1,-15961){\makebox(0,0)[b]{\smash{{\SetFigFontNFSS{20}{24.0}{\familydefault}{\mddefault}{\updefault}0:0:100 (100)}}}}
\put(  1,-16561){\makebox(0,0)[b]{\smash{{\SetFigFontNFSS{20}{24.0}{\familydefault}{\mddefault}{\updefault}0:0:100 (100)}}}}
\put(  1,-17161){\makebox(0,0)[b]{\smash{{\SetFigFontNFSS{20}{24.0}{\familydefault}{\mddefault}{\updefault}0:0:100 (100)}}}}
\put(  1,-17761){\makebox(0,0)[b]{\smash{{\SetFigFontNFSS{20}{24.0}{\familydefault}{\mddefault}{\updefault}0:0:100 (100)}}}}
\put(  1,-18361){\makebox(0,0)[b]{\smash{{\SetFigFontNFSS{20}{24.0}{\familydefault}{\mddefault}{\updefault}0:0:100 (100)}}}}
\put(  1,-18961){\makebox(0,0)[b]{\smash{{\SetFigFontNFSS{20}{24.0}{\familydefault}{\mddefault}{\updefault}0:0:100 (100)}}}}
\put(  1,-19561){\makebox(0,0)[b]{\smash{{\SetFigFontNFSS{20}{24.0}{\familydefault}{\mddefault}{\updefault}0:49:50 (76)}}}}
\put(  1,-20161){\makebox(0,0)[b]{\smash{{\SetFigFontNFSS{20}{24.0}{\familydefault}{\mddefault}{\updefault}0:100:0 (100)}}}}
\put(2101,-14761){\makebox(0,0)[b]{\smash{{\SetFigFontNFSS{20}{24.0}{\familydefault}{\mddefault}{\updefault}0:0:100 (100)}}}}
\put(2101,-15361){\makebox(0,0)[b]{\smash{{\SetFigFontNFSS{20}{24.0}{\familydefault}{\mddefault}{\updefault}0:0:100 (100)}}}}
\put(2101,-15961){\makebox(0,0)[b]{\smash{{\SetFigFontNFSS{20}{24.0}{\familydefault}{\mddefault}{\updefault}0:0:100 (100)}}}}
\put(2101,-16561){\makebox(0,0)[b]{\smash{{\SetFigFontNFSS{20}{24.0}{\familydefault}{\mddefault}{\updefault}0:0:100 (100)}}}}
\put(2101,-17161){\makebox(0,0)[b]{\smash{{\SetFigFontNFSS{20}{24.0}{\familydefault}{\mddefault}{\updefault}0:0:100 (100)}}}}
\put(2101,-17761){\makebox(0,0)[b]{\smash{{\SetFigFontNFSS{20}{24.0}{\familydefault}{\mddefault}{\updefault}0:0:100 (100)}}}}
\put(2101,-18361){\makebox(0,0)[b]{\smash{{\SetFigFontNFSS{20}{24.0}{\familydefault}{\mddefault}{\updefault}0:0:100 (100)}}}}
\put(2101,-18961){\makebox(0,0)[b]{\smash{{\SetFigFontNFSS{20}{24.0}{\familydefault}{\mddefault}{\updefault}0:0:100 (100)}}}}
\put(2101,-19561){\makebox(0,0)[b]{\smash{{\SetFigFontNFSS{20}{24.0}{\familydefault}{\mddefault}{\updefault}0:50:49 (77)}}}}
\put(2101,-20161){\makebox(0,0)[b]{\smash{{\SetFigFontNFSS{20}{24.0}{\familydefault}{\mddefault}{\updefault}0:100:0 (100)}}}}
\put(4201,-14761){\makebox(0,0)[b]{\smash{{\SetFigFontNFSS{20}{24.0}{\familydefault}{\mddefault}{\updefault}0:0:100 (100)}}}}
\put(4201,-15361){\makebox(0,0)[b]{\smash{{\SetFigFontNFSS{20}{24.0}{\familydefault}{\mddefault}{\updefault}0:0:100 (100)}}}}
\put(4201,-15961){\makebox(0,0)[b]{\smash{{\SetFigFontNFSS{20}{24.0}{\familydefault}{\mddefault}{\updefault}0:0:100 (100)}}}}
\put(4201,-16561){\makebox(0,0)[b]{\smash{{\SetFigFontNFSS{20}{24.0}{\familydefault}{\mddefault}{\updefault}0:0:100 (100)}}}}
\put(4201,-17161){\makebox(0,0)[b]{\smash{{\SetFigFontNFSS{20}{24.0}{\familydefault}{\mddefault}{\updefault}0:0:100 (100)}}}}
\put(4201,-17761){\makebox(0,0)[b]{\smash{{\SetFigFontNFSS{20}{24.0}{\familydefault}{\mddefault}{\updefault}0:0:100 (100)}}}}
\put(4201,-18361){\makebox(0,0)[b]{\smash{{\SetFigFontNFSS{20}{24.0}{\familydefault}{\mddefault}{\updefault}0:0:100 (100)}}}}
\put(4201,-18961){\makebox(0,0)[b]{\smash{{\SetFigFontNFSS{20}{24.0}{\familydefault}{\mddefault}{\updefault}0:0:100 (100)}}}}
\put(4201,-19561){\makebox(0,0)[b]{\smash{{\SetFigFontNFSS{20}{24.0}{\familydefault}{\mddefault}{\updefault}0:48:51 (77)}}}}
\put(4201,-20161){\makebox(0,0)[b]{\smash{{\SetFigFontNFSS{20}{24.0}{\familydefault}{\mddefault}{\updefault}0:100:0 (100)}}}}
\put(6301,-14761){\makebox(0,0)[b]{\smash{{\SetFigFontNFSS{20}{24.0}{\familydefault}{\mddefault}{\updefault}0:0:100 (100)}}}}
\put(6301,-15361){\makebox(0,0)[b]{\smash{{\SetFigFontNFSS{20}{24.0}{\familydefault}{\mddefault}{\updefault}0:0:100 (100)}}}}
\put(6301,-15961){\makebox(0,0)[b]{\smash{{\SetFigFontNFSS{20}{24.0}{\familydefault}{\mddefault}{\updefault}0:0:100 (100)}}}}
\put(6301,-16561){\makebox(0,0)[b]{\smash{{\SetFigFontNFSS{20}{24.0}{\familydefault}{\mddefault}{\updefault}0:0:100 (100)}}}}
\put(6301,-17161){\makebox(0,0)[b]{\smash{{\SetFigFontNFSS{20}{24.0}{\familydefault}{\mddefault}{\updefault}0:0:100 (100)}}}}
\put(6301,-17761){\makebox(0,0)[b]{\smash{{\SetFigFontNFSS{20}{24.0}{\familydefault}{\mddefault}{\updefault}0:0:100 (100)}}}}
\put(6301,-18361){\makebox(0,0)[b]{\smash{{\SetFigFontNFSS{20}{24.0}{\familydefault}{\mddefault}{\updefault}0:0:100 (100)}}}}
\put(6301,-18961){\makebox(0,0)[b]{\smash{{\SetFigFontNFSS{20}{24.0}{\familydefault}{\mddefault}{\updefault}0:0:100 (100)}}}}
\put(6301,-19561){\makebox(0,0)[b]{\smash{{\SetFigFontNFSS{20}{24.0}{\familydefault}{\mddefault}{\updefault}0:49:50 (77)}}}}
\put(6301,-20161){\makebox(0,0)[b]{\smash{{\SetFigFontNFSS{20}{24.0}{\familydefault}{\mddefault}{\updefault}0:100:0 (100)}}}}
\put(8401,-14761){\makebox(0,0)[b]{\smash{{\SetFigFontNFSS{20}{24.0}{\familydefault}{\mddefault}{\updefault}0:0:100 (100)}}}}
\put(8401,-15361){\makebox(0,0)[b]{\smash{{\SetFigFontNFSS{20}{24.0}{\familydefault}{\mddefault}{\updefault}0:0:100 (100)}}}}
\put(8401,-15961){\makebox(0,0)[b]{\smash{{\SetFigFontNFSS{20}{24.0}{\familydefault}{\mddefault}{\updefault}0:0:100 (100)}}}}
\put(8401,-16561){\makebox(0,0)[b]{\smash{{\SetFigFontNFSS{20}{24.0}{\familydefault}{\mddefault}{\updefault}0:0:100 (100)}}}}
\put(8401,-17161){\makebox(0,0)[b]{\smash{{\SetFigFontNFSS{20}{24.0}{\familydefault}{\mddefault}{\updefault}0:0:100 (100)}}}}
\put(8401,-17761){\makebox(0,0)[b]{\smash{{\SetFigFontNFSS{20}{24.0}{\familydefault}{\mddefault}{\updefault}0:0:100 (100)}}}}
\put(8401,-18361){\makebox(0,0)[b]{\smash{{\SetFigFontNFSS{20}{24.0}{\familydefault}{\mddefault}{\updefault}0:0:100 (100)}}}}
\put(8401,-18961){\makebox(0,0)[b]{\smash{{\SetFigFontNFSS{20}{24.0}{\familydefault}{\mddefault}{\updefault}0:0:100 (100)}}}}
\put(8401,-19561){\makebox(0,0)[b]{\smash{{\SetFigFontNFSS{20}{24.0}{\familydefault}{\mddefault}{\updefault}0:50:49 (77)}}}}
\put(8401,-20161){\makebox(0,0)[b]{\smash{{\SetFigFontNFSS{20}{24.0}{\familydefault}{\mddefault}{\updefault}0:100:0 (100)}}}}
\put(10501,-14761){\makebox(0,0)[b]{\smash{{\SetFigFontNFSS{20}{24.0}{\familydefault}{\mddefault}{\updefault}0:0:100 (100)}}}}
\put(10501,-15361){\makebox(0,0)[b]{\smash{{\SetFigFontNFSS{20}{24.0}{\familydefault}{\mddefault}{\updefault}0:0:100 (100)}}}}
\put(10501,-15961){\makebox(0,0)[b]{\smash{{\SetFigFontNFSS{20}{24.0}{\familydefault}{\mddefault}{\updefault}0:0:100 (100)}}}}
\put(10501,-16561){\makebox(0,0)[b]{\smash{{\SetFigFontNFSS{20}{24.0}{\familydefault}{\mddefault}{\updefault}0:0:100 (100)}}}}
\put(10501,-17161){\makebox(0,0)[b]{\smash{{\SetFigFontNFSS{20}{24.0}{\familydefault}{\mddefault}{\updefault}0:0:100 (100)}}}}
\put(10501,-17761){\makebox(0,0)[b]{\smash{{\SetFigFontNFSS{20}{24.0}{\familydefault}{\mddefault}{\updefault}0:0:100 (100)}}}}
\put(10501,-18361){\makebox(0,0)[b]{\smash{{\SetFigFontNFSS{20}{24.0}{\familydefault}{\mddefault}{\updefault}0:0:100 (100)}}}}
\put(10501,-18961){\makebox(0,0)[b]{\smash{{\SetFigFontNFSS{20}{24.0}{\familydefault}{\mddefault}{\updefault}0:0:100 (100)}}}}
\put(10501,-19561){\makebox(0,0)[b]{\smash{{\SetFigFontNFSS{20}{24.0}{\familydefault}{\mddefault}{\updefault}0:49:50 (77)}}}}
\put(10501,-20161){\makebox(0,0)[b]{\smash{{\SetFigFontNFSS{20}{24.0}{\familydefault}{\mddefault}{\updefault}0:100:0 (100)}}}}
\put(12601,-14761){\makebox(0,0)[b]{\smash{{\SetFigFontNFSS{20}{24.0}{\familydefault}{\mddefault}{\updefault}0:0:100 (100)}}}}
\put(12601,-15361){\makebox(0,0)[b]{\smash{{\SetFigFontNFSS{20}{24.0}{\familydefault}{\mddefault}{\updefault}0:0:100 (100)}}}}
\put(12601,-15961){\makebox(0,0)[b]{\smash{{\SetFigFontNFSS{20}{24.0}{\familydefault}{\mddefault}{\updefault}0:0:100 (100)}}}}
\put(12601,-16561){\makebox(0,0)[b]{\smash{{\SetFigFontNFSS{20}{24.0}{\familydefault}{\mddefault}{\updefault}0:0:100 (100)}}}}
\put(12601,-17161){\makebox(0,0)[b]{\smash{{\SetFigFontNFSS{20}{24.0}{\familydefault}{\mddefault}{\updefault}0:0:100 (100)}}}}
\put(12601,-17761){\makebox(0,0)[b]{\smash{{\SetFigFontNFSS{20}{24.0}{\familydefault}{\mddefault}{\updefault}0:0:100 (100)}}}}
\put(12601,-18361){\makebox(0,0)[b]{\smash{{\SetFigFontNFSS{20}{24.0}{\familydefault}{\mddefault}{\updefault}0:0:100 (100)}}}}
\put(12601,-18961){\makebox(0,0)[b]{\smash{{\SetFigFontNFSS{20}{24.0}{\familydefault}{\mddefault}{\updefault}0:0:100 (100)}}}}
\put(12601,-19561){\makebox(0,0)[b]{\smash{{\SetFigFontNFSS{20}{24.0}{\familydefault}{\mddefault}{\updefault}0:48:51 (77)}}}}
\put(12601,-20161){\makebox(0,0)[b]{\smash{{\SetFigFontNFSS{20}{24.0}{\familydefault}{\mddefault}{\updefault}0:100:0 (100)}}}}
\put(14701,-14761){\makebox(0,0)[b]{\smash{{\SetFigFontNFSS{20}{24.0}{\familydefault}{\mddefault}{\updefault}0:0:100 (100)}}}}
\put(14701,-15361){\makebox(0,0)[b]{\smash{{\SetFigFontNFSS{20}{24.0}{\familydefault}{\mddefault}{\updefault}0:0:100 (100)}}}}
\put(14701,-15961){\makebox(0,0)[b]{\smash{{\SetFigFontNFSS{20}{24.0}{\familydefault}{\mddefault}{\updefault}0:0:100 (100)}}}}
\put(14701,-16561){\makebox(0,0)[b]{\smash{{\SetFigFontNFSS{20}{24.0}{\familydefault}{\mddefault}{\updefault}0:0:100 (100)}}}}
\put(14701,-17161){\makebox(0,0)[b]{\smash{{\SetFigFontNFSS{20}{24.0}{\familydefault}{\mddefault}{\updefault}0:0:100 (100)}}}}
\put(14701,-17761){\makebox(0,0)[b]{\smash{{\SetFigFontNFSS{20}{24.0}{\familydefault}{\mddefault}{\updefault}0:0:100 (100)}}}}
\put(14701,-18361){\makebox(0,0)[b]{\smash{{\SetFigFontNFSS{20}{24.0}{\familydefault}{\mddefault}{\updefault}0:0:100 (100)}}}}
\put(14701,-18961){\makebox(0,0)[b]{\smash{{\SetFigFontNFSS{20}{24.0}{\familydefault}{\mddefault}{\updefault}0:0:100 (100)}}}}
\put(14701,-19561){\makebox(0,0)[b]{\smash{{\SetFigFontNFSS{20}{24.0}{\familydefault}{\mddefault}{\updefault}0:51:48 (77)}}}}
\put(14701,-20161){\makebox(0,0)[b]{\smash{{\SetFigFontNFSS{20}{24.0}{\familydefault}{\mddefault}{\updefault}0:100:0 (100)}}}}
\put(16801,-14761){\makebox(0,0)[b]{\smash{{\SetFigFontNFSS{20}{24.0}{\familydefault}{\mddefault}{\updefault}51:0:48 (76)}}}}
\put(16801,-15361){\makebox(0,0)[b]{\smash{{\SetFigFontNFSS{20}{24.0}{\familydefault}{\mddefault}{\updefault}52:0:47 (76)}}}}
\put(16801,-15961){\makebox(0,0)[b]{\smash{{\SetFigFontNFSS{20}{24.0}{\familydefault}{\mddefault}{\updefault}49:0:50 (76)}}}}
\put(16801,-16561){\makebox(0,0)[b]{\smash{{\SetFigFontNFSS{20}{24.0}{\familydefault}{\mddefault}{\updefault}48:0:51 (76)}}}}
\put(16801,-17161){\makebox(0,0)[b]{\smash{{\SetFigFontNFSS{20}{24.0}{\familydefault}{\mddefault}{\updefault}49:0:50 (76)}}}}
\put(16801,-17761){\makebox(0,0)[b]{\smash{{\SetFigFontNFSS{20}{24.0}{\familydefault}{\mddefault}{\updefault}50:0:49 (76)}}}}
\put(16801,-18361){\makebox(0,0)[b]{\smash{{\SetFigFontNFSS{20}{24.0}{\familydefault}{\mddefault}{\updefault}50:0:49 (76)}}}}
\put(16801,-18961){\makebox(0,0)[b]{\smash{{\SetFigFontNFSS{20}{24.0}{\familydefault}{\mddefault}{\updefault}50:0:49 (76)}}}}
\put(16801,-19561){\makebox(0,0)[b]{\smash{{\SetFigFontNFSS{20}{24.0}{\familydefault}{\mddefault}{\updefault}32:32:34 (70)}}}}
\put(16801,-20161){\makebox(0,0)[b]{\smash{{\SetFigFontNFSS{20}{24.0}{\familydefault}{\mddefault}{\updefault}0:100:0 (100)}}}}
\put(18901,-14761){\makebox(0,0)[b]{\smash{{\SetFigFontNFSS{20}{24.0}{\familydefault}{\mddefault}{\updefault}100:0:0 (100)}}}}
\put(18901,-15361){\makebox(0,0)[b]{\smash{{\SetFigFontNFSS{20}{24.0}{\familydefault}{\mddefault}{\updefault}100:0:0 (100)}}}}
\put(18901,-15961){\makebox(0,0)[b]{\smash{{\SetFigFontNFSS{20}{24.0}{\familydefault}{\mddefault}{\updefault}100:0:0 (100)}}}}
\put(18901,-16561){\makebox(0,0)[b]{\smash{{\SetFigFontNFSS{20}{24.0}{\familydefault}{\mddefault}{\updefault}100:0:0 (100)}}}}
\put(18901,-17161){\makebox(0,0)[b]{\smash{{\SetFigFontNFSS{20}{24.0}{\familydefault}{\mddefault}{\updefault}100:0:0 (100)}}}}
\put(18901,-17761){\makebox(0,0)[b]{\smash{{\SetFigFontNFSS{20}{24.0}{\familydefault}{\mddefault}{\updefault}100:0:0 (100)}}}}
\put(18901,-18361){\makebox(0,0)[b]{\smash{{\SetFigFontNFSS{20}{24.0}{\familydefault}{\mddefault}{\updefault}100:0:0 (100)}}}}
\put(18901,-18961){\makebox(0,0)[b]{\smash{{\SetFigFontNFSS{20}{24.0}{\familydefault}{\mddefault}{\updefault}100:0:0 (100)}}}}
\put(18901,-19561){\makebox(0,0)[b]{\smash{{\SetFigFontNFSS{20}{24.0}{\familydefault}{\mddefault}{\updefault}100:0:0 (100)}}}}
\put(18901,-20161){\makebox(0,0)[b]{\smash{{\SetFigFontNFSS{20}{24.0}{\familydefault}{\mddefault}{\updefault}52:47:0 (85)}}}}
\put(  1,-6361){\makebox(0,0)[b]{\smash{{\SetFigFontNFSS{20}{24.0}{\familydefault}{\mddefault}{\updefault}$\theta_1 = 0/27$}}}}
\put(2101,-6361){\makebox(0,0)[b]{\smash{{\SetFigFontNFSS{20}{24.0}{\familydefault}{\mddefault}{\updefault}$\theta_1 = 1/27$}}}}
\put(4201,-6361){\makebox(0,0)[b]{\smash{{\SetFigFontNFSS{20}{24.0}{\familydefault}{\mddefault}{\updefault}$\theta_1 = 2/27$}}}}
\put(6301,-6361){\makebox(0,0)[b]{\smash{{\SetFigFontNFSS{20}{24.0}{\familydefault}{\mddefault}{\updefault}$\theta_1 = 3/27$}}}}
\put(8401,-6361){\makebox(0,0)[b]{\smash{{\SetFigFontNFSS{20}{24.0}{\familydefault}{\mddefault}{\updefault}$\theta_1 = 4/27$}}}}
\put(10501,-6361){\makebox(0,0)[b]{\smash{{\SetFigFontNFSS{20}{24.0}{\familydefault}{\mddefault}{\updefault}$\theta_1 = 5/27$}}}}
\put(12601,-6361){\makebox(0,0)[b]{\smash{{\SetFigFontNFSS{20}{24.0}{\familydefault}{\mddefault}{\updefault}$\theta_1 = 6/27$}}}}
\put(14701,-6361){\makebox(0,0)[b]{\smash{{\SetFigFontNFSS{20}{24.0}{\familydefault}{\mddefault}{\updefault}$\theta_1 = 7/27$}}}}
\put(16801,-6361){\makebox(0,0)[b]{\smash{{\SetFigFontNFSS{20}{24.0}{\familydefault}{\mddefault}{\updefault}$\theta_1 = 8/27$}}}}
\put(18901,-6361){\makebox(0,0)[b]{\smash{{\SetFigFontNFSS{20}{24.0}{\familydefault}{\mddefault}{\updefault}$\theta_1 = 9/27$}}}}
\put(-2099,-6961){\makebox(0,0)[b]{\smash{{\SetFigFontNFSS{20}{24.0}{\familydefault}{\mddefault}{\updefault}$\theta_2 = 0/27$}}}}
\put(-2099,-7561){\makebox(0,0)[b]{\smash{{\SetFigFontNFSS{20}{24.0}{\familydefault}{\mddefault}{\updefault}$\theta_2 = 1/27$}}}}
\put(-2099,-8161){\makebox(0,0)[b]{\smash{{\SetFigFontNFSS{20}{24.0}{\familydefault}{\mddefault}{\updefault}$\theta_2 = 2/27$}}}}
\put(-2099,-8761){\makebox(0,0)[b]{\smash{{\SetFigFontNFSS{20}{24.0}{\familydefault}{\mddefault}{\updefault}$\theta_2 = 3/27$}}}}
\put(-2099,-9361){\makebox(0,0)[b]{\smash{{\SetFigFontNFSS{20}{24.0}{\familydefault}{\mddefault}{\updefault}$\theta_2 = 4/27$}}}}
\put(-2099,-9961){\makebox(0,0)[b]{\smash{{\SetFigFontNFSS{20}{24.0}{\familydefault}{\mddefault}{\updefault}$\theta_2 = 5/27$}}}}
\put(-2099,-10561){\makebox(0,0)[b]{\smash{{\SetFigFontNFSS{20}{24.0}{\familydefault}{\mddefault}{\updefault}$\theta_2 = 6/27$}}}}
\put(-2099,-11161){\makebox(0,0)[b]{\smash{{\SetFigFontNFSS{20}{24.0}{\familydefault}{\mddefault}{\updefault}$\theta_2 = 7/27$}}}}
\put(-2099,-11761){\makebox(0,0)[b]{\smash{{\SetFigFontNFSS{20}{24.0}{\familydefault}{\mddefault}{\updefault}$\theta_2 = 8/27$}}}}
\put(-2099,-12361){\makebox(0,0)[b]{\smash{{\SetFigFontNFSS{20}{24.0}{\familydefault}{\mddefault}{\updefault}$\theta_2 = 9/27$}}}}
\put(  1,-6961){\makebox(0,0)[b]{\smash{{\SetFigFontNFSS{20}{24.0}{\familydefault}{\mddefault}{\updefault}33:33:33 (47)}}}}
\put(  1,-7561){\makebox(0,0)[b]{\smash{{\SetFigFontNFSS{20}{24.0}{\familydefault}{\mddefault}{\updefault}30:39:29 (48)}}}}
\put(  1,-8161){\makebox(0,0)[b]{\smash{{\SetFigFontNFSS{20}{24.0}{\familydefault}{\mddefault}{\updefault}26:47:26 (51)}}}}
\put(  1,-8761){\makebox(0,0)[b]{\smash{{\SetFigFontNFSS{20}{24.0}{\familydefault}{\mddefault}{\updefault}20:58:20 (57)}}}}
\put(  1,-9361){\makebox(0,0)[b]{\smash{{\SetFigFontNFSS{20}{24.0}{\familydefault}{\mddefault}{\updefault}13:73:13 (69)}}}}
\put(  1,-9961){\makebox(0,0)[b]{\smash{{\SetFigFontNFSS{20}{24.0}{\familydefault}{\mddefault}{\updefault}0:100:0 (100)}}}}
\put(  1,-10561){\makebox(0,0)[b]{\smash{{\SetFigFontNFSS{20}{24.0}{\familydefault}{\mddefault}{\updefault}0:100:0 (100)}}}}
\put(  1,-11161){\makebox(0,0)[b]{\smash{{\SetFigFontNFSS{20}{24.0}{\familydefault}{\mddefault}{\updefault}0:100:0 (100)}}}}
\put(  1,-11761){\makebox(0,0)[b]{\smash{{\SetFigFontNFSS{20}{24.0}{\familydefault}{\mddefault}{\updefault}0:100:0 (100)}}}}
\put(  1,-12361){\makebox(0,0)[b]{\smash{{\SetFigFontNFSS{20}{24.0}{\familydefault}{\mddefault}{\updefault}0:100:0 (100)}}}}
\put(2101,-6961){\makebox(0,0)[b]{\smash{{\SetFigFontNFSS{20}{24.0}{\familydefault}{\mddefault}{\updefault}39:30:29 (48)}}}}
\put(2101,-7561){\makebox(0,0)[b]{\smash{{\SetFigFontNFSS{20}{24.0}{\familydefault}{\mddefault}{\updefault}36:36:26 (48)}}}}
\put(2101,-8161){\makebox(0,0)[b]{\smash{{\SetFigFontNFSS{20}{24.0}{\familydefault}{\mddefault}{\updefault}33:44:22 (51)}}}}
\put(2101,-8761){\makebox(0,0)[b]{\smash{{\SetFigFontNFSS{20}{24.0}{\familydefault}{\mddefault}{\updefault}28:55:16 (56)}}}}
\put(2101,-9361){\makebox(0,0)[b]{\smash{{\SetFigFontNFSS{20}{24.0}{\familydefault}{\mddefault}{\updefault}24:70:4 (68)}}}}
\put(2101,-9961){\makebox(0,0)[b]{\smash{{\SetFigFontNFSS{20}{24.0}{\familydefault}{\mddefault}{\updefault}7:92:0 (90)}}}}
\put(2101,-10561){\makebox(0,0)[b]{\smash{{\SetFigFontNFSS{20}{24.0}{\familydefault}{\mddefault}{\updefault}0:100:0 (100)}}}}
\put(2101,-11161){\makebox(0,0)[b]{\smash{{\SetFigFontNFSS{20}{24.0}{\familydefault}{\mddefault}{\updefault}0:100:0 (100)}}}}
\put(2101,-11761){\makebox(0,0)[b]{\smash{{\SetFigFontNFSS{20}{24.0}{\familydefault}{\mddefault}{\updefault}0:100:0 (100)}}}}
\put(2101,-12361){\makebox(0,0)[b]{\smash{{\SetFigFontNFSS{20}{24.0}{\familydefault}{\mddefault}{\updefault}0:100:0 (100)}}}}
\put(4201,-6961){\makebox(0,0)[b]{\smash{{\SetFigFontNFSS{20}{24.0}{\familydefault}{\mddefault}{\updefault}48:25:25 (51)}}}}
\put(4201,-7561){\makebox(0,0)[b]{\smash{{\SetFigFontNFSS{20}{24.0}{\familydefault}{\mddefault}{\updefault}44:33:21 (51)}}}}
\put(4201,-8161){\makebox(0,0)[b]{\smash{{\SetFigFontNFSS{20}{24.0}{\familydefault}{\mddefault}{\updefault}41:41:16 (53)}}}}
\put(4201,-8761){\makebox(0,0)[b]{\smash{{\SetFigFontNFSS{20}{24.0}{\familydefault}{\mddefault}{\updefault}38:51:9 (57)}}}}
\put(4201,-9361){\makebox(0,0)[b]{\smash{{\SetFigFontNFSS{20}{24.0}{\familydefault}{\mddefault}{\updefault}33:66:0 (67)}}}}
\put(4201,-9961){\makebox(0,0)[b]{\smash{{\SetFigFontNFSS{20}{24.0}{\familydefault}{\mddefault}{\updefault}17:82:0 (80)}}}}
\put(4201,-10561){\makebox(0,0)[b]{\smash{{\SetFigFontNFSS{20}{24.0}{\familydefault}{\mddefault}{\updefault}0:100:0 (100)}}}}
\put(4201,-11161){\makebox(0,0)[b]{\smash{{\SetFigFontNFSS{20}{24.0}{\familydefault}{\mddefault}{\updefault}0:100:0 (100)}}}}
\put(4201,-11761){\makebox(0,0)[b]{\smash{{\SetFigFontNFSS{20}{24.0}{\familydefault}{\mddefault}{\updefault}0:100:0 (100)}}}}
\put(4201,-12361){\makebox(0,0)[b]{\smash{{\SetFigFontNFSS{20}{24.0}{\familydefault}{\mddefault}{\updefault}0:100:0 (100)}}}}
\put(6301,-6961){\makebox(0,0)[b]{\smash{{\SetFigFontNFSS{20}{24.0}{\familydefault}{\mddefault}{\updefault}58:20:20 (57)}}}}
\put(6301,-7561){\makebox(0,0)[b]{\smash{{\SetFigFontNFSS{20}{24.0}{\familydefault}{\mddefault}{\updefault}55:28:16 (56)}}}}
\put(6301,-8161){\makebox(0,0)[b]{\smash{{\SetFigFontNFSS{20}{24.0}{\familydefault}{\mddefault}{\updefault}52:38:8 (58)}}}}
\put(6301,-8761){\makebox(0,0)[b]{\smash{{\SetFigFontNFSS{20}{24.0}{\familydefault}{\mddefault}{\updefault}49:50:0 (62)}}}}
\put(6301,-9361){\makebox(0,0)[b]{\smash{{\SetFigFontNFSS{20}{24.0}{\familydefault}{\mddefault}{\updefault}40:59:0 (64)}}}}
\put(6301,-9961){\makebox(0,0)[b]{\smash{{\SetFigFontNFSS{20}{24.0}{\familydefault}{\mddefault}{\updefault}26:73:0 (73)}}}}
\put(6301,-10561){\makebox(0,0)[b]{\smash{{\SetFigFontNFSS{20}{24.0}{\familydefault}{\mddefault}{\updefault}0:100:0 (100)}}}}
\put(6301,-11161){\makebox(0,0)[b]{\smash{{\SetFigFontNFSS{20}{24.0}{\familydefault}{\mddefault}{\updefault}0:100:0 (100)}}}}
\put(6301,-11761){\makebox(0,0)[b]{\smash{{\SetFigFontNFSS{20}{24.0}{\familydefault}{\mddefault}{\updefault}0:100:0 (100)}}}}
\put(6301,-12361){\makebox(0,0)[b]{\smash{{\SetFigFontNFSS{20}{24.0}{\familydefault}{\mddefault}{\updefault}0:100:0 (100)}}}}
\put(8401,-6961){\makebox(0,0)[b]{\smash{{\SetFigFontNFSS{20}{24.0}{\familydefault}{\mddefault}{\updefault}73:13:13 (69)}}}}
\put(8401,-7561){\makebox(0,0)[b]{\smash{{\SetFigFontNFSS{20}{24.0}{\familydefault}{\mddefault}{\updefault}70:24:4 (68)}}}}
\put(8401,-8161){\makebox(0,0)[b]{\smash{{\SetFigFontNFSS{20}{24.0}{\familydefault}{\mddefault}{\updefault}66:33:0 (67)}}}}
\put(8401,-8761){\makebox(0,0)[b]{\smash{{\SetFigFontNFSS{20}{24.0}{\familydefault}{\mddefault}{\updefault}58:41:0 (64)}}}}
\put(8401,-9361){\makebox(0,0)[b]{\smash{{\SetFigFontNFSS{20}{24.0}{\familydefault}{\mddefault}{\updefault}49:50:0 (64)}}}}
\put(8401,-9961){\makebox(0,0)[b]{\smash{{\SetFigFontNFSS{20}{24.0}{\familydefault}{\mddefault}{\updefault}37:62:0 (68)}}}}
\put(8401,-10561){\makebox(0,0)[b]{\smash{{\SetFigFontNFSS{20}{24.0}{\familydefault}{\mddefault}{\updefault}13:86:0 (85)}}}}
\put(8401,-11161){\makebox(0,0)[b]{\smash{{\SetFigFontNFSS{20}{24.0}{\familydefault}{\mddefault}{\updefault}0:100:0 (100)}}}}
\put(8401,-11761){\makebox(0,0)[b]{\smash{{\SetFigFontNFSS{20}{24.0}{\familydefault}{\mddefault}{\updefault}0:100:0 (100)}}}}
\put(8401,-12361){\makebox(0,0)[b]{\smash{{\SetFigFontNFSS{20}{24.0}{\familydefault}{\mddefault}{\updefault}0:100:0 (100)}}}}
\put(10501,-6961){\makebox(0,0)[b]{\smash{{\SetFigFontNFSS{20}{24.0}{\familydefault}{\mddefault}{\updefault}100:0:0 (100)}}}}
\put(10501,-7561){\makebox(0,0)[b]{\smash{{\SetFigFontNFSS{20}{24.0}{\familydefault}{\mddefault}{\updefault}91:8:0 (90)}}}}
\put(10501,-8161){\makebox(0,0)[b]{\smash{{\SetFigFontNFSS{20}{24.0}{\familydefault}{\mddefault}{\updefault}82:17:0 (80)}}}}
\put(10501,-8761){\makebox(0,0)[b]{\smash{{\SetFigFontNFSS{20}{24.0}{\familydefault}{\mddefault}{\updefault}73:26:0 (73)}}}}
\put(10501,-9361){\makebox(0,0)[b]{\smash{{\SetFigFontNFSS{20}{24.0}{\familydefault}{\mddefault}{\updefault}63:36:0 (68)}}}}
\put(10501,-9961){\makebox(0,0)[b]{\smash{{\SetFigFontNFSS{20}{24.0}{\familydefault}{\mddefault}{\updefault}50:49:0 (66)}}}}
\put(10501,-10561){\makebox(0,0)[b]{\smash{{\SetFigFontNFSS{20}{24.0}{\familydefault}{\mddefault}{\updefault}29:70:0 (74)}}}}
\put(10501,-11161){\makebox(0,0)[b]{\smash{{\SetFigFontNFSS{20}{24.0}{\familydefault}{\mddefault}{\updefault}0:100:0 (100)}}}}
\put(10501,-11761){\makebox(0,0)[b]{\smash{{\SetFigFontNFSS{20}{24.0}{\familydefault}{\mddefault}{\updefault}0:100:0 (100)}}}}
\put(10501,-12361){\makebox(0,0)[b]{\smash{{\SetFigFontNFSS{20}{24.0}{\familydefault}{\mddefault}{\updefault}0:100:0 (100)}}}}
\put(12601,-6961){\makebox(0,0)[b]{\smash{{\SetFigFontNFSS{20}{24.0}{\familydefault}{\mddefault}{\updefault}100:0:0 (100)}}}}
\put(12601,-7561){\makebox(0,0)[b]{\smash{{\SetFigFontNFSS{20}{24.0}{\familydefault}{\mddefault}{\updefault}100:0:0 (100)}}}}
\put(12601,-8161){\makebox(0,0)[b]{\smash{{\SetFigFontNFSS{20}{24.0}{\familydefault}{\mddefault}{\updefault}100:0:0 (100)}}}}
\put(12601,-8761){\makebox(0,0)[b]{\smash{{\SetFigFontNFSS{20}{24.0}{\familydefault}{\mddefault}{\updefault}100:0:0 (100)}}}}
\put(12601,-9361){\makebox(0,0)[b]{\smash{{\SetFigFontNFSS{20}{24.0}{\familydefault}{\mddefault}{\updefault}86:13:0 (85)}}}}
\put(12601,-9961){\makebox(0,0)[b]{\smash{{\SetFigFontNFSS{20}{24.0}{\familydefault}{\mddefault}{\updefault}69:30:0 (73)}}}}
\put(12601,-10561){\makebox(0,0)[b]{\smash{{\SetFigFontNFSS{20}{24.0}{\familydefault}{\mddefault}{\updefault}50:49:0 (69)}}}}
\put(12601,-11161){\makebox(0,0)[b]{\smash{{\SetFigFontNFSS{20}{24.0}{\familydefault}{\mddefault}{\updefault}13:86:0 (87)}}}}
\put(12601,-11761){\makebox(0,0)[b]{\smash{{\SetFigFontNFSS{20}{24.0}{\familydefault}{\mddefault}{\updefault}0:100:0 (100)}}}}
\put(12601,-12361){\makebox(0,0)[b]{\smash{{\SetFigFontNFSS{20}{24.0}{\familydefault}{\mddefault}{\updefault}0:100:0 (100)}}}}
\put(14701,-6961){\makebox(0,0)[b]{\smash{{\SetFigFontNFSS{20}{24.0}{\familydefault}{\mddefault}{\updefault}100:0:0 (100)}}}}
\put(14701,-7561){\makebox(0,0)[b]{\smash{{\SetFigFontNFSS{20}{24.0}{\familydefault}{\mddefault}{\updefault}100:0:0 (100)}}}}
\put(14701,-8161){\makebox(0,0)[b]{\smash{{\SetFigFontNFSS{20}{24.0}{\familydefault}{\mddefault}{\updefault}100:0:0 (100)}}}}
\put(14701,-8761){\makebox(0,0)[b]{\smash{{\SetFigFontNFSS{20}{24.0}{\familydefault}{\mddefault}{\updefault}100:0:0 (100)}}}}
\put(14701,-9361){\makebox(0,0)[b]{\smash{{\SetFigFontNFSS{20}{24.0}{\familydefault}{\mddefault}{\updefault}100:0:0 (100)}}}}
\put(14701,-9961){\makebox(0,0)[b]{\smash{{\SetFigFontNFSS{20}{24.0}{\familydefault}{\mddefault}{\updefault}100:0:0 (100)}}}}
\put(14701,-10561){\makebox(0,0)[b]{\smash{{\SetFigFontNFSS{20}{24.0}{\familydefault}{\mddefault}{\updefault}86:13:0 (87)}}}}
\put(14701,-11161){\makebox(0,0)[b]{\smash{{\SetFigFontNFSS{20}{24.0}{\familydefault}{\mddefault}{\updefault}50:49:0 (72)}}}}
\put(14701,-11761){\makebox(0,0)[b]{\smash{{\SetFigFontNFSS{20}{24.0}{\familydefault}{\mddefault}{\updefault}0:100:0 (100)}}}}
\put(14701,-12361){\makebox(0,0)[b]{\smash{{\SetFigFontNFSS{20}{24.0}{\familydefault}{\mddefault}{\updefault}0:100:0 (100)}}}}
\put(16801,-6961){\makebox(0,0)[b]{\smash{{\SetFigFontNFSS{20}{24.0}{\familydefault}{\mddefault}{\updefault}100:0:0 (100)}}}}
\put(16801,-7561){\makebox(0,0)[b]{\smash{{\SetFigFontNFSS{20}{24.0}{\familydefault}{\mddefault}{\updefault}100:0:0 (100)}}}}
\put(16801,-8161){\makebox(0,0)[b]{\smash{{\SetFigFontNFSS{20}{24.0}{\familydefault}{\mddefault}{\updefault}100:0:0 (100)}}}}
\put(16801,-8761){\makebox(0,0)[b]{\smash{{\SetFigFontNFSS{20}{24.0}{\familydefault}{\mddefault}{\updefault}100:0:0 (100)}}}}
\put(16801,-9361){\makebox(0,0)[b]{\smash{{\SetFigFontNFSS{20}{24.0}{\familydefault}{\mddefault}{\updefault}100:0:0 (100)}}}}
\put(16801,-9961){\makebox(0,0)[b]{\smash{{\SetFigFontNFSS{20}{24.0}{\familydefault}{\mddefault}{\updefault}100:0:0 (100)}}}}
\put(16801,-10561){\makebox(0,0)[b]{\smash{{\SetFigFontNFSS{20}{24.0}{\familydefault}{\mddefault}{\updefault}100:0:0 (100)}}}}
\put(16801,-11161){\makebox(0,0)[b]{\smash{{\SetFigFontNFSS{20}{24.0}{\familydefault}{\mddefault}{\updefault}100:0:0 (100)}}}}
\put(16801,-11761){\makebox(0,0)[b]{\smash{{\SetFigFontNFSS{20}{24.0}{\familydefault}{\mddefault}{\updefault}49:50:0 (76)}}}}
\put(16801,-12361){\makebox(0,0)[b]{\smash{{\SetFigFontNFSS{20}{24.0}{\familydefault}{\mddefault}{\updefault}0:100:0 (100)}}}}
\put(18901,-6961){\makebox(0,0)[b]{\smash{{\SetFigFontNFSS{20}{24.0}{\familydefault}{\mddefault}{\updefault}100:0:0 (100)}}}}
\put(18901,-7561){\makebox(0,0)[b]{\smash{{\SetFigFontNFSS{20}{24.0}{\familydefault}{\mddefault}{\updefault}100:0:0 (100)}}}}
\put(18901,-8161){\makebox(0,0)[b]{\smash{{\SetFigFontNFSS{20}{24.0}{\familydefault}{\mddefault}{\updefault}100:0:0 (100)}}}}
\put(18901,-8761){\makebox(0,0)[b]{\smash{{\SetFigFontNFSS{20}{24.0}{\familydefault}{\mddefault}{\updefault}100:0:0 (100)}}}}
\put(18901,-9361){\makebox(0,0)[b]{\smash{{\SetFigFontNFSS{20}{24.0}{\familydefault}{\mddefault}{\updefault}100:0:0 (100)}}}}
\put(18901,-9961){\makebox(0,0)[b]{\smash{{\SetFigFontNFSS{20}{24.0}{\familydefault}{\mddefault}{\updefault}100:0:0 (100)}}}}
\put(18901,-10561){\makebox(0,0)[b]{\smash{{\SetFigFontNFSS{20}{24.0}{\familydefault}{\mddefault}{\updefault}100:0:0 (100)}}}}
\put(18901,-11161){\makebox(0,0)[b]{\smash{{\SetFigFontNFSS{20}{24.0}{\familydefault}{\mddefault}{\updefault}100:0:0 (100)}}}}
\put(18901,-11761){\makebox(0,0)[b]{\smash{{\SetFigFontNFSS{20}{24.0}{\familydefault}{\mddefault}{\updefault}100:0:0 (100)}}}}
\put(18901,-12361){\makebox(0,0)[b]{\smash{{\SetFigFontNFSS{20}{24.0}{\familydefault}{\mddefault}{\updefault}49:50:0 (86)}}}}
\put(8401,-5311){\makebox(0,0)[b]{\smash{{\SetFigFontNFSS{20}{24.0}{\familydefault}{\mddefault}{\updefault}Densities and clustering coefficient when the interactions are described by the matrix $M_8$ with $\theta_3 = 19/27$}}}}
\put(  1,1439){\makebox(0,0)[b]{\smash{{\SetFigFontNFSS{20}{24.0}{\familydefault}{\mddefault}{\updefault}$\theta_1 = 9/27$}}}}
\put(2101,1439){\makebox(0,0)[b]{\smash{{\SetFigFontNFSS{20}{24.0}{\familydefault}{\mddefault}{\updefault}$\theta_1 = 11/27$}}}}
\put(4201,1439){\makebox(0,0)[b]{\smash{{\SetFigFontNFSS{20}{24.0}{\familydefault}{\mddefault}{\updefault}$\theta_1 = 13/27$}}}}
\put(6301,1439){\makebox(0,0)[b]{\smash{{\SetFigFontNFSS{20}{24.0}{\familydefault}{\mddefault}{\updefault}$\theta_1 = 15/27$}}}}
\put(8401,1439){\makebox(0,0)[b]{\smash{{\SetFigFontNFSS{20}{24.0}{\familydefault}{\mddefault}{\updefault}$\theta_1 = 17/27$}}}}
\put(10501,1439){\makebox(0,0)[b]{\smash{{\SetFigFontNFSS{20}{24.0}{\familydefault}{\mddefault}{\updefault}$\theta_1 = 19/27$}}}}
\put(12601,1439){\makebox(0,0)[b]{\smash{{\SetFigFontNFSS{20}{24.0}{\familydefault}{\mddefault}{\updefault}$\theta_1 = 21/27$}}}}
\put(14701,1439){\makebox(0,0)[b]{\smash{{\SetFigFontNFSS{20}{24.0}{\familydefault}{\mddefault}{\updefault}$\theta_1 = 23/27$}}}}
\put(16801,1439){\makebox(0,0)[b]{\smash{{\SetFigFontNFSS{20}{24.0}{\familydefault}{\mddefault}{\updefault}$\theta_1 = 25/27$}}}}
\put(18901,1439){\makebox(0,0)[b]{\smash{{\SetFigFontNFSS{20}{24.0}{\familydefault}{\mddefault}{\updefault}$\theta_1 = 27/27$}}}}
\put(-2099,839){\makebox(0,0)[b]{\smash{{\SetFigFontNFSS{20}{24.0}{\familydefault}{\mddefault}{\updefault}$\theta_3 = 9/27$}}}}
\put(-2099,239){\makebox(0,0)[b]{\smash{{\SetFigFontNFSS{20}{24.0}{\familydefault}{\mddefault}{\updefault}$\theta_3 = 11/27$}}}}
\put(-2099,-361){\makebox(0,0)[b]{\smash{{\SetFigFontNFSS{20}{24.0}{\familydefault}{\mddefault}{\updefault}$\theta_3 = 13/27$}}}}
\put(-2099,-961){\makebox(0,0)[b]{\smash{{\SetFigFontNFSS{20}{24.0}{\familydefault}{\mddefault}{\updefault}$\theta_3 = 15/27$}}}}
\put(-2099,-1561){\makebox(0,0)[b]{\smash{{\SetFigFontNFSS{20}{24.0}{\familydefault}{\mddefault}{\updefault}$\theta_3 = 17/27$}}}}
\put(-2099,-2161){\makebox(0,0)[b]{\smash{{\SetFigFontNFSS{20}{24.0}{\familydefault}{\mddefault}{\updefault}$\theta_3 = 19/27$}}}}
\put(-2099,-2761){\makebox(0,0)[b]{\smash{{\SetFigFontNFSS{20}{24.0}{\familydefault}{\mddefault}{\updefault}$\theta_3 = 21/27$}}}}
\put(-2099,-3361){\makebox(0,0)[b]{\smash{{\SetFigFontNFSS{20}{24.0}{\familydefault}{\mddefault}{\updefault}$\theta_3 = 23/27$}}}}
\put(-2099,-3961){\makebox(0,0)[b]{\smash{{\SetFigFontNFSS{20}{24.0}{\familydefault}{\mddefault}{\updefault}$\theta_3 = 25/27$}}}}
\put(-2099,-4561){\makebox(0,0)[b]{\smash{{\SetFigFontNFSS{20}{24.0}{\familydefault}{\mddefault}{\updefault}$\theta_3 = 27/27$}}}}
\put(  1,839){\makebox(0,0)[b]{\smash{{\SetFigFontNFSS{20}{24.0}{\familydefault}{\mddefault}{\updefault}0:0:100 (100)}}}}
\put(  1,239){\makebox(0,0)[b]{\smash{{\SetFigFontNFSS{20}{24.0}{\familydefault}{\mddefault}{\updefault}0:0:100 (100)}}}}
\put(  1,-361){\makebox(0,0)[b]{\smash{{\SetFigFontNFSS{20}{24.0}{\familydefault}{\mddefault}{\updefault}0:0:100 (100)}}}}
\put(  1,-961){\makebox(0,0)[b]{\smash{{\SetFigFontNFSS{20}{24.0}{\familydefault}{\mddefault}{\updefault}0:0:100 (100)}}}}
\put(  1,-1561){\makebox(0,0)[b]{\smash{{\SetFigFontNFSS{20}{24.0}{\familydefault}{\mddefault}{\updefault}0:0:100 (100)}}}}
\put(  1,-2161){\makebox(0,0)[b]{\smash{{\SetFigFontNFSS{20}{24.0}{\familydefault}{\mddefault}{\updefault}0:57:42 (97)}}}}
\put(  1,-2761){\makebox(0,0)[b]{\smash{{\SetFigFontNFSS{20}{24.0}{\familydefault}{\mddefault}{\updefault}0:100:0 (100)}}}}
\put(  1,-3361){\makebox(0,0)[b]{\smash{{\SetFigFontNFSS{20}{24.0}{\familydefault}{\mddefault}{\updefault}0:100:0 (100)}}}}
\put(  1,-3961){\makebox(0,0)[b]{\smash{{\SetFigFontNFSS{20}{24.0}{\familydefault}{\mddefault}{\updefault}0:100:0 (100)}}}}
\put(  1,-4561){\makebox(0,0)[b]{\smash{{\SetFigFontNFSS{20}{24.0}{\familydefault}{\mddefault}{\updefault}0:100:0 (100)}}}}
\put(2101,839){\makebox(0,0)[b]{\smash{{\SetFigFontNFSS{20}{24.0}{\familydefault}{\mddefault}{\updefault}0:0:100 (100)}}}}
\put(2101,239){\makebox(0,0)[b]{\smash{{\SetFigFontNFSS{20}{24.0}{\familydefault}{\mddefault}{\updefault}0:0:100 (100)}}}}
\put(2101,-361){\makebox(0,0)[b]{\smash{{\SetFigFontNFSS{20}{24.0}{\familydefault}{\mddefault}{\updefault}0:0:100 (100)}}}}
\put(2101,-961){\makebox(0,0)[b]{\smash{{\SetFigFontNFSS{20}{24.0}{\familydefault}{\mddefault}{\updefault}0:0:100 (100)}}}}
\put(2101,-1561){\makebox(0,0)[b]{\smash{{\SetFigFontNFSS{20}{24.0}{\familydefault}{\mddefault}{\updefault}0:0:100 (100)}}}}
\put(2101,-2161){\makebox(0,0)[b]{\smash{{\SetFigFontNFSS{20}{24.0}{\familydefault}{\mddefault}{\updefault}0:50:49 (97)}}}}
\put(2101,-2761){\makebox(0,0)[b]{\smash{{\SetFigFontNFSS{20}{24.0}{\familydefault}{\mddefault}{\updefault}0:100:0 (100)}}}}
\put(2101,-3361){\makebox(0,0)[b]{\smash{{\SetFigFontNFSS{20}{24.0}{\familydefault}{\mddefault}{\updefault}0:100:0 (100)}}}}
\put(2101,-3961){\makebox(0,0)[b]{\smash{{\SetFigFontNFSS{20}{24.0}{\familydefault}{\mddefault}{\updefault}0:100:0 (100)}}}}
\put(2101,-4561){\makebox(0,0)[b]{\smash{{\SetFigFontNFSS{20}{24.0}{\familydefault}{\mddefault}{\updefault}0:100:0 (100)}}}}
\put(4201,839){\makebox(0,0)[b]{\smash{{\SetFigFontNFSS{20}{24.0}{\familydefault}{\mddefault}{\updefault}0:0:100 (100)}}}}
\put(4201,239){\makebox(0,0)[b]{\smash{{\SetFigFontNFSS{20}{24.0}{\familydefault}{\mddefault}{\updefault}0:0:100 (100)}}}}
\put(4201,-361){\makebox(0,0)[b]{\smash{{\SetFigFontNFSS{20}{24.0}{\familydefault}{\mddefault}{\updefault}0:0:100 (100)}}}}
\put(4201,-961){\makebox(0,0)[b]{\smash{{\SetFigFontNFSS{20}{24.0}{\familydefault}{\mddefault}{\updefault}0:0:100 (100)}}}}
\put(4201,-1561){\makebox(0,0)[b]{\smash{{\SetFigFontNFSS{20}{24.0}{\familydefault}{\mddefault}{\updefault}0:0:100 (100)}}}}
\put(4201,-2161){\makebox(0,0)[b]{\smash{{\SetFigFontNFSS{20}{24.0}{\familydefault}{\mddefault}{\updefault}0:57:42 (97)}}}}
\put(4201,-2761){\makebox(0,0)[b]{\smash{{\SetFigFontNFSS{20}{24.0}{\familydefault}{\mddefault}{\updefault}0:100:0 (100)}}}}
\put(4201,-3361){\makebox(0,0)[b]{\smash{{\SetFigFontNFSS{20}{24.0}{\familydefault}{\mddefault}{\updefault}0:100:0 (100)}}}}
\put(4201,-3961){\makebox(0,0)[b]{\smash{{\SetFigFontNFSS{20}{24.0}{\familydefault}{\mddefault}{\updefault}0:100:0 (100)}}}}
\put(4201,-4561){\makebox(0,0)[b]{\smash{{\SetFigFontNFSS{20}{24.0}{\familydefault}{\mddefault}{\updefault}0:100:0 (100)}}}}
\put(6301,839){\makebox(0,0)[b]{\smash{{\SetFigFontNFSS{20}{24.0}{\familydefault}{\mddefault}{\updefault}0:0:100 (100)}}}}
\put(6301,239){\makebox(0,0)[b]{\smash{{\SetFigFontNFSS{20}{24.0}{\familydefault}{\mddefault}{\updefault}0:0:100 (100)}}}}
\put(6301,-361){\makebox(0,0)[b]{\smash{{\SetFigFontNFSS{20}{24.0}{\familydefault}{\mddefault}{\updefault}0:0:100 (100)}}}}
\put(6301,-961){\makebox(0,0)[b]{\smash{{\SetFigFontNFSS{20}{24.0}{\familydefault}{\mddefault}{\updefault}0:0:100 (100)}}}}
\put(6301,-1561){\makebox(0,0)[b]{\smash{{\SetFigFontNFSS{20}{24.0}{\familydefault}{\mddefault}{\updefault}0:0:100 (100)}}}}
\put(6301,-2161){\makebox(0,0)[b]{\smash{{\SetFigFontNFSS{20}{24.0}{\familydefault}{\mddefault}{\updefault}0:50:49 (97)}}}}
\put(6301,-2761){\makebox(0,0)[b]{\smash{{\SetFigFontNFSS{20}{24.0}{\familydefault}{\mddefault}{\updefault}0:100:0 (100)}}}}
\put(6301,-3361){\makebox(0,0)[b]{\smash{{\SetFigFontNFSS{20}{24.0}{\familydefault}{\mddefault}{\updefault}0:100:0 (100)}}}}
\put(6301,-3961){\makebox(0,0)[b]{\smash{{\SetFigFontNFSS{20}{24.0}{\familydefault}{\mddefault}{\updefault}0:100:0 (100)}}}}
\put(6301,-4561){\makebox(0,0)[b]{\smash{{\SetFigFontNFSS{20}{24.0}{\familydefault}{\mddefault}{\updefault}0:100:0 (100)}}}}
\put(8401,839){\makebox(0,0)[b]{\smash{{\SetFigFontNFSS{20}{24.0}{\familydefault}{\mddefault}{\updefault}0:0:100 (100)}}}}
\put(8401,239){\makebox(0,0)[b]{\smash{{\SetFigFontNFSS{20}{24.0}{\familydefault}{\mddefault}{\updefault}0:0:100 (100)}}}}
\put(8401,-361){\makebox(0,0)[b]{\smash{{\SetFigFontNFSS{20}{24.0}{\familydefault}{\mddefault}{\updefault}0:0:100 (100)}}}}
\put(8401,-961){\makebox(0,0)[b]{\smash{{\SetFigFontNFSS{20}{24.0}{\familydefault}{\mddefault}{\updefault}0:0:100 (100)}}}}
\put(8401,-1561){\makebox(0,0)[b]{\smash{{\SetFigFontNFSS{20}{24.0}{\familydefault}{\mddefault}{\updefault}0:0:100 (100)}}}}
\put(8401,-2161){\makebox(0,0)[b]{\smash{{\SetFigFontNFSS{20}{24.0}{\familydefault}{\mddefault}{\updefault}0:49:50 (97)}}}}
\put(8401,-2761){\makebox(0,0)[b]{\smash{{\SetFigFontNFSS{20}{24.0}{\familydefault}{\mddefault}{\updefault}0:100:0 (100)}}}}
\put(8401,-3361){\makebox(0,0)[b]{\smash{{\SetFigFontNFSS{20}{24.0}{\familydefault}{\mddefault}{\updefault}0:100:0 (100)}}}}
\put(8401,-3961){\makebox(0,0)[b]{\smash{{\SetFigFontNFSS{20}{24.0}{\familydefault}{\mddefault}{\updefault}0:100:0 (100)}}}}
\put(8401,-4561){\makebox(0,0)[b]{\smash{{\SetFigFontNFSS{20}{24.0}{\familydefault}{\mddefault}{\updefault}0:100:0 (100)}}}}
\put(10501,839){\makebox(0,0)[b]{\smash{{\SetFigFontNFSS{20}{24.0}{\familydefault}{\mddefault}{\updefault}54:0:45 (97)}}}}
\put(10501,239){\makebox(0,0)[b]{\smash{{\SetFigFontNFSS{20}{24.0}{\familydefault}{\mddefault}{\updefault}57:0:42 (97)}}}}
\put(10501,-361){\makebox(0,0)[b]{\smash{{\SetFigFontNFSS{20}{24.0}{\familydefault}{\mddefault}{\updefault}45:0:54 (97)}}}}
\put(10501,-961){\makebox(0,0)[b]{\smash{{\SetFigFontNFSS{20}{24.0}{\familydefault}{\mddefault}{\updefault}60:0:39 (97)}}}}
\put(10501,-1561){\makebox(0,0)[b]{\smash{{\SetFigFontNFSS{20}{24.0}{\familydefault}{\mddefault}{\updefault}44:0:55 (97)}}}}
\put(10501,-2161){\makebox(0,0)[b]{\smash{{\SetFigFontNFSS{20}{24.0}{\familydefault}{\mddefault}{\updefault}35:27:37 (95)}}}}
\put(10501,-2761){\makebox(0,0)[b]{\smash{{\SetFigFontNFSS{20}{24.0}{\familydefault}{\mddefault}{\updefault}0:100:0 (100)}}}}
\put(10501,-3361){\makebox(0,0)[b]{\smash{{\SetFigFontNFSS{20}{24.0}{\familydefault}{\mddefault}{\updefault}0:100:0 (100)}}}}
\put(10501,-3961){\makebox(0,0)[b]{\smash{{\SetFigFontNFSS{20}{24.0}{\familydefault}{\mddefault}{\updefault}0:100:0 (100)}}}}
\put(10501,-4561){\makebox(0,0)[b]{\smash{{\SetFigFontNFSS{20}{24.0}{\familydefault}{\mddefault}{\updefault}0:100:0 (100)}}}}
\put(12601,839){\makebox(0,0)[b]{\smash{{\SetFigFontNFSS{20}{24.0}{\familydefault}{\mddefault}{\updefault}100:0:0 (100)}}}}
\put(12601,239){\makebox(0,0)[b]{\smash{{\SetFigFontNFSS{20}{24.0}{\familydefault}{\mddefault}{\updefault}100:0:0 (100)}}}}
\put(12601,-361){\makebox(0,0)[b]{\smash{{\SetFigFontNFSS{20}{24.0}{\familydefault}{\mddefault}{\updefault}100:0:0 (100)}}}}
\put(12601,-961){\makebox(0,0)[b]{\smash{{\SetFigFontNFSS{20}{24.0}{\familydefault}{\mddefault}{\updefault}100:0:0 (100)}}}}
\put(12601,-1561){\makebox(0,0)[b]{\smash{{\SetFigFontNFSS{20}{24.0}{\familydefault}{\mddefault}{\updefault}100:0:0 (100)}}}}
\put(12601,-2161){\makebox(0,0)[b]{\smash{{\SetFigFontNFSS{20}{24.0}{\familydefault}{\mddefault}{\updefault}100:0:0 (100)}}}}
\put(12601,-2761){\makebox(0,0)[b]{\smash{{\SetFigFontNFSS{20}{24.0}{\familydefault}{\mddefault}{\updefault}51:48:0 (97)}}}}
\put(12601,-3361){\makebox(0,0)[b]{\smash{{\SetFigFontNFSS{20}{24.0}{\familydefault}{\mddefault}{\updefault}0:100:0 (100)}}}}
\put(12601,-3961){\makebox(0,0)[b]{\smash{{\SetFigFontNFSS{20}{24.0}{\familydefault}{\mddefault}{\updefault}0:100:0 (100)}}}}
\put(12601,-4561){\makebox(0,0)[b]{\smash{{\SetFigFontNFSS{20}{24.0}{\familydefault}{\mddefault}{\updefault}0:100:0 (100)}}}}
\put(14701,839){\makebox(0,0)[b]{\smash{{\SetFigFontNFSS{20}{24.0}{\familydefault}{\mddefault}{\updefault}100:0:0 (100)}}}}
\put(14701,239){\makebox(0,0)[b]{\smash{{\SetFigFontNFSS{20}{24.0}{\familydefault}{\mddefault}{\updefault}100:0:0 (100)}}}}
\put(14701,-361){\makebox(0,0)[b]{\smash{{\SetFigFontNFSS{20}{24.0}{\familydefault}{\mddefault}{\updefault}100:0:0 (100)}}}}
\put(14701,-961){\makebox(0,0)[b]{\smash{{\SetFigFontNFSS{20}{24.0}{\familydefault}{\mddefault}{\updefault}100:0:0 (100)}}}}
\put(14701,-1561){\makebox(0,0)[b]{\smash{{\SetFigFontNFSS{20}{24.0}{\familydefault}{\mddefault}{\updefault}100:0:0 (100)}}}}
\put(14701,-2161){\makebox(0,0)[b]{\smash{{\SetFigFontNFSS{20}{24.0}{\familydefault}{\mddefault}{\updefault}100:0:0 (100)}}}}
\put(14701,-2761){\makebox(0,0)[b]{\smash{{\SetFigFontNFSS{20}{24.0}{\familydefault}{\mddefault}{\updefault}100:0:0 (100)}}}}
\put(14701,-3361){\makebox(0,0)[b]{\smash{{\SetFigFontNFSS{20}{24.0}{\familydefault}{\mddefault}{\updefault}45:54:0 (97)}}}}
\put(14701,-3961){\makebox(0,0)[b]{\smash{{\SetFigFontNFSS{20}{24.0}{\familydefault}{\mddefault}{\updefault}0:100:0 (100)}}}}
\put(14701,-4561){\makebox(0,0)[b]{\smash{{\SetFigFontNFSS{20}{24.0}{\familydefault}{\mddefault}{\updefault}0:100:0 (100)}}}}
\put(16801,839){\makebox(0,0)[b]{\smash{{\SetFigFontNFSS{20}{24.0}{\familydefault}{\mddefault}{\updefault}100:0:0 (100)}}}}
\put(16801,239){\makebox(0,0)[b]{\smash{{\SetFigFontNFSS{20}{24.0}{\familydefault}{\mddefault}{\updefault}100:0:0 (100)}}}}
\put(16801,-361){\makebox(0,0)[b]{\smash{{\SetFigFontNFSS{20}{24.0}{\familydefault}{\mddefault}{\updefault}100:0:0 (100)}}}}
\put(16801,-961){\makebox(0,0)[b]{\smash{{\SetFigFontNFSS{20}{24.0}{\familydefault}{\mddefault}{\updefault}100:0:0 (100)}}}}
\put(16801,-1561){\makebox(0,0)[b]{\smash{{\SetFigFontNFSS{20}{24.0}{\familydefault}{\mddefault}{\updefault}100:0:0 (100)}}}}
\put(16801,-2161){\makebox(0,0)[b]{\smash{{\SetFigFontNFSS{20}{24.0}{\familydefault}{\mddefault}{\updefault}100:0:0 (100)}}}}
\put(16801,-2761){\makebox(0,0)[b]{\smash{{\SetFigFontNFSS{20}{24.0}{\familydefault}{\mddefault}{\updefault}100:0:0 (100)}}}}
\put(16801,-3361){\makebox(0,0)[b]{\smash{{\SetFigFontNFSS{20}{24.0}{\familydefault}{\mddefault}{\updefault}100:0:0 (100)}}}}
\put(16801,-3961){\makebox(0,0)[b]{\smash{{\SetFigFontNFSS{20}{24.0}{\familydefault}{\mddefault}{\updefault}51:48:0 (96)}}}}
\put(16801,-4561){\makebox(0,0)[b]{\smash{{\SetFigFontNFSS{20}{24.0}{\familydefault}{\mddefault}{\updefault}0:100:0 (100)}}}}
\put(18901,839){\makebox(0,0)[b]{\smash{{\SetFigFontNFSS{20}{24.0}{\familydefault}{\mddefault}{\updefault}100:0:0 (100)}}}}
\put(18901,239){\makebox(0,0)[b]{\smash{{\SetFigFontNFSS{20}{24.0}{\familydefault}{\mddefault}{\updefault}100:0:0 (100)}}}}
\put(18901,-361){\makebox(0,0)[b]{\smash{{\SetFigFontNFSS{20}{24.0}{\familydefault}{\mddefault}{\updefault}100:0:0 (100)}}}}
\put(18901,-961){\makebox(0,0)[b]{\smash{{\SetFigFontNFSS{20}{24.0}{\familydefault}{\mddefault}{\updefault}100:0:0 (100)}}}}
\put(18901,-1561){\makebox(0,0)[b]{\smash{{\SetFigFontNFSS{20}{24.0}{\familydefault}{\mddefault}{\updefault}100:0:0 (100)}}}}
\put(18901,-2161){\makebox(0,0)[b]{\smash{{\SetFigFontNFSS{20}{24.0}{\familydefault}{\mddefault}{\updefault}100:0:0 (100)}}}}
\put(18901,-2761){\makebox(0,0)[b]{\smash{{\SetFigFontNFSS{20}{24.0}{\familydefault}{\mddefault}{\updefault}100:0:0 (100)}}}}
\put(18901,-3361){\makebox(0,0)[b]{\smash{{\SetFigFontNFSS{20}{24.0}{\familydefault}{\mddefault}{\updefault}100:0:0 (100)}}}}
\put(18901,-3961){\makebox(0,0)[b]{\smash{{\SetFigFontNFSS{20}{24.0}{\familydefault}{\mddefault}{\updefault}100:0:0 (100)}}}}
\put(18901,-4561){\makebox(0,0)[b]{\smash{{\SetFigFontNFSS{20}{24.0}{\familydefault}{\mddefault}{\updefault}50:49:0 (68)}}}}
\end{picture}%

%% file: diagrams-three-type.pstex_t
\begin{picture}(0,0)%
\includegraphics{diagrams-three-type.pstex}%
\end{picture}%
\setlength{\unitlength}{3947sp}%
\begingroup\makeatletter\ifx\SetFigFontNFSS\undefined%
\gdef\SetFigFontNFSS#1#2#3#4#5{%
  \reset@font\fontsize{#1}{#2pt}%
  \fontfamily{#3}\fontseries{#4}\fontshape{#5}%
  \selectfont}%
\fi\endgroup%
\begin{picture}(20193,10629)(-4004,-11476)
\put(-2249,-9661){\makebox(0,0)[b]{\smash{{\SetFigFontNFSS{25}{30.0}{\familydefault}{\mddefault}{\updefault}$\theta_3$}}}}
\put(3601,-10561){\makebox(0,0)[b]{\smash{{\SetFigFontNFSS{25}{30.0}{\familydefault}{\mddefault}{\updefault}$\theta_1$}}}}
\put(  1,-11461){\makebox(0,0)[b]{\smash{{\SetFigFontNFSS{25}{30.0}{\familydefault}{\mddefault}{\updefault}coex.}}}}
\put(5851,-9661){\makebox(0,0)[b]{\smash{{\SetFigFontNFSS{25}{30.0}{\familydefault}{\mddefault}{\updefault}1}}}}
\put(-3599,-8236){\makebox(0,0)[b]{\smash{{\SetFigFontNFSS{25}{30.0}{\familydefault}{\mddefault}{\updefault}3}}}}
\put( 76,-5086){\makebox(0,0)[b]{\smash{{\SetFigFontNFSS{25}{30.0}{\familydefault}{\mddefault}{\updefault}2}}}}
\put(-3674,-2761){\makebox(0,0)[b]{\smash{{\SetFigFontNFSS{25}{30.0}{\familydefault}{\mddefault}{\updefault}2:3}}}}
\put(-374,-6211){\makebox(0,0)[b]{\smash{{\SetFigFontNFSS{25}{30.0}{\familydefault}{\mddefault}{\updefault}$\theta_2$}}}}
\put(6001,-4261){\makebox(0,0)[b]{\smash{{\SetFigFontNFSS{25}{30.0}{\familydefault}{\mddefault}{\updefault}1:2}}}}
\put(14401,-4111){\makebox(0,0)[b]{\smash{{\SetFigFontNFSS{25}{30.0}{\familydefault}{\mddefault}{\updefault}$\theta_1$}}}}
\put(7951,-3136){\makebox(0,0)[b]{\smash{{\SetFigFontNFSS{25}{30.0}{\familydefault}{\mddefault}{\updefault}$\theta_2$}}}}
\put(10201,-11461){\makebox(0,0)[b]{\smash{{\SetFigFontNFSS{25}{30.0}{\familydefault}{\mddefault}{\updefault}3}}}}
\put(6751,-2611){\makebox(0,0)[b]{\smash{{\SetFigFontNFSS{25}{30.0}{\familydefault}{\mddefault}{\updefault}2}}}}
\put(16051,-4261){\makebox(0,0)[b]{\smash{{\SetFigFontNFSS{25}{30.0}{\familydefault}{\mddefault}{\updefault}1}}}}
\put(11251,-3661){\makebox(0,0)[b]{\smash{{\SetFigFontNFSS{25}{30.0}{\familydefault}{\mddefault}{\updefault}1:2}}}}
\put(8626,-7111){\makebox(0,0)[b]{\smash{{\SetFigFontNFSS{25}{30.0}{\familydefault}{\mddefault}{\updefault}2:3}}}}
\put(12301,-7711){\makebox(0,0)[b]{\smash{{\SetFigFontNFSS{25}{30.0}{\familydefault}{\mddefault}{\updefault}1:3}}}}
\put(9826,-9961){\makebox(0,0)[b]{\smash{{\SetFigFontNFSS{25}{30.0}{\familydefault}{\mddefault}{\updefault}$\theta_3$}}}}
\put(1501,-2461){\makebox(0,0)[b]{\smash{{\SetFigFontNFSS{25}{30.0}{\familydefault}{\mddefault}{\updefault}$\Delta_1$}}}}
\put(11701,-6061){\makebox(0,0)[b]{\smash{{\SetFigFontNFSS{25}{30.0}{\familydefault}{\mddefault}{\updefault}$\Delta_1$}}}}
\put(10501,-4861){\makebox(0,0)[b]{\smash{{\SetFigFontNFSS{25}{30.0}{\familydefault}{\mddefault}{\updefault}coex.}}}}
\put(12601,-1186){\makebox(0,0)[b]{\smash{{\SetFigFontNFSS{25}{30.0}{\familydefault}{\mddefault}{\updefault}$\Delta_2$}}}}
\put(4501,-6061){\makebox(0,0)[b]{\smash{{\SetFigFontNFSS{25}{30.0}{\familydefault}{\mddefault}{\updefault}$\Delta_2$}}}}
\end{picture}%

%% file: tab-matrix_9.pstex_t
\begin{picture}(0,0)%
\includegraphics{tab-matrix_9.pstex}%
\end{picture}%
\setlength{\unitlength}{3947sp}%
\begingroup\makeatletter\ifx\SetFigFontNFSS\undefined%
\gdef\SetFigFontNFSS#1#2#3#4#5{%
  \reset@font\fontsize{#1}{#2pt}%
  \fontfamily{#3}\fontseries{#4}\fontshape{#5}%
  \selectfont}%
\fi\endgroup%
\begin{picture}(23178,7269)(-3086,-5443)
\put(  1,1439){\makebox(0,0)[b]{\smash{{\SetFigFontNFSS{20}{24.0}{\familydefault}{\mddefault}{\updefault}$\theta_1 = 0.00$}}}}
\put(2101,1439){\makebox(0,0)[b]{\smash{{\SetFigFontNFSS{20}{24.0}{\familydefault}{\mddefault}{\updefault}$\theta_1 = 0.10$}}}}
\put(4201,1439){\makebox(0,0)[b]{\smash{{\SetFigFontNFSS{20}{24.0}{\familydefault}{\mddefault}{\updefault}$\theta_1 = 0.20$}}}}
\put(6301,1439){\makebox(0,0)[b]{\smash{{\SetFigFontNFSS{20}{24.0}{\familydefault}{\mddefault}{\updefault}$\theta_1 = 0.30$}}}}
\put(8401,1439){\makebox(0,0)[b]{\smash{{\SetFigFontNFSS{20}{24.0}{\familydefault}{\mddefault}{\updefault}$\theta_1 = 0.40$}}}}
\put(10501,1439){\makebox(0,0)[b]{\smash{{\SetFigFontNFSS{20}{24.0}{\familydefault}{\mddefault}{\updefault}$\theta_1 = 0.50$}}}}
\put(12601,1439){\makebox(0,0)[b]{\smash{{\SetFigFontNFSS{20}{24.0}{\familydefault}{\mddefault}{\updefault}$\theta_1 = 0.60$}}}}
\put(14701,1439){\makebox(0,0)[b]{\smash{{\SetFigFontNFSS{20}{24.0}{\familydefault}{\mddefault}{\updefault}$\theta_1 = 0.70$}}}}
\put(16801,1439){\makebox(0,0)[b]{\smash{{\SetFigFontNFSS{20}{24.0}{\familydefault}{\mddefault}{\updefault}$\theta_1 = 0.80$}}}}
\put(18901,1439){\makebox(0,0)[b]{\smash{{\SetFigFontNFSS{20}{24.0}{\familydefault}{\mddefault}{\updefault}$\theta_1 = 0.90$}}}}
\put(-2099,839){\makebox(0,0)[b]{\smash{{\SetFigFontNFSS{20}{24.0}{\familydefault}{\mddefault}{\updefault}$\theta_2 = 0.00$}}}}
\put(-2099,239){\makebox(0,0)[b]{\smash{{\SetFigFontNFSS{20}{24.0}{\familydefault}{\mddefault}{\updefault}$\theta_2 = 0.10$}}}}
\put(-2099,-361){\makebox(0,0)[b]{\smash{{\SetFigFontNFSS{20}{24.0}{\familydefault}{\mddefault}{\updefault}$\theta_2 = 0.20$}}}}
\put(-2099,-961){\makebox(0,0)[b]{\smash{{\SetFigFontNFSS{20}{24.0}{\familydefault}{\mddefault}{\updefault}$\theta_2 = 0.30$}}}}
\put(-2099,-1561){\makebox(0,0)[b]{\smash{{\SetFigFontNFSS{20}{24.0}{\familydefault}{\mddefault}{\updefault}$\theta_2 = 0.40$}}}}
\put(-2099,-2161){\makebox(0,0)[b]{\smash{{\SetFigFontNFSS{20}{24.0}{\familydefault}{\mddefault}{\updefault}$\theta_2 = 0.50$}}}}
\put(-2099,-2761){\makebox(0,0)[b]{\smash{{\SetFigFontNFSS{20}{24.0}{\familydefault}{\mddefault}{\updefault}$\theta_2 = 0.60$}}}}
\put(-2099,-3361){\makebox(0,0)[b]{\smash{{\SetFigFontNFSS{20}{24.0}{\familydefault}{\mddefault}{\updefault}$\theta_2 = 0.70$}}}}
\put(-2099,-3961){\makebox(0,0)[b]{\smash{{\SetFigFontNFSS{20}{24.0}{\familydefault}{\mddefault}{\updefault}$\theta_2 = 0.80$}}}}
\put(-2099,-4561){\makebox(0,0)[b]{\smash{{\SetFigFontNFSS{20}{24.0}{\familydefault}{\mddefault}{\updefault}$\theta_2 = 0.90$}}}}
\put(  1,839){\makebox(0,0)[b]{\smash{{\SetFigFontNFSS{20}{24.0}{\familydefault}{\mddefault}{\updefault}36:16:46 (66)}}}}
\put(  1,239){\makebox(0,0)[b]{\smash{{\SetFigFontNFSS{20}{24.0}{\familydefault}{\mddefault}{\updefault}31:19:49 (67)}}}}
\put(  1,-361){\makebox(0,0)[b]{\smash{{\SetFigFontNFSS{20}{24.0}{\familydefault}{\mddefault}{\updefault}31:18:49 (69)}}}}
\put(  1,-961){\makebox(0,0)[b]{\smash{{\SetFigFontNFSS{20}{24.0}{\familydefault}{\mddefault}{\updefault}28:18:52 (70)}}}}
\put(  1,-1561){\makebox(0,0)[b]{\smash{{\SetFigFontNFSS{20}{24.0}{\familydefault}{\mddefault}{\updefault}25:19:54 (71)}}}}
\put(  1,-2161){\makebox(0,0)[b]{\smash{{\SetFigFontNFSS{20}{24.0}{\familydefault}{\mddefault}{\updefault}24:20:55 (72)}}}}
\put(  1,-2761){\makebox(0,0)[b]{\smash{{\SetFigFontNFSS{20}{24.0}{\familydefault}{\mddefault}{\updefault}22:20:57 (73)}}}}
\put(  1,-3361){\makebox(0,0)[b]{\smash{{\SetFigFontNFSS{20}{24.0}{\familydefault}{\mddefault}{\updefault}20:21:57 (74)}}}}
\put(  1,-3961){\makebox(0,0)[b]{\smash{{\SetFigFontNFSS{20}{24.0}{\familydefault}{\mddefault}{\updefault}19:22:58 (75)}}}}
\put(  1,-4561){\makebox(0,0)[b]{\smash{{\SetFigFontNFSS{20}{24.0}{\familydefault}{\mddefault}{\updefault}18:22:59 (75)}}}}
\put(2101,839){\makebox(0,0)[b]{\smash{{\SetFigFontNFSS{20}{24.0}{\familydefault}{\mddefault}{\updefault}40:17:41 (67)}}}}
\put(2101,239){\makebox(0,0)[b]{\smash{{\SetFigFontNFSS{20}{24.0}{\familydefault}{\mddefault}{\updefault}36:19:44 (67)}}}}
\put(2101,-361){\makebox(0,0)[b]{\smash{{\SetFigFontNFSS{20}{24.0}{\familydefault}{\mddefault}{\updefault}33:20:46 (69)}}}}
\put(2101,-961){\makebox(0,0)[b]{\smash{{\SetFigFontNFSS{20}{24.0}{\familydefault}{\mddefault}{\updefault}29:21:48 (70)}}}}
\put(2101,-1561){\makebox(0,0)[b]{\smash{{\SetFigFontNFSS{20}{24.0}{\familydefault}{\mddefault}{\updefault}27:22:49 (71)}}}}
\put(2101,-2161){\makebox(0,0)[b]{\smash{{\SetFigFontNFSS{20}{24.0}{\familydefault}{\mddefault}{\updefault}26:21:52 (72)}}}}
\put(2101,-2761){\makebox(0,0)[b]{\smash{{\SetFigFontNFSS{20}{24.0}{\familydefault}{\mddefault}{\updefault}25:22:51 (73)}}}}
\put(2101,-3361){\makebox(0,0)[b]{\smash{{\SetFigFontNFSS{20}{24.0}{\familydefault}{\mddefault}{\updefault}24:23:52 (73)}}}}
\put(2101,-3961){\makebox(0,0)[b]{\smash{{\SetFigFontNFSS{20}{24.0}{\familydefault}{\mddefault}{\updefault}22:25:52 (73)}}}}
\put(2101,-4561){\makebox(0,0)[b]{\smash{{\SetFigFontNFSS{20}{24.0}{\familydefault}{\mddefault}{\updefault}18:25:56 (75)}}}}
\put(4201,839){\makebox(0,0)[b]{\smash{{\SetFigFontNFSS{20}{24.0}{\familydefault}{\mddefault}{\updefault}43:18:38 (67)}}}}
\put(4201,239){\makebox(0,0)[b]{\smash{{\SetFigFontNFSS{20}{24.0}{\familydefault}{\mddefault}{\updefault}39:19:40 (68)}}}}
\put(4201,-361){\makebox(0,0)[b]{\smash{{\SetFigFontNFSS{20}{24.0}{\familydefault}{\mddefault}{\updefault}36:20:43 (69)}}}}
\put(4201,-961){\makebox(0,0)[b]{\smash{{\SetFigFontNFSS{20}{24.0}{\familydefault}{\mddefault}{\updefault}34:21:44 (70)}}}}
\put(4201,-1561){\makebox(0,0)[b]{\smash{{\SetFigFontNFSS{20}{24.0}{\familydefault}{\mddefault}{\updefault}31:23:44 (71)}}}}
\put(4201,-2161){\makebox(0,0)[b]{\smash{{\SetFigFontNFSS{20}{24.0}{\familydefault}{\mddefault}{\updefault}29:24:46 (71)}}}}
\put(4201,-2761){\makebox(0,0)[b]{\smash{{\SetFigFontNFSS{20}{24.0}{\familydefault}{\mddefault}{\updefault}26:24:49 (72)}}}}
\put(4201,-3361){\makebox(0,0)[b]{\smash{{\SetFigFontNFSS{20}{24.0}{\familydefault}{\mddefault}{\updefault}24:27:47 (73)}}}}
\put(4201,-3961){\makebox(0,0)[b]{\smash{{\SetFigFontNFSS{20}{24.0}{\familydefault}{\mddefault}{\updefault}22:27:49 (74)}}}}
\put(4201,-4561){\makebox(0,0)[b]{\smash{{\SetFigFontNFSS{20}{24.0}{\familydefault}{\mddefault}{\updefault}22:26:51 (75)}}}}
\put(6301,839){\makebox(0,0)[b]{\smash{{\SetFigFontNFSS{20}{24.0}{\familydefault}{\mddefault}{\updefault}47:17:34 (68)}}}}
\put(6301,239){\makebox(0,0)[b]{\smash{{\SetFigFontNFSS{20}{24.0}{\familydefault}{\mddefault}{\updefault}43:20:35 (69)}}}}
\put(6301,-361){\makebox(0,0)[b]{\smash{{\SetFigFontNFSS{20}{24.0}{\familydefault}{\mddefault}{\updefault}39:20:39 (70)}}}}
\put(6301,-961){\makebox(0,0)[b]{\smash{{\SetFigFontNFSS{20}{24.0}{\familydefault}{\mddefault}{\updefault}38:22:39 (71)}}}}
\put(6301,-1561){\makebox(0,0)[b]{\smash{{\SetFigFontNFSS{20}{24.0}{\familydefault}{\mddefault}{\updefault}33:24:42 (72)}}}}
\put(6301,-2161){\makebox(0,0)[b]{\smash{{\SetFigFontNFSS{20}{24.0}{\familydefault}{\mddefault}{\updefault}31:25:43 (72)}}}}
\put(6301,-2761){\makebox(0,0)[b]{\smash{{\SetFigFontNFSS{20}{24.0}{\familydefault}{\mddefault}{\updefault}28:26:45 (72)}}}}
\put(6301,-3361){\makebox(0,0)[b]{\smash{{\SetFigFontNFSS{20}{24.0}{\familydefault}{\mddefault}{\updefault}28:26:45 (74)}}}}
\put(6301,-3961){\makebox(0,0)[b]{\smash{{\SetFigFontNFSS{20}{24.0}{\familydefault}{\mddefault}{\updefault}25:27:46 (74)}}}}
\put(6301,-4561){\makebox(0,0)[b]{\smash{{\SetFigFontNFSS{20}{24.0}{\familydefault}{\mddefault}{\updefault}23:27:49 (75)}}}}
\put(8401,839){\makebox(0,0)[b]{\smash{{\SetFigFontNFSS{20}{24.0}{\familydefault}{\mddefault}{\updefault}50:18:30 (69)}}}}
\put(8401,239){\makebox(0,0)[b]{\smash{{\SetFigFontNFSS{20}{24.0}{\familydefault}{\mddefault}{\updefault}45:20:33 (69)}}}}
\put(8401,-361){\makebox(0,0)[b]{\smash{{\SetFigFontNFSS{20}{24.0}{\familydefault}{\mddefault}{\updefault}42:22:35 (70)}}}}
\put(8401,-961){\makebox(0,0)[b]{\smash{{\SetFigFontNFSS{20}{24.0}{\familydefault}{\mddefault}{\updefault}38:23:37 (71)}}}}
\put(8401,-1561){\makebox(0,0)[b]{\smash{{\SetFigFontNFSS{20}{24.0}{\familydefault}{\mddefault}{\updefault}36:24:38 (72)}}}}
\put(8401,-2161){\makebox(0,0)[b]{\smash{{\SetFigFontNFSS{20}{24.0}{\familydefault}{\mddefault}{\updefault}33:25:41 (73)}}}}
\put(8401,-2761){\makebox(0,0)[b]{\smash{{\SetFigFontNFSS{20}{24.0}{\familydefault}{\mddefault}{\updefault}31:27:40 (73)}}}}
\put(8401,-3361){\makebox(0,0)[b]{\smash{{\SetFigFontNFSS{20}{24.0}{\familydefault}{\mddefault}{\updefault}29:28:41 (73)}}}}
\put(8401,-3961){\makebox(0,0)[b]{\smash{{\SetFigFontNFSS{20}{24.0}{\familydefault}{\mddefault}{\updefault}26:28:44 (74)}}}}
\put(8401,-4561){\makebox(0,0)[b]{\smash{{\SetFigFontNFSS{20}{24.0}{\familydefault}{\mddefault}{\updefault}26:30:43 (74)}}}}
\put(10501,839){\makebox(0,0)[b]{\smash{{\SetFigFontNFSS{20}{24.0}{\familydefault}{\mddefault}{\updefault}53:18:27 (71)}}}}
\put(10501,239){\makebox(0,0)[b]{\smash{{\SetFigFontNFSS{20}{24.0}{\familydefault}{\mddefault}{\updefault}49:20:29 (71)}}}}
\put(10501,-361){\makebox(0,0)[b]{\smash{{\SetFigFontNFSS{20}{24.0}{\familydefault}{\mddefault}{\updefault}45:22:32 (71)}}}}
\put(10501,-961){\makebox(0,0)[b]{\smash{{\SetFigFontNFSS{20}{24.0}{\familydefault}{\mddefault}{\updefault}40:23:36 (72)}}}}
\put(10501,-1561){\makebox(0,0)[b]{\smash{{\SetFigFontNFSS{20}{24.0}{\familydefault}{\mddefault}{\updefault}38:25:36 (72)}}}}
\put(10501,-2161){\makebox(0,0)[b]{\smash{{\SetFigFontNFSS{20}{24.0}{\familydefault}{\mddefault}{\updefault}34:26:38 (73)}}}}
\put(10501,-2761){\makebox(0,0)[b]{\smash{{\SetFigFontNFSS{20}{24.0}{\familydefault}{\mddefault}{\updefault}33:28:38 (74)}}}}
\put(10501,-3361){\makebox(0,0)[b]{\smash{{\SetFigFontNFSS{20}{24.0}{\familydefault}{\mddefault}{\updefault}30:29:39 (74)}}}}
\put(10501,-3961){\makebox(0,0)[b]{\smash{{\SetFigFontNFSS{20}{24.0}{\familydefault}{\mddefault}{\updefault}29:31:39 (75)}}}}
\put(10501,-4561){\makebox(0,0)[b]{\smash{{\SetFigFontNFSS{20}{24.0}{\familydefault}{\mddefault}{\updefault}26:31:42 (75)}}}}
\put(12601,839){\makebox(0,0)[b]{\smash{{\SetFigFontNFSS{20}{24.0}{\familydefault}{\mddefault}{\updefault}55:19:25 (72)}}}}
\put(12601,239){\makebox(0,0)[b]{\smash{{\SetFigFontNFSS{20}{24.0}{\familydefault}{\mddefault}{\updefault}50:20:28 (72)}}}}
\put(12601,-361){\makebox(0,0)[b]{\smash{{\SetFigFontNFSS{20}{24.0}{\familydefault}{\mddefault}{\updefault}46:22:30 (72)}}}}
\put(12601,-961){\makebox(0,0)[b]{\smash{{\SetFigFontNFSS{20}{24.0}{\familydefault}{\mddefault}{\updefault}42:25:32 (73)}}}}
\put(12601,-1561){\makebox(0,0)[b]{\smash{{\SetFigFontNFSS{20}{24.0}{\familydefault}{\mddefault}{\updefault}39:27:33 (73)}}}}
\put(12601,-2161){\makebox(0,0)[b]{\smash{{\SetFigFontNFSS{20}{24.0}{\familydefault}{\mddefault}{\updefault}36:28:35 (74)}}}}
\put(12601,-2761){\makebox(0,0)[b]{\smash{{\SetFigFontNFSS{20}{24.0}{\familydefault}{\mddefault}{\updefault}34:29:36 (74)}}}}
\put(12601,-3361){\makebox(0,0)[b]{\smash{{\SetFigFontNFSS{20}{24.0}{\familydefault}{\mddefault}{\updefault}32:28:38 (74)}}}}
\put(12601,-3961){\makebox(0,0)[b]{\smash{{\SetFigFontNFSS{20}{24.0}{\familydefault}{\mddefault}{\updefault}30:31:37 (75)}}}}
\put(12601,-4561){\makebox(0,0)[b]{\smash{{\SetFigFontNFSS{20}{24.0}{\familydefault}{\mddefault}{\updefault}28:33:38 (76)}}}}
\put(14701,839){\makebox(0,0)[b]{\smash{{\SetFigFontNFSS{20}{24.0}{\familydefault}{\mddefault}{\updefault}57:18:23 (73)}}}}
\put(14701,239){\makebox(0,0)[b]{\smash{{\SetFigFontNFSS{20}{24.0}{\familydefault}{\mddefault}{\updefault}52:21:26 (73)}}}}
\put(14701,-361){\makebox(0,0)[b]{\smash{{\SetFigFontNFSS{20}{24.0}{\familydefault}{\mddefault}{\updefault}49:22:27 (73)}}}}
\put(14701,-961){\makebox(0,0)[b]{\smash{{\SetFigFontNFSS{20}{24.0}{\familydefault}{\mddefault}{\updefault}43:27:29 (73)}}}}
\put(14701,-1561){\makebox(0,0)[b]{\smash{{\SetFigFontNFSS{20}{24.0}{\familydefault}{\mddefault}{\updefault}42:26:31 (74)}}}}
\put(14701,-2161){\makebox(0,0)[b]{\smash{{\SetFigFontNFSS{20}{24.0}{\familydefault}{\mddefault}{\updefault}38:28:32 (74)}}}}
\put(14701,-2761){\makebox(0,0)[b]{\smash{{\SetFigFontNFSS{20}{24.0}{\familydefault}{\mddefault}{\updefault}36:30:33 (74)}}}}
\put(14701,-3361){\makebox(0,0)[b]{\smash{{\SetFigFontNFSS{20}{24.0}{\familydefault}{\mddefault}{\updefault}34:31:34 (75)}}}}
\put(14701,-3961){\makebox(0,0)[b]{\smash{{\SetFigFontNFSS{20}{24.0}{\familydefault}{\mddefault}{\updefault}31:32:35 (76)}}}}
\put(14701,-4561){\makebox(0,0)[b]{\smash{{\SetFigFontNFSS{20}{24.0}{\familydefault}{\mddefault}{\updefault}30:32:36 (76)}}}}
\put(16801,839){\makebox(0,0)[b]{\smash{{\SetFigFontNFSS{20}{24.0}{\familydefault}{\mddefault}{\updefault}59:19:21 (76)}}}}
\put(16801,239){\makebox(0,0)[b]{\smash{{\SetFigFontNFSS{20}{24.0}{\familydefault}{\mddefault}{\updefault}52:21:25 (73)}}}}
\put(16801,-361){\makebox(0,0)[b]{\smash{{\SetFigFontNFSS{20}{24.0}{\familydefault}{\mddefault}{\updefault}49:24:26 (74)}}}}
\put(16801,-961){\makebox(0,0)[b]{\smash{{\SetFigFontNFSS{20}{24.0}{\familydefault}{\mddefault}{\updefault}44:26:28 (74)}}}}
\put(16801,-1561){\makebox(0,0)[b]{\smash{{\SetFigFontNFSS{20}{24.0}{\familydefault}{\mddefault}{\updefault}41:28:29 (74)}}}}
\put(16801,-2161){\makebox(0,0)[b]{\smash{{\SetFigFontNFSS{20}{24.0}{\familydefault}{\mddefault}{\updefault}42:28:29 (75)}}}}
\put(16801,-2761){\makebox(0,0)[b]{\smash{{\SetFigFontNFSS{20}{24.0}{\familydefault}{\mddefault}{\updefault}36:30:32 (75)}}}}
\put(16801,-3361){\makebox(0,0)[b]{\smash{{\SetFigFontNFSS{20}{24.0}{\familydefault}{\mddefault}{\updefault}33:32:33 (75)}}}}
\put(16801,-3961){\makebox(0,0)[b]{\smash{{\SetFigFontNFSS{20}{24.0}{\familydefault}{\mddefault}{\updefault}32:33:33 (76)}}}}
\put(16801,-4561){\makebox(0,0)[b]{\smash{{\SetFigFontNFSS{20}{24.0}{\familydefault}{\mddefault}{\updefault}31:34:34 (76)}}}}
\put(18901,839){\makebox(0,0)[b]{\smash{{\SetFigFontNFSS{20}{24.0}{\familydefault}{\mddefault}{\updefault}58:19:21 (76)}}}}
\put(18901,239){\makebox(0,0)[b]{\smash{{\SetFigFontNFSS{20}{24.0}{\familydefault}{\mddefault}{\updefault}55:22:22 (75)}}}}
\put(18901,-361){\makebox(0,0)[b]{\smash{{\SetFigFontNFSS{20}{24.0}{\familydefault}{\mddefault}{\updefault}52:23:23 (75)}}}}
\put(18901,-961){\makebox(0,0)[b]{\smash{{\SetFigFontNFSS{20}{24.0}{\familydefault}{\mddefault}{\updefault}48:25:25 (75)}}}}
\put(18901,-1561){\makebox(0,0)[b]{\smash{{\SetFigFontNFSS{20}{24.0}{\familydefault}{\mddefault}{\updefault}44:27:28 (75)}}}}
\put(18901,-2161){\makebox(0,0)[b]{\smash{{\SetFigFontNFSS{20}{24.0}{\familydefault}{\mddefault}{\updefault}40:31:27 (75)}}}}
\put(18901,-2761){\makebox(0,0)[b]{\smash{{\SetFigFontNFSS{20}{24.0}{\familydefault}{\mddefault}{\updefault}39:30:29 (76)}}}}
\put(18901,-3361){\makebox(0,0)[b]{\smash{{\SetFigFontNFSS{20}{24.0}{\familydefault}{\mddefault}{\updefault}35:33:30 (76)}}}}
\put(18901,-3961){\makebox(0,0)[b]{\smash{{\SetFigFontNFSS{20}{24.0}{\familydefault}{\mddefault}{\updefault}35:33:31 (76)}}}}
\put(18901,-4561){\makebox(0,0)[b]{\smash{{\SetFigFontNFSS{20}{24.0}{\familydefault}{\mddefault}{\updefault}31:35:32 (77)}}}}
\put(8401,-5311){\makebox(0,0)[b]{\smash{{\SetFigFontNFSS{20}{24.0}{\familydefault}{\mddefault}{\updefault}Densities and clustering coefficient when the interactions are described by the matrix $M_9$ with $\theta_3 = 0.80$}}}}
\end{picture}%